\documentclass[a4paper,english,openright,11pt]{smfartVScomptage}

\usepackage{smfenum,smfthm,amsmath,amssymb,amscd,mathrsfs,euscript,xcolor}
\usepackage{stmaryrd,mathabx,upgreek}
\usepackage[latin1]{inputenc}
\usepackage{graphicx}
\usepackage{mathdots} 
\input{xypic}

\numberwithin{equation}{section} 
\theoremstyle{plain}

\newcounter{nonum}

\def\AA{\mathbb{A}}
\def\BB{{B}}
\def\CC{\mathbb{C}}
\def\DD{B}
 
\def\GG{\mathrm{G}}
\def\HH{\mathrm{H}}
\def\LL{\mathrm{L}}
\def\MM{\mathrm{M}}
\def\NN{\mathrm{N}}
\def\PP{\mathrm{P}}
\def\QQ{\mathbb{Q}}
\def\QQQ{\mathrm{Q}}
\def\RR{\mathbb{R}}
\def\TT{\boldsymbol{\Theta}}

\def\VV{\mathrm{V}}
 
\def\ZZ{\mathbb{Z}} 

\def\A{{ A}}
\def\B{{ B}}
\def\C{{ C}}
\def\D{{ D}}
\def\E{{ E}}
\def\F{{ F}}
\def\G{{ G}}
\def\H{{ H}}

\def\J{{ J}}
\def\K{{ K}}
\def\L{{ L}}
\def\M{{ M}}
\def\N{{ N}}
\def\P{{ P}}

\def\SS{{\rm S}}

\def\U{{ U}}

\def\W{{ W}}

\def\Cc{\EuScript{C}}

\def\Oo{\EuScript{O}}

\def\La{\Lambda}
\def\Om{{\it\Omega}}
\def\De{{\it\Delta}}
\def\Ce{{\it\Gamma}}
\def\Si{\Sigma}
\def\si{\sigma}
\def\ta{\widetilde{\tau}} 
\def\be{\mu} 
\def\la{\lambda} 
\def\el{{\it\La}_{\rho}} 
\def\pp{\textsf{p}}
\def\qq{\textsf{q}}
\def\at{\widetilde{ {\it\La} }}
\def\apt{ {\it\La} }
\def\bt{\widetilde{\li}}
\def\ph{\phi}
\def\sie{\omega}
\def\li{\mu}
\def\ga{\g}

\def\a{\alpha} 
\def\b{\beta}
\def\d{\delta}
\def\e{\varepsilon}

\def\g{\gamma}
\def\h{\varphi}
\def\k{\kappa}
\def\l{\lambda}

\def\n{\eta}

\def\p{\mathfrak{p}}

\def\s{\sigma}
\def\t{\theta}

\def\w{\varpi}

\def\aa{\mathfrak{a}}
\def\bb{\mathfrak{b}}

\def\ii{\iota}
\def\kk{\boldsymbol{k}}
\def\ll{\boldsymbol{l}}

\def\ss{r}
\def\tt{{\sf t}}

\def\>{\geqslant}
\def\<{\leqslant}

\def\Hom{{\rm Hom}}

\def\Mat{\boldsymbol{{\sf M}}}
\def\GL{{\rm GL}}
\def\SL{{\rm SL}}
\def\Sp{{\rm Sp}}
\def\Gal{{\rm Gal}}

\def\Ind{{\rm Ind}}

\def\ip{{\rm Ind}}

\def\St{{\rm St}}
\def\MW{{\rm MW}}

\def\tdt{\times\dots\times}
\def\odo{\otimes\dots\otimes}

\def\BJ{\boldsymbol{\J}}
\def\bk{\boldsymbol{\kappa}}
\def\bs{\boldsymbol{\rho}}
\def\bl{\boldsymbol{\lambda}}
\def\bx{\boldsymbol{\xi}}

\def\Nrd{{\rm Nrd}}
\def\trd{{\rm trd}}
\def\bc{\boldsymbol{\sf b}}
\def\id{{1}}

\def\LJ{\textbf{\textsf{LJ}}}
\def\ad{t}
\def\pl{u}

\definecolor{orange}{RGB}{255,127,0}
\definecolor{dgreen}{RGB}{0,192,0}

\newcommand{\CA}{\EuScript{A}}

\newcommand{\CF}{k}

\def\frS{\mathfrak{S}}

\def\sl{\mathfrak{sl}}
\def\diag{\mathrm{diag}}

\def\idots{\reflectbox{$\ddots$}}
\def\rd{\,\mathrm{d}}

\def\JL{\mathrm{JL}}

\def\bsl{\backslash}

\renewcommand{\Re}{\mathrm{Re}}

\long\def\MSC#1\EndMSC{\def\arg{#1}\ifx\arg\empty\relax\else
     {\par\narrower\noindent%
     2010 Mathematics Subject Classification: #1\par}\fi}

\long\def\KEY#1\EndKEY{\def\arg{#1}\ifx\arg\empty\relax\else
	{\par\narrower\noindent Keywords and Phrases: #1\par}\fi}

\title{Discrete series representations of quaternionic $\GL_n(D)$
with symplectic periods}

\author{Nadir Matringe}
\address{Institute of Mathematical Sciences,
  NYU Shanghai,
  3663 Zhongshan Road
  North Shanghai,
  200062,
  China
  and
  Institut de Mathématiques de Jussieu-Paris Rive Gauche,
  Université Paris Cité 75205,
  Paris,
  France}
\email{nrm6864@nyu.edu, matringe@img-prg.fr}

\author{Vincent S\'echerre}
\address{
Laboratoire de Math\'emati\-ques de Versailles, 
UVSQ, 
CNRS, 
Universit\'e Paris-Saclay,
78035 Versailles, France}
\email{vincent.secherre@uvsq.fr}

\author{Shaun Stevens}
\address{
  School of Engineering,
  Mathematics and Physics,
  University of East Anglia,
  Norwich NR4 7TJ,
  United Kingdom}
\email{shaun.stevens@uea.ac.uk}

\author{Miyu Suzuki}
\address{
  Department of Mathematics,
  Faculty of Science,
  Kyoto University,
  Kitashirakawa Oiwake-cho,
  Sakyo-ku,
  Kyoto 606-8502,
  Japan}
\email{suzuki.miyu.4c@kyoto-u.ac.jp}

\begin{abstract}
For a non-Archimedean locally compact field $F$ of odd residue characteristic 
and~cha\-racteristic $0$,
we prove a conjecture of D.~Prasad predicting that,
for an integer $n \> 1$ and a non-split quaternionic $F$-algebra $D$,
a discrete series representation of $\GL_n(D)$ has a symplectic period
if and only if it is cuspidal and 
its Jacquet--Langlands~trans\-fer to $\GL_{2n}(F)$ is non-cuspidal.
\end{abstract} 

\begin{document} 

\maketitle

\MSC 22E50, 11F70
\EndMSC
\KEY 
Cuspidal representations, 
Discrete series,
Jacquet--Langlands correspondence,
Quaternionic forms,
Symplectic periods
\EndKEY

\thispagestyle{empty}

\section{Introduction}

\subsection{}
\label{mickaelmartinez}

Let $F$ be a non-Archimedean locally compact field,
and~$D$ be a non-split quaternion algebra of centre $F$.
Fix an integer $n \> 1$, and set $G = \GL_n(D)$.
This is an inner form~of~$\GL_{2n}(F)$,~which
can be equipped with an involution $\s$ whose fixed
point subgroup $H$ is equal to $\Sp_n(D)$,
the~non-quasi-split inner form of the symplectic~group~$\Sp_{2n}(F)$. 
{In the framework of the local relative Langlands program, one is
interested in the classi\-fi\-cation of the 
irreducible (smooth, complex)~re\-pr\-esentations of $G$ 
which are distinguished by~$H$,
that is, 
which admit non-zero $H$-invariant linear forms.
The split version where the pair $(G,H)$ is replaced by
$(G',H')=(\GL_{2n}(F),\Sp_{2n}(F))$~has
been thoroughly investigated by Jacquet--Rallis \cite{JR},
Heumos-Rallis \cite{Heumos-Rallis}
and Offen \cite{Offen-Sp1,Offen-Sp2}, both locally and globally.
The pair $(G',H')$ is a vanishing pair  in the sense that there is~no cuspidal
represen\-tation of $G'$ distinguished by $H'$,
both locally and globally.
The striking difference~with the pair 
$(G,H)$ under consideration in this work
is that the latter does not share this property,
as~obser\-ved~by Verma in \cite{Verma}.}
More precisely, 
for~dis\-crete series representations~of~$G$,
{Dipendra} Prasad~pro\-posed the following conjecture
(see also \cite{Verma}~Con\-jec\-ture~7.1)
stated in~terms~of~the local Jacquet--Lang\-lands correspondence,
a bijection between the discrete series~of $G$ and $\text{GL}_{2n}(F)$.
For the~defini\-tion of the notation $\St_2$, see \S\ref{DSclassiflocal} below.

\begin{conj}
\label{CONJPRASAD}
\begin{enumerate}
\item
There is a discrete series representation of $\GL_n(D)$ distinguished by 
$\Sp_n(D)$ if and only if $n$ is odd. 
\item
Suppose that the integer $n$ is odd.
The discrete series representations of $\GL_n(D)$ 
which~are distingui\-shed by $\Sp_n(D)$ are exactly the cuspidal 
representations of $\GL_n(D)$ whose Jacquet--Lang\-lands transfer to 
$\GL_{2n}(F)$ is of the form $\St_2(\tau)$ for some 
cus\-pi\-dal~repre\-sen\-tation $\tau$ of $\GL_{n}(F)$. 
\end{enumerate}
\end{conj}

Note that,
if one replaces the groups $G$ and $H$ by their split forms $\GL_{2n}(F)$ 
and $\Sp_{2n}(F)$,
then \cite{OffenSayagJFA08}~Theo\-rem~1 implies that
there is no generic (in particular, no discrete series)
representation~of
$\GL_{2n}(F)$ distinguished by $\Sp_{2n}(F)$,
whatever the parity of $n$.

\subsection{}
\label{saintpaul}

In this article, we prove the following results.
Let $p$ be the residue characteristic of $F$.

\begin{theo}
\label{MAINTHM1}
Suppose that $F$ has odd residue characteristic
and that the Jacquet--Langlands transfer of any cuspidal representation
of $\GL_n(D)$ distinguished by $\Sp_n(D)$ is~non-cuspidal.~Then any
cuspidal representation of $\GL_n(D)$ whose~Jac\-quet--Langlands
transfer is non-cuspidal is~dis\-tinguished by $\Sp_n(D)$.
\end{theo}

It then follows from well-known properties of the Jacquet--Langlands 
correspondence~(see~\S\ref{DSclassiflocal} and \S\ref{JLcor})
that
\begin{itemize}
\item 
there is a cuspidal representation of $\GL_n(D)$ 
distinguished by $\Sp_n(D)$ if and only if~$n$~is~odd, 
\item
if $n$ is odd, 
a cuspidal representation $\pi$ of $\GL_n(D)$ 
is distingui\-shed by $\Sp_n(D)$~if and only if its Jacquet--Lang\-lands
transfer to $\GL_{2n}(F)$ is a discrete series representation of the form
$\St_2(\tau)$,
where $\tau$ is a cus\-pi\-dal~repre\-sen\-tation of $\GL_{n}(F)$
uniquely determined by $\pi$ up to isomorphism.
\end{itemize}

In the case when $F$ has characteristic $0$,
Ver\-ma~(\cite{Verma} Theorem 1.2)
proved,
by using a~globali\-sation argument,~that any cuspidal representation
of $G$ which is distinguished by $H$
has a~non-cus\-pidal~Jacquet--Lang\-lands transfer to $\GL_{2n}(F)$
(see also Theorem \ref{Vermacarqcq} below).
It follows that,
when $F$ has characteristic~$0$ and $p\neq2$,
any cuspidal representation of $\GL_n(D)$ whose~Jac\-quet--Langlands
transfer is non-cuspidal is~dis\-tinguished by $\Sp_n(D)$.

\begin{theo}
\label{MAINTHM2}
Suppose that $F$ has characteristic~$0$.
Then any discrete series~re\-presentation of $\GL_n(D)$ distinguished by 
$\Sp_n(D)$ is cuspidal.
\end{theo}

Putting Theorems \ref{MAINTHM1} and \ref{MAINTHM2}
together with Verma's result, 
we obtain~the following corollary.

\begin{coro}
\label{MAINTHM3}
Suppose that $F$ is a non-Archimedean locally compact field of
characteristic~$0$ and odd residue characteristic.
Then Prasad's Conjecture \ref{CONJPRASAD} holds.
\end{coro}

The proofs of Theorems \ref{MAINTHM1} and \ref{MAINTHM2} use 
quite different tools and methods. 
Let us first explain how we prove Theorem \ref{MAINTHM1}.

\subsection{}
\label{baltimore}

The strategy of the proof of Theorem \ref{MAINTHM1} is,
given a cuspidal representation $\pi$ of $G$ whose Jacquet--Langlands
transfer is non-cuspidal,
to produce a pair~$(\BJ,\bl)$
made of a compact mod centre, open subgroup $\BJ$ of $G$
and an irreducible~re\-pre\-sen\-tation $\bl$ of~$\BJ$ such that:
\begin{itemize}
\item 
$\bl$ is distinguished by $\BJ \cap H$, 
\item
the compact induction of $\bl$ to $\G$ is isomorphic to $\pi$.
\end{itemize}
By a simple application of Mackey's formula,
this will imply that $\pi$ is distinguished by $\H$.~The
construction of a suitable pair $(\BJ,\bl)$
is based on Bushnell--Kutzko's theory of types,
as we~ex\-plain below.

\subsection{}
\label{lapinblanc}

Start with a cuspidal irreducible representation~$\pi$ of $G$.
By \cite{BK,SS}, it is compactly~indu\-ced from a Bushnell--Kutzko type:
this is a pair $(\BJ, \bl)$ with the following properties: 
\begin{itemize}
\item
the group $\BJ$ is open and compact mod centre,
it has a unique maximal compact subgroup $\BJ^0$
and a unique maximal normal pro-$p$-subgroup $\BJ^1$,
\item
the representation $\bl$ of $\BJ$ is irreducible and factors
(non-canonically) as $\bk \otimes \bs$, 
where $\bk$ is a representation of $\BJ$ whose restriction to $\BJ^1$ is 
irreducible and $\bs$ is an irreducible representation of $\BJ$ whose 
restriction to $\BJ^1$ is trivial,
\item
the quotient $\BJ^0/\BJ^1$ is isomorphic to $\text{GL}_m(\ll)$ for some 
integer $m$ dividing $n$ and some finite extension $\ll$ of the residue field
of $F$,
and the restriction of $\bs$ to $\BJ^0$ is the inflation of a cuspidal 
representation $\varrho$ of $\text{GL}_m(\ll)$. 
\end{itemize}

Our first task is to prove that,
if the Jacquet--Langlands transfer of $\pi$ to $\GL_{2n}(F)$ is non-cus\-pi\-dal,
then, 
among all possible Bushnell--Kutzko types $(\BJ,\bl)$ whose compact induction 
to $G$ is isomorphic to $\pi$ (they form a single $G$-conjugacy class),
there is one such that $\BJ$ is stable by~$\s$ and
$\bk$ can be chosen to be distinguished by $\BJ \cap H$.

\subsection{}
\label{par15i}

Assuming this has been done, our argument is as follows:
\begin{enumerate}
\item
The fact that $\bk$ is distinguished by $\BJ \cap H$
together with the decomposition
\begin{equation*}
\Hom_{\BJ \cap H}(\bk \otimes \bs, \CC) \simeq
\Hom_{\BJ \cap H}(\bk,\CC) \otimes \Hom_{\BJ \cap H}(\bs, \CC)
\end{equation*}
implies that $\bk \otimes \bs$ is distinguished by $\BJ \cap H$
if and only if $\bs$ is distinguished by $\BJ \cap H$.
\item
The representation $\bs$ is distinguished by $\BJ\cap\H$
if and only if the cuspidal representation $\varrho$~of $\GL_m(\ll)$ is
distinguished by a unitary group, 
or equivalently,
$\varrho$ is invariant by the non-tri\-vial~automorphism of $\ll/\ll_0$,
where $\ll_0$ is a subfield of $\ll$ over which $\ll$ is quadratic.
\item
The fact that the Jacquet--Langlands transfer of $\pi$ is non-cuspidal 
implies that $\varrho$ is invariant by $\Gal(\ll/\ll_0)$. 
\end{enumerate}

Note that (2) is reminiscent of \cite{Verma} Section 5.
See Section \ref{Sec6} below for more details.

\subsection{}
\label{farouche}

It remains to prove that $\BJ$ and $\bk$ can be chosen as in \S\ref{lapinblanc}.
The construction~of~$\bk$ relies~on the notion of simple character,
which is the core of Bushnell--Kutzko's type theory.~The
cuspidal representation $\pi$ of \S\ref{lapinblanc}
contains a simple character,
and the set of simple characters contai\-ned in~$\pi$ form a single 
$G$-conjugacy class.
We first prove that, 
if the Jacquet--Langlands transfer of $\pi$ to $\GL_{2n}(F)$ is non-cus\-pi\-dal,
then,
among all simple characters~contai\-ned~in $\pi$,
there is a simple character $\t$ such that $\t\circ\s = \t^{-1}$.
Zou \cite{ZouO} proved a similar~result for~cus\-pidal representations
of~$\GL_n(F)$ with respect to an orthogonal involution, 
and we explain how to transfer it to $G$~in~an appropriate manner. 
(Note that,
if the Jacquet--Langlands transfer of $\pi$ is cus\-pi\-dal,
such a $\t$~may not exist.)

Next,
fix a simple character $\t$ as above, 
and let $\BJ$ denote its normalizer in $G$,
which is stable by $\s$. 
A standard construction (see for instance \cite{VSautodual,NRV}) provides us with:
\begin{itemize}
\item 
a representation $\bk$ of $\BJ$ such that the
contragredient of $\bk\circ\s$ is isomorphic to $\bk$, 
\item
a quadratic character $\chi$ of $\BJ\cap\H$ 
such that the vector space $\Hom_{\BJ\cap\H}(\bk,\chi)$ is non-zero. 
\end{itemize}
To prove that the character $\chi$ is trivial,
we show that, 
if $\chi$ were non-trivial,
one could construct an~$H$-distinguished cuspidal representation of $G$
with cuspidal transfer to $\GL_{2n}(F)$, 
thus contra\-dic\-ting the assumption of Theorem~\ref{MAINTHM1}.

Together with the argument of \S\ref{par15i},
this finishes the proof of Theorem~\ref{MAINTHM1}.

\subsection{}
\label{par18now}

Now let us go back to \S\ref{saintpaul},
assuming that $n$ is odd.
{Associated with any cuspidal represen\-tation $\pi$ of $\GL_n(D)$
whose Jacquet--Lang\-lands transfer to $\GL_{2n}(F)$ is non-cuspidal,
there exists~a unique cuspidal representation $\tau$ of $\GL_n(F)$ such that
the transfer of $\pi$ is equal to $\St_2(\tau)$.
This defines a map}
\begin{equation*}
\pi\mapsto\tau
\end{equation*}
from 
cuspidal representations of $\GL_n(D)$ whose Jacquet--Lang\-lands
transfer to $\GL_{2n}(F)$ is non-cuspidal to 
cuspidal representations of $\GL_n(F)$,
and this map is a bijection (see Remark \ref{saccard}).~As
suggested by Prasad,
the inverse of this map can be thought of as a
`non-abelian'~base change,~de\-noted $\bc_{D/F}$,
from cuspidal representations of $\GL_n(F)$
to those~of $\GL_n(D)$. 
For instan\-ce, if~$n=1$,
the map $\bc_{D/F}$ is just $\chi \mapsto \chi\circ\Nrd$,
where $\Nrd$ is~the~redu\-ced~norm from~$D^\times$
to~$\F^\times$~and $\chi$ ranges over the set of all characters of 
$\F^\times$.
{When $F$ has characteristic~$0$ and~odd residue characteristic,
it follows from Corollary \ref{MAINTHM3} that}
the im\-age of the map $\bc_{D/F}$ is made of those cuspidal
represen\-ta\-tions
of $\GL_n(D)$ which are dis\-tinguished by $\Sp_n(D)$.

A type theoretic, ex\-pli\-cit description of $\bc_{D/F}$ 
can be extracted from \cite{BHJL3,SSCMAT19,Dotto},
at least~up~to inertia, that is: 
given a cuspidal representation $\tau$ of $\GL_n(F)$,
described as the compact induction of a~Bushnell--Kutzko type,
one has an explicit~des\-cription of the type of the inertial
class~of the cuspidal representation
$\bc_{D/F}(\tau)$ in terms of the type of $\tau$.

The case of cuspidal representations of depth $0$ has been considered
in \cite{Verma} Section 5.
The~ex\-pli\-cit description of $\bc_{D/F}$
provided by \cite{Verma} Proposition 5.1, Remark 5.2 is somewhat
incomplete
(see Remarks~\ref{buffetfroid1} and \ref{buffetfroid2} below). 
In Section \ref{thecircle},
thanks to \cite{SZlevel01,SZlevel02,BHlevel0},~we
provide a full~des\-cription of $\bc_{D/F}$~for~cus\-pidal repre\-sentations of 
depth $0$. 

\subsection{}

We now explain how we prove Theorem \ref{MAINTHM2}.
The proof is based on
an idea from the~first author's~pre\-vious work \cite{MatJFA}: 
we study the functional equation of local intertwining periods~in~or\-der to 
reduce the study of distinction of discrete series representations to the 
cuspidal case.~We note that this idea has been successfully applied in
\cite{SX} as well.
However in the case at hand,~the argument here is different from,
and in fact more involved than,
the one used in \cite{MatJFA}~and~\cite{SX}, sin\-ce
the result to prove is of different flavour, and we cannot avoid using
intertwining periods attached to orbits which are neither closed, nor open,
hence using the results of \cite{MOY}.

Let $\pi$ be an irreducible discrete series representation of $G$.
Associated with it ({via} the~classifi\-ca\-tion of the discrete
series of $G$ which we recall in \S\ref{DSclassiflocal}),
there are a divisor $m$ of $n$,
a cuspidal ir\-reducible representation $\rho$ of $\GL_{n/m}(D)$ and
an integer $r\in\{1,2\}$,
such that,
if we set
\begin{equation}
\label{defIs}
  I(s,\rho) = \rho\nu^{sr(m-1)} \times \rho\nu^{sr(m-3)}
  \times \cdots \times \rho\nu^{sr(1-m)}
\end{equation}
for any $s\in\CC$,
where $\nu$ denotes the character ``normalized absolute value of the
reduced norm'' of $\GL_{n/m}(D)$
(see Sec\-tion \ref{apfelstrudel} for the notation),
then $\pi$ is the unique irreducible quotient of $I(-1, \rho)$. 
Note that $\pi$ is cuspidal if and only if $m=1$.
We set $\ad=n/m$.

\subsection{}

The proof of Theorem \ref{MAINTHM2} is by contradiction,
assuming that $\pi$ is distinguished by $H$ and $m\>2$.
First,
by using Offen's geometric lemma \cite{OJNT}
and Verma \cite{Verma} Theorem 1.2,
we prove that the distinction of $\pi$ 
implies that $\rho$ is distinguished by $\Sp_{\ad}(D)$ and $r=2$
(Proposition \ref{PhilippedOrleans} and Corollary \ref{Anker}).

Let $M$ be the standard Levi subgroup 
$\GL_{\ad}(D) \times\cdots\times \GL_{\ad}(D)$
and $P$ be the standard~parabo\-lic~subgroup 
generated by $M$ and upper triangular matrices of $G$.
Fix an $\n\in\G$ such that $P\n\H$ is~open in $G$ and
$M \cap \eta H\eta^{-1}$ is equal to
$\Sp_{\ad}(D) \times\cdots\times \Sp_{\ad}(D)$.
Since $\rho$ is distinguished, 
there~is a non-zero $M \cap \eta H\eta^{-1}$-invariant
linear form $\be$ on the inducing~repre\-sentation 
of \eqref{defIs}.
For any flat section $\varphi_s\in I(s, \rho)$~(see \S\ref{DuchesseduMaine}
for a definition),
the integral
\begin{equation*}
J(s, \varphi_s, \be) = \int_{(\eta^{-1}P\eta \cap H)\bsl H} \be(\varphi_s(\eta h))\,\rd h
\end{equation*}
converges for $\Re(s)$ in a certain right half plane,
has meromorphic continuation to $\CC$
and defines an $H$-invariant~li\-near form
$J(s,\cdot,\be)$ on $I(s, \rho)$,
called the \emph{open intertwining period}.

Let $M(s,w)$ denote the standard intertwining operator
from $ I(s, \rho)$ to $I(-s, \rho)$
associated~to the long\-est
element $w$
of the symmetric group $\frS_m$.
As $I(s, \rho)$ is irreducible for generic $s$ and the spa\-ce 
$\Hom_{H}(I(s, \rho), \CC)$ has dimension at most $1$
for~such~$s$ by Verma \cite{Verma} Theorem~1.1,~there~is~a
meromorphic function $\alpha(s,\rho)$ satisfying
\begin{equation*}
\label{gingerale} 
  J(-s, M(s, w)\h_s, \be) = \a(s, \rho)J(s, \h_s, \be)
\end{equation*}
for any flat section $\h_s\in I(s, \rho)$.
The next assertion is the key to the proof of Theorem \ref{MAINTHM2}.
    
\begin{prop}[Corollary \ref{cor:gamma}]
\label{corgammaintro} 
The meromorphic function $\alpha(s, \rho)$ is holomorphic and non-zero at 
$s=1$.
\end{prop}

Let $M^*\in \Hom_{G}(I(-1, \rho), I(1,\rho))$
be a non-zero~in\-tertwining operator.
Since $m\>2$,
its~image is~isomorphic to $\pi$ and $M^* M(1, w)$ is zero.
Since $\Hom_{H}(I(-1, \rho), \CC)$ has dimension $1$
by Proposi\-tion \ref{PhilippedOrleans},
there is a ${\it\La}\in\Hom_{H}(\pi, \CC)$ such that
$J(-1, \h, \be)={\it\La}M^*\h$ for all $\varphi\in I(-1, \rho)$.
Then~we see that 
\[
  \alpha(1, \rho) J(1, \h , \be) = J(-1, M(1, w)\h , \be)
  = {\it\La}M^* M(1, w)\h = 0
\] 
for $\varphi\in I(1,  \rho)$.
Combined with Proposition \ref{corgammaintro},
this implies that $J(1, \varphi, \be)=0$ for all $\varphi\in I(1, \rho)$.
This contradicts Proposition \ref{Felzl},
which asserts that $J(1, \cdot, \be)$ is a non-zero linear form. 

\subsection{}

For the proof of Proposition \ref{corgammaintro},
we compute $\alpha(s, \rho)$ via a global method. 
First we~globa\-lise $\rho$.
There are
\begin{itemize}
\item 
a totally imaginary number field $k$, with ring of adèles $\AA$, 
{such that there is a unique place $\pl$ above $p$ and $k_{\pl}$,
the completion of $k$ at $\pl$,
is equal to $F$,}
\item 
a quaternion algebra $\DD$ over $k$ such that $\DD_{\pl}=\DD\otimes k_{\pl}$ 
is equal to $D$, 
\item
a cuspidal automorphic representation $\Pi$ of
$\GL_{\ad}(\DD\otimes_{k}\AA)$ 
with a non-zero $\Sp_{\ad}(\DD\otimes_{k}\AA)$-pe\-riod~and
whose local component at $\pl$ is isomorphic to $\rho$. 
\end{itemize}
Given $s\in\CC$,
let $\h$ be a non-zero automorphic form on
$\GL_n(\DD\otimes_k\AA)$ in the parabolically induced representation
$\Pi \tdt \Pi$
which decomposes into a product of local factors $\h_v$.
Similarly to the local setting,
we can define the global intertwining period by the meromorphic continuation 
of
\begin{equation*}
  J(s, \h) = \int_{(\eta^{-1}\PP\eta\cap \HH)(\AA)\bsl \HH(\AA)}
  \left( \int_{(\MM\cap\n\HH\n^{-1})(k)\bsl (\MM\cap\n\HH\n^{-1})(\AA)}
    \h (m\eta h)\,\rd m \right) e^{\langle s\rho_P, H_P(\eta
    h)\rangle}\,\rd h.
\end{equation*}
(For unexplained notations, see the later sections.)
It has the product decomposition
\begin{equation*}
J(s, \h) = \prod_v J_v(s, \ph_v, \be_v)
\end{equation*}
and, for each $v$, we have the functional equation
\begin{equation*}
J_v(-s, M_v(s, w)\h_v, \be_v) = \a_v(s)J_v(s, \h_v, \be_v)
\end{equation*}
where $\a_v(s)$ is a meromorphic function such that $\a_{\pl}(s)=\a(s, \rho)$. 
In \S\ref{sec:intertwining-split}, we show~that~$\a_v(s)$ can be written in terms
of Rankin--Selberg $\g$-factors {for all $v$ at which $\DD_v$ is split.} 
We also~obtain a formula for $J_v(s, \ph_v, \be_v)$ at almost all places. 
Using the functional equation of global intertwin\-ing~periods of 
Section \ref{sec:global}, we obtain an equality of the form
\begin{equation}
\label{expresso}
  \prod_{v\in S} \a_v(s)
  = \prod_{v\in S}\prod_{1\leq i<j\leq m} 
 {\g(2(j-i)s+2, \Si_{v}^{\phantom{\vee}}, \Si_{v}^\vee,\psi_v^{\phantom{\vee}})}
  {\g(2(j-i)s, \Si_{v}^{\phantom{\vee}}, \Si_{v}^\vee,\psi_v^{\phantom{\vee}})}^{-1}
  \end{equation}
where $S$ is a finite set of finite places of $k$ such that $\DD_v$ splits
for all $v\not\in S$ and $\Si$ is a certain~cus\-pidal~automorphic 
representation of $\GL_{\ad}(\AA)$ associated with $\Pi$ via the global
Jacquet--Lang\-lands correspondence. 
We deduce from \eqref{expresso}
{and the fact that $\pl$ is the only place of $\CF$ above $p$
that 
\[
\a(s, \rho) = c\cdot 
  \prod_{1\leq i<j\leq m} {\g(2(j-i)s+2, \Si_{\pl}^{\phantom{\vee}},
    \Si_{\pl}^\vee,\psi_{\pl}^{\phantom{\vee}})}
  {\g(2(j-i)s, \Si_{\pl}^{\phantom{\vee}}, 
    \Si_{\pl}^\vee,\psi_{\pl}^{\phantom{\vee}})}^{-1} 
\]}
{for some constant $c\in \CC^\times$.} 
Hence we see that $\alpha(1, \rho)\neq0$. 

\subsection{}

Let us comment on the assumption on the residue characteristic in Theorem
\ref{MAINTHM1}.
The~only places where we~use~this~as\-sumption are:
Proposition \ref{camille}, Lemma \ref{kappasigmaselfdual} and
Proposition \ref{chitrivial}.~Pro\-po\-si\-tion \ref{camille} is the main~dif\-ficulty:
this proposition might not hold when $p=2$.

\subsection{}

Let us comment on the assumption on the characteristic of $F$
in Theorem \ref{MAINTHM2}.
Our~proof uses the theory of local and global intertwining periods,
which are only available~in characteristic $0$ so far,
although at least locally when $p\neq 2$ this restriction is probably
removable just by~check\-ing the original sources.
We also use this assumption when we apply 
Theorem \ref{Vermacarqcq}
in the proof of Corollary \ref{Anker}
(see also Remark \ref{VermacarqcqZ}).

\section*{Acknowledgements}

We thank D.~Prasad for giving his talk at the conference 
\textit{Representations of $p$-adic Groups~and the Langlands Correspondence, in
honour of Colin Bushnell},
September 2024,~which brought our attention to his
Conjecture~\ref{CONJPRASAD}, and in particular for suggesting that we~res\-trict
to $\SL_n(D)$ in~the proof of Theorem \ref{Vermacarqcq}. 

Nadir Matringe thanks Kyoto University for its hospitality,
as parts of this work were elaborated~while~he was invited there.
Vincent Sécherre was supported by the Institut Universitaire~de France.
Shaun Stevens was suppor\-ted by the Engineering and Physical Sciences 
Research Coun\-cil (grant~EP/V061739/1).
Miyu Suzuki was supported by JSPS KAKENHI (JP22K13891 and JP23K20785).

\section{Quaternion algebras and symplectic groups}

\subsection{}
\label{Kaboul}

Let $F$ be a field.
Given any finite-dimensional central division $F$-algebra $\De$
and any~inte\-ger~$n\>1$, 
we write $\Mat_n(\De)$ for the central simple $F$-algebra made of
all $n\times n$ matrices~with~en\-tries~in $\De$
and $\GL_n(\De)$ for the group of its invertible elements.
The latter is the group~of~$F$-ra\-tio\-nal points of a connected
reductive algebraic group defined over $F$.

\subsection{}
\label{Seville}

Fix an integer $n\>1$,
and set $A=\Mat_n(\De)$ and $G=\GL_n(\De)$.
We write~$\Nrd_{A/F}$ and $\trd_{A/F}$ for the reduced norm and trace
of~$A$ over $F$, respectively.

Let $(n_1, n_2, \ldots, n_r)$ be a composition of $n$,
that is,
a family of positive integers whose sum is equal to $n$.
Associated with it,
there are the standard Levi subgroup
\begin{equation*}
\M = \GL_{n_1}(\De)\tdt\GL_{n_r}(\De)
\end{equation*}
considered as a subgroup of block diagonal matrices of $G$,
and the standard parabolic subgroup $\P$ of $G$ generated by $\M$ and all 
upper triangular matrices.

Denoting by $\N$ the unipotent radical of $\P$,
we have the standard Levi decomposition $\P=\M\N$.

\subsection{}
\label{KyberGate}
  
Let $D$ be a quaternion~algebra over $F$, 
that is,
a central simple $F$-algebra~of dimension~$4$.
The algebra $D$ is either isomorphic~to $\Mat_2(F)$
-- in~which case we say that it is split --
or a division algebra.
In both cases,
it is equipped~with the canonical anti-in\-vo\-lu\-tion
\begin{equation}
\label{antiinvD}
x \mapsto \overline{x}=\trd_{D/F}(x)-x.
\end{equation}
One has the identity $x\overline{x}=\overline{x}x=\Nrd_{D/F}(x)$
for any $x\in D$.
Note that an element of~$D$ is~in\-verti\-ble~if and only if its reduced norm
is non-zero.

Given an $a\in A = \Mat_n(D)$,
for an $n\>1$, 
we write ${}^{\tt}a$ for the transpose of $a$ with respect to the
antidiagonal
and $\overline{a}$ for the matrix obtained by applying \eqref{antiinvD}
to each entry of $a$.
We define an anti-involution 
\begin{equation}
\label{antiinvA}
a \mapsto a^* = {}^{\tt}\overline{a}
\end{equation}
on the $F$-algebra $A$.
The group $G=\GL_n(D)$,
made of invertible elements of $A$,
is then equipped with the involution $\s : x \mapsto (x^*)^{-1}$. 
The subgroup $\G^\s$ made of all elements of $G$ that are fixed~by $\s$
is denoted by $\Sp_n(D)$.

When $D$ is split,
any isomorphism from $D$ to $\Mat_2(F)$
transports the canonical anti-involution of $D$ to
that of $\Mat_2(F)$, 
and induces an algebra isomorphism $A\simeq\Mat_{2n}(F)$
transporting \eqref{antiinvA} to the anti-invo\-lu\-tion
$x \mapsto \Om \cdot {}^{\intercal} x \cdot \Om^{-1}$
where
\begin{equation}
\label{defOm}
\Om = \Om_{2n} =
\begin{pmatrix}
1 &&&& \\
&-1&&& \\
&&\ddots&& \\
&&&1&\\
&&&&-1
\end{pmatrix}
\in \GL_{2n}(F)
\end{equation}
and ${}^{\intercal} x$ is the transpose of $x\in \Mat_{2n}(F)$~with
respect to the antidiagonal.
It thus induces a group isomorphism $G\simeq\GL_{2n}(F)$
which sends~the subgroup $\Sp_n(D)$ to the symplectic group $\Sp_{2n}(F)$,
the latter being defined as the sub\-group of $\GL_{2n}(F)$ made of 
those matrices $x\in \GL_{2n}(F)$ such that ${}^{\intercal} x \Om x=\Om$.

When $D$ is non-split,
the group $\GL_n(D)$ is an inner form of $\GL_{2n}(F)$
and $\Sp_n(D)$ is~an inner form of $\Sp_{2n}(F)$.

\begin{rema}
\label{remtranspose}
The reader may be more familiar with the transpose
with respect to the~dia\-gonal.
Let $J=J_{n}$ denote the antidiagonal matrix
of $\GL_{n}(F) \subseteq G$
with antidiagonal~entries~all equal to $1$.  
Then the transpose of a matrix $a\in\Mat_n(D)$ with respect to the
diagonal~is the~conju\-ga\-te of $ {}^{\tt} a$ by $J$. 
\end{rema}

\section{Preliminaries on groups and representations} 
\label{apfelstrudel}

\subsection{}
\label{coucou}

Let $G$ be a locally compact, totally disconnected topological group. 
By \textit{representation} of~a closed subgroup $H$ of $G$,
we mean a smooth, com\-plex representation of $H$.
By \textit{character} of $H$,~we mean a group homomorphism from $H$
to $\CC^\times$ with open kernel.~If $\pi$ is a representation of $H$, 
we denote by $\pi^\vee$ its contragredient.
Given a character $\chi$~of $H$,
we denote by $\pi\chi$ the~re\-pre\-sentation $h \mapsto \chi(h)\pi(h)$
of $H$.

If $\s$ is a continous involution of $G$,
we denote by $\pi^\s$ the representation~$\pi\circ\s$ of $\s(H)$.
Given a~closed subgroup $K$ of $H$,
the representation $\pi$ is said to be \textit{distinguished by} $K$
if its~under\-ly\-ing vector spa\-ce~$V$~carries
a non-zero linear form ${\it\Lambda}$ such 
that ${\it\Lambda} (\pi(x)v) = {\it\Lambda} (v)$~for all $x\in {K}$,
$v\in V$. 

We also denote by $\d_H$ the modulus character of $H$.

\subsection{}
\label{defdebase}

Given a non-Archimedean locally compact field $F$, 
we will denote
by $\Oo_F$ its ring of integers,
by $\p_F$ the maximal ideal of $\Oo_F$,
by $\kk_F$ its residue field and
by $|\cdot|_F$ the absolute value on $F$~sen\-ding~any uniformizer to the
inverse of the cardinality of $\kk_F$.

Similarly, 
given a finite-dimensional central division $F$-algebra $\De$, 
we denote
by $\Oo_{\De}$ its ring~of inte\-gers,
by $\p_{\De}$ the maximal ideal of $\Oo_{\De}$ and
by $\kk_{\De}$ its residue field.

\subsection{}
\label{defip}

Let $F$ be a non-Archimedean locally compact field, 
$\De$ be a finite-dimensional central~divi\-sion $F$-algebra
and $n$ be a positive integer.
The group $\G=\GL_n(\De)$ is locally compact and~total\-ly~disconnected. 

Let $(n_1,\dots,n_r)$ be a composition of $n$ and
$P=MN$ be the standard parabolic subgroup~of~$\G$
associated with it (\S\ref{Seville}).
Given a representation $\s$ of $M$,
we denote by $\ip^G_P(\s)$ the~repre\-sen\-ta\-tion of $G$
obtained from $\s$
by (normalized) parabolic induction along $P$.

For $i=1,\dots,r$,~let $\pi_i$ be
a representation of $\GL_{n_i}(\De)$.
We write
\begin{equation*}
\pi_1 \times \dots \times \pi_r
\end{equation*}
for the parabolically induced representation 
$\ip^G_P(\pi_1 \otimes \dots \otimes \pi_r)$
of $G$.

\subsection{}
\label{defgeneric}

Suppose that $\De=F$ and 
let $U$ denote the subgroup of upper~trian\-gular unipotent matrices
of $\G=\GL_n(F)$. 
Fix a non-trivial additive character $\psi : F \to \CC^\times$.
It gives rise to a character
\begin{equation*}
u \mapsto \psi(u_{1,2}+u_{2,3}+\cdots+u_{n-1, n}) 
\end{equation*}
of $U$, which we still denote by $\psi$.

Given any irreducible representation $\pi$ of $\G$,
the dimension of the vector space $\Hom_{U}(\pi, \psi)$~is at most $1$
(see \cite{GK}). 
We say that $\pi$ is \emph{generic} if this space is non-zero.

\section{The Jacquet--Langlands correspondence}
\label{JL}

\subsection{}
\label{DSclassiflocal}

Let $F$ be a non-Archimedean locally compact field,
$\De$ be a finite-dimensional central~divi\-sion
$F$-algebra~of~re\-du\-ced degree denoted $d$ 
and $n$ be a positive integer.
Let us recall the~classifi\-cation of the discrete series of the groups
$\GL_n(\De)$, $n\>1$
(\cite{ZelevinskiAENS80,Tadic,BHLS}).

Given any cuspidal representation $\rho$ of $\GL_{n}(\De)$ for some $n\>1$,
there is a unique positive~inte\-ger $r=r(\rho)$~such that,
for any integer $m\>2$, 
the parabolically induced representation
\begin{equation}
\label{iprhom}
\rho\nu^{r(1-m)/2} \times \rho \nu^{r(3-m)/2} \times \dots \times
\rho\nu^{r(m-1)/2} 
\end{equation}
{is reducible,
where $\nu$ denotes the character of $\GL_n(\De)$ defined as 
the~com\-posi\-tion 
of the norma\-li\-zed absolute value of $F$ with the reduced norm.
For $m\>1$,
the representation \eqref{iprhom} has a unique irreducible quotient,
which we denote by $\St_m(\rho)$.
This quotient is a discrete series representation of $\GL_{nm}(\De)$,
which is unitary if and only if $\rho$ is unitary.}

The integer $r$ associated with $\rho$
(it~is denoted $s(\rho)$ in \cite{VSautodual} \S3.5)
has the following properties:
\begin{itemize}
\item
it divides the reduced degree $d$ of $\De$ (\cite{VSautodual} Remark 3.15(1)),
\item
it is prime to $n$ (\cite{VSautodual} Remark 3.15(2)).
\end{itemize} 
In particular,
when $\De$ is isomorphic to $F$,
one has $r=1$ for all cuspidal representations $\rho$.

Conversely,
if $\pi$ is a discrete series representation of $\GL_{n}(\De)$ for some 
$n\>1$, 
there are a~uni\-que integer $m$~di\-viding $n$
and a cuspidal representation $\rho$ of $\GL_{n/m}(\De)$,
uniquely determined up to~iso\-morphism,
such that $\pi$ is isomorphic to $\St_m(\rho)$. 

\begin{rema}
\label{academie}
Note that \eqref{iprhom} also has a unique irreducible subrepresentation,
which we will denote by $\Sp_m(\rho)$.
\end{rema}

\subsection{}
\label{JLcor}

The (local) Jacquet--Langlands correspondence (\cite{Rogawski,DKV,Badulescu})
is a bijection between the~dis\-cre\-te se\-ries of $\GL_n(\De)$ and that of
$\GL_{nd}(F)$
characterised by a character relation on elliptic~regu\-lar~con\-ju\-ga\-cy
classes. 
If $\pi$ is a discrete series representation of $\GL_n(\De)$ for some $n\>1$,
its~Jac\-quet--Lang\-lands transfer will~be denoted by ${}^{\rm JL}\pi$. 

Let $\rho$ be a cuspidal representation of $\GL_n(\De)$ for some $n\>1$,
and set $r=r(\rho)$.
Its ~Jac\-quet--Lang\-lands transfer ${}^{\rm JL}\rho$~is isomorphic to $\St_{r}(\tau)$
for a cus\-pidal represen\-ta\-tion $\tau$ of $\GL_{nd/r}(F)$~and, 
for all $m\>1$,
the~Jac\-quet--Lang\-lands transfer of $\St_{m}(\rho)$ is $\St_{mr}(\tau)$
(see for instance \cite{MSjl}~Proposi\-tion 12.2).

Conversely,
given a positive integer $n\>1$, 
a divisor $k$ of $nd$
and a cuspidal representation~$\tau$~of $\GL_{nd/k}(F)$,
the~discrete series~re\-pre\-sentation $\pi$ of $\GL_n(\De)$
whose Jacquet--Langlands transfer~is $\St_k(\tau)$
is of the form $\St_{m}(\rho)$
for some cuspidal representation $\rho$ of $\GL_{nd/m}(\De)$, 
whe\-re $m$ is the greatest
common divisor of $k$ and $n$
(\cite{VSautodual} Remark 3.15(3)).

\subsection{}
\label{DSclassifglobal}

Let $k$ be a number field
and $\BB$ be a finite-dimensional central~division $k$-algebra
of redu\-ced~degree $d$.
Let $\AA=\AA_k$ be the ring of adèles of $k$.
For each place $v$~of~$k$,
let $k_v$ be~the~comple\-tion~of $k$ at $v$
and set $\BB_v=\BB\otimes k_{v}$.

We recall the~clas\-si\-fication~of the~dis\-cre\-te series
automorphic~re\-pre\-sentations of $\GL_n(\BB\otimes_{k}\AA)$
for all $n\>1$ 
(see \cite{MWres}
and \cite{BadulescuJLglobal} Propo\-si\-tion~5.7, Remark 5.6
and \cite{Badulescu-Renard} Proposition 18.2).

Given an integer $n\>1$, 
we denote by $\nu$ the automorphic character of 
$\GL_n(\BB\otimes_{k}\AA)$~obtained by~composing the reduced norm
$\GL_n(\BB\otimes_{k}\AA) \to \AA^\times$
with the idelic norm $\AA^\times\to\CC^\times$.
Thus, for each place $v$ of $k$,
the local component of $\nu$ at $v$,
denoted by $\nu_v$,
is the character ``normalized~ab\-solute~value of the reduced norm''
of $\GL_n(\BB_v)$.

Given any cuspidal automorphic representation $\Si$ of
$\GL_n(\BB\otimes_{k}\AA)$ for some $n\>1$,
there is~a
positive integer $r=r(\Si)$~such that,
for any $m\>1$, 
the parabolically induced~representa\-tion
\begin{equation}
\label{defDSG}
\Si\nu^{r(1-m)/2} \times \Si\nu^{r(3-m)/2} \times \dots \times
\Si\nu^{r(m-1)/2} 
\end{equation}
has a unique constituent which is a discrete series automorphic representation
of $\GL_{nm}(\BB\otimes_{k}\AA)$. 
This constituent is denoted by $\MW_m(\Si)$. 
Note that, 
if $\BB\simeq k$,
one has $r=1$ for all cuspidal~au\-to\-morphic representations $\Si$. 

Conversely,
if $\Pi$ is a discrete series automorphic representation of 
$\GL_{n}(\BB\otimes_{k}\AA)$ for some~inte\-ger~$n\>1$, 
there are a unique inte\-ger $m$ dividing $n$
and a unique cuspidal automorphic represen\-ta\-tion $\Si$ of
$\GL_{n/m}(\BB\otimes\AA)$
such that $\Pi$ is isomorphic to $\MW_m(\Si)$. 

\subsection{}
\label{JLcorglobal}
\label{sec:global-JL}

The (global) Jacquet--Langlands correspondence
is an injection $\Pi\mapsto{}^{\rm JL}\Pi$
from~the au\-to\-morphic dis\-cre\-te se\-ries of $\GL_n(\BB\otimes_{k}\AA)$
to that of $\GL_{nd}(\AA)$ characterised by the fact that,
for any discrete~se\-ries automorphic representation $\Pi$ of 
$\GL_n(\BB\otimes_{k}\AA)$ 
and any place $v$ of $k$ such~that $\BB_v$ is split,
the local components of ${}^{\rm JL}\Pi$ and $\Pi$ at $v$ are isomorphic 
once $\GL_n(\BB_v)$ and $\GL_{nd}(k_v)$ are identified
(\cite{BadulescuJLglobal} Theorem 5.1 
and \cite{Badulescu-Renard} Theorem 1.4).

At finite places of $k$ where $\BB$ does not split,
we will only need the following result.

\begin{lemm}
\label{HeleneZoubiratchvili}
Let $\Pi$ be a discrete~se\-ries automorphic representation of 
$\GL_n(\BB\otimes_{k}\AA)$.
Let~$v$ be~a finite place such that $\BB_v$ is not split
and let $\pi$ denote the local component of $\Pi$ at $v$.
Suppose that the Jacquet--Lang\-lands~transfer of $\pi$ to
$\GL_{nd}(k_v)$ is cuspidal.
Then ${}^{\rm JL}\Pi$ is cuspidal and its local component at $v$
is ${}^{\rm JL}\pi$.
\end{lemm}

\begin{proof}
Let $\pi'$ denote the local component of ${}^{\rm JL}\Pi$ at $v$.
Let $\LJ$ denote the Langlands--Jacquet morphism
(defined in \cite{BadulescuJLglobal} \S2.7)
from the Grothendieck group of the category of representations~of
$\GL_{nd}(k_v)$ of finite length to that of $\GL_{n}(\BB_v)$. 
By \cite{BadulescuJLglobal} Theorem 5.1(a),
there is a sign $\epsilon\in\{-1,1\}$ such that 
$\LJ(\pi') = \epsilon\cdot\pi$. 

On the other hand,
given a unitary representation $\kappa'$ of $\GL_{nd}(k_v)$
such that $\LJ(\kappa')$ is non-zero, 
it follows from 
the~classi\-fication of the unitary dual of $\GL_{nd}(k_v)$ in
\cite{TadicUnitaryDual}
and the description~of~the image of unitary~re\-pre\-sentations
by $\LJ$
in \cite{BadulescuJLglobal} Section 3 that,
if $\k$ is the unique unitary repre\-sen\-tation of $\GL_n(\BB_v)$ 
such that $\LJ(\kappa') \in \{-\k,\k\}$
and if we denote by $\rho_1+\dots+\rho_r$ the cus\-pi\-dal~sup\-port~of~$\k$
(which means that $\k$ is an irreducible component of
$\rho_1\times\dots\times\rho_r$),
then
\begin{equation*}
{\rm cusp}(\k')={\rm cusp}({}^{\JL}\rho_1)+\dots+{\rm cusp}({}^{\JL}\rho_r).
\end{equation*}
Applying this to the unitary representation $\pi'$, 
we obtain $\pi' = {}^{\rm JL}\pi$. 
\end{proof}

\section{A necessary condition of distinction for cuspidal
representations}
\label{sensdirect}

In this section, 
$F$ is a non-Archimedean locally compact field
and $D$ is a non-split quaternion $F$-algebra.
Fix a positive integer $n\>1$ and write $G=\GL_n(D)$. 
It is equipped with the~invo\-lu\-tion $\s$ defined in \S\ref{KyberGate}.
In \S\ref{Nathalie} only,
the field $F$ will be assumed to have characteristic $0$.

\subsection{}
\label{saccard}

Let $\pi$ be a cuspidal representation of $G$.
Associated with it in \S\ref{DSclassiflocal},
there is a positive~inte\-ger $r=r(\pi)$ which divides the reduced degree of
$D$ and is prime to $n$.

As the reduced degree of $D$ is equal to $2$, 
we immediately deduce that,
if $\pi$
has a non-cuspidal Jacquet--Langlands transfer to $\GL_{2n}(F)$,
then $r=2$ and $n$ is odd.
Its Jacquet--Langlands~trans\-fer ${}^{\rm JL}\pi$ thus has the form 
$\St_2(\tau)$ for some cuspidal representation of $\GL_n(F)$.

Conversely,
if $n$ is odd,
and if $\tau$ is any cuspidal representation of $\GL_n(F)$,
it follows from \S\ref{JLcor} that 
the unique discrete series~re\-pre\-sentation of $G$
whose Jacquet--Langlands transfer is $\St_2(\tau)$~is cuspidal. 

\subsection{}

{The following lemma will be used in the proof of Theorem 
\ref{Vermacarqcq}, and later in Section \ref{sec:final}.} 

\begin{lemm}
\label{lem:globalize}
Let $\pi$ be a {unitary}
cuspidal representation of $G$ distinguished by 
$\G^\s$.
Let
\begin{itemize}
\item 
$k$ be a global field together with a finite place $u$
dividing $p$ such that $k_u$ is isomorphic to $F$,
\item 
$\DD$ be a (non-split) quaternion $k$-algebra such that
$\DD_u=\DD\otimes_kk_u$ is non-split.
\end{itemize}
Then there exists a cuspidal automorphic representation $\Pi$ 
of $\GL_n(\DD\otimes_k \AA)$ such that
\begin{enumerate}
\item
$\Pi$ has a non-zero $\Sp_n(\DD\otimes_k \AA)$-period,
that is, 
there is a $\varphi\in\Pi$ such that
\begin{equation*}
\int_{\Sp_n(\DD)\bsl\Sp_n(\DD\otimes{\AA})} \varphi(h)\,\rd h\neq0,
\end{equation*}
\item 
the local component of $\Pi$ at $u$ is isomorphic to $\pi$.
\end{enumerate}
\end{lemm}

\begin{proof}
Let $\pi$ be a unitary cuspidal irreducible representation of $G$.
Assume that $\pi$ is distinguished~by $\G^\s$. 
Let $Z$ denote the centre of $G$,
which is isomorphic to $F^\times$,
and let $G'=\SL_n(D)$ be the~ker\-nel of the reduced norm
from $\GL_n(D)$ to $F^\times$.
By \cite{HenniartJALG01} Theorem 4.2,
the restriction of $\pi$ to the normal, cocompact, closed subgroup
$G_1=ZG'$ is semisimple of finite length.~Let~$\pi_1$ be
an~irre\-ducible sum\-mand of this restriction. 
The centre $Z$ acts on it through $\omega$,
the central~cha\-racter~of $\pi$.
The restriction of $\pi_1$ to $G'$, denoted $\pi'$, is thus irreducible.

Let $k$ be a global field together with a finite place $u$
dividing $p$ such that $k_u$ is isomorphic~to $F$.
Thus $k$ is a finite extension of $\QQ$ when $F$ has characteristic $0$,
and the field of rational~func\-tions~over a smooth irreducible
projective curve defined over a finite field of characteristic~$p$~if
$F$ has characteristic $p$.

Let $\DD$ be a quaternion algebra over $k$ such that
$\DD\otimes_kk_u$
is non-split (it is thus isomorphic~to $D$).~Let
$\GG$ be the $k$-group $\GL_n(\DD)$
and $\GG'$ be the $k$-group $\SL_n(\DD)$.
The latter is an inner form~of $\SL_{2n}$ over $k$
which contains $\HH=\Sp_n(\DD)$.
The connected component of the centre of~$\GG'$
is~trivial and $\HH$ is a closed algebraic $k$-subgroup of $\GG'$
with no non-trivial character. 
Let $V$ denote~the~$k$-vec\-tor space made of all~ma\-tri\-ces
$a \in \Mat_{n}(\DD)$ such that $a^*=a$
and {consider}~the~algebraic~re\-pre\-sen\-tation of $\GG'$ on $V$
defined by $(g,a) \mapsto \s(g) a g^{-1}$.
This representation is semisimple (it is even irreducible)
and~the~$\GG'$-sta\-bi\-li\-zer of the identity matrix (on $V$) is $\HH$. 

We now apply either \cite{PrasadSchulzePillot} Theorem 4.1
(if $F$ has characteristic $0$)
or \cite{GanLomeli} Theorem 1.3 (if~$F$~has characteristic $p$):
there exists a cuspidal automorphic representation $\Pi'$ 
of $\GG'(\AA)$ with a non-zero $\HH(\AA)$-period 
and such that the local component of $\Pi'$ at $u$ is isomorphic to $\pi'$.
(Here $\AA$ denotes the ring of ad{\`e}les of $k$.)

By \cite{LabesseSchwermer} Theorem 5.2.2, 
the representation $\Pi'$ occurs as a subrepresentation
in the restriction to~$\GG'(\AA)$ of a
cuspidal automorphic representation $\Pi$ of $\GG(\AA)$.
Since $\GG'$ contains $\HH$,
the~repre\-sentation $\Pi$ has a non-zero $\HH(\AA)$-period.
It follows that:
\begin{enumerate}
\item[(1)] 
for any finite place $v$, 
the local component $\Pi_v$ of $\Pi$ at $v$ is distinguished by $\HH(k_v)$, 
\item[(2)] 
the restriction to $\GG'(k_u)$ of the local component $\Pi_u$ of $\Pi$ at $u$ 
contains $\pi'$.
\end{enumerate}
More precisely,
let us prove that 
$\Pi_u$ is isomorphic to $\pi\otimes(\chi\circ{\Nrd_{A/F}})$
for some character $\chi$ of the group $F^\times$.
Arguing as at the beginning of the proof of the theorem,
the restriction of $\Pi_u$ to $G_1$ is semi\-sim\-ple~of finite length
and contains an irreducible summand $\pi_2$ whose
restriction to $G'$~is isomorphic to $\pi'$.
If $\mu$ denotes the central character of $\Pi_u$,
we thus have $\pi_2(zx)=\mu(z)\pi'(x)$~for all $z\in Z$ and $x\in G'$.
The representation $\Pi_u$ is unitary as a local~component of the
unitary~re\-pre\-sentation $\Pi$.
Its central character $\mu$ is thus unitary.
Similarly, we have $\pi_1(zx)=\omega(z)\pi'(x)$ for all $z\in Z$ and
$x\in G'$.
By~twist\-ing $\pi$~by a unitary character of $G$,
we may assume that $\omega=\mu$. 
By \cite{LabesseSchwermer} Proposition 2.2.2,
we get the expected result.

Let us fix a unitary character $\Theta$ of $\AA^\times/k^\times$
whose local component at $u$ is $\chi$. 
By twisting $\Pi$~by
$\Theta^{-1}$ composed with the reduced norm from $\GG(\AA)$
to~$\AA^\times$,
we obtain a cuspidal automorphic~re\-pre\-sentation 
of $\GG(\AA)$ having the required properties. 
\end{proof}

\subsection{}
\label{Nathalie}

In this paragraph,
$F$ is a non-Archimedean locally compact field of characteristic $0$.

\begin{theo}
\label{Vermacarqcq}
Suppose that $F$ has characteristic $0$,
and let $\pi$ be a cuspidal representation~of $G$.
If $\pi$ is distinguished by $\G^\s$, then
${}^{\rm JL}\pi$ is non-cuspidal.
\end{theo}

\begin{proof}
This follows from \cite{Verma} Theorem 1.2
but we give a more~de\-tailed proof,
based on~Lem\-mas \ref{HeleneZoubiratchvili} and \ref{lem:globalize}. 
Our argument is {inspired by} \cite{Verma} Section 6,
in particular the proof of Theorem 6.3.

Let $\pi$ be a cuspidal irreducible representation of $G$
distinguished by $\G^\s$.
Assume that ${}^{\rm JL}\pi$ is cuspidal.

Since any unramified character of $G$ is trivial on $G^\s$,
and since the Jacquet--Langlands corres\-pondence is compatible
with torsion by unramified characters,
we may and will assume,
by~twist\-ing by an appropriate unramified character, 
that $\pi$ is unitary. 

Let $k$ be a number field together with a finite place $u$
dividing $p$ such that $k_u$ is isomorphic~to $F$.
Let $\DD$ be a quaternion division algebra over $k$ such that
$\DD_u$ is non-split
(it is thus isomorphic to $D$)
and $\DD_v$ is split for all Archimedean places $v$.
By Lemma \ref{lem:globalize},
there exists a cuspidal~auto\-morphic representation $\Pi$ 
of $\GL_n(\DD\otimes_k \AA)$ such that
\begin{enumerate}
\item
$\Pi$ has a non-zero $\Sp_n(\DD\otimes_k \AA)$-period, 
\item 
the local component of $\Pi$ at $u$ is isomorphic to $\pi$.
\end{enumerate}
The Jacquet--Lang\-lands transfer ${}^{\rm JL}\Pi$ of such a $\Pi$
is a discrete series automorphic representation of
$\GL_{2n}(\AA)$ with the~follow\-ing~properties:
\begin{enumerate}
\item[(3)]
by Lemma \ref{HeleneZoubiratchvili},
its local component at $u$ is cuspidal, isomorphic to ${}^{\rm JL}\pi$,
thus ${}^{\rm JL}\Pi$ is~cus\-pi\-dal, 
\item[(4)]
{for any finite place $v$ such that $\DD_v$ is split,
the local components of ${}^{\rm JL}\Pi$ and $\Pi$ at $v$~are~iso\-morphic.}
\end{enumerate}
Since ${}^{\rm JL}\Pi$ is cuspidal, it is generic.
Therefore:
\begin{enumerate}
\item[(5)]
for any finite place $v$, 
the local component of ${}^{\rm JL}\Pi$ at $v$ is generic.
\end{enumerate}
On the other hand,
it follows from (1) that 
\begin{enumerate}
\item[(6)]
for any finite place $v$, 
the local component of $\Pi$ at $v$ is distinguished by $\Sp_{n}(\DD_v)$.
\end{enumerate}
It follows from (4), (5) and (6) that,
if $v$ is a finite place of $k$ such that $\DD_v$ is split,
$\Pi_v$ is an~irredu\-cible~representation of $\GL_{2n}(k_v)$
which is ge\-ne\-ric~and distinguished by $\Sp_n(\DD_v) \simeq \Sp_{2n}(k_v)$.
This contradicts \cite{OffenSayagJFA08}~Theo\-rem~1,
which says that no generic irreducible representation of $\GL_{2n}(k_v)$~is
distinguished by $\Sp_{2n}(k_v)$.
\end{proof}

\begin{rema}
\label{VermacarqcqZ}
{We assumed that $F$ has characteristic $0$ in Theorem \ref{Vermacarqcq} 
  because a global~Jac\-quet--Langlands correspondence
  for discrete series automorphic representations of
  $\GL_n(\BB\otimes_{k}\AA)$
  is not known to exist in characteristic $p$,
  except when $n=1$ (see \cite{BadulescuRoche}).}
\end{rema}

{Sections \ref{SecTypes} to \ref{Sec6} 
are devoted to the proof of Theorem~\ref{MAINTHM1}:
assuming that $F$ is any non-Archi\-me\-dean locally compact field
with odd residue characteristic, 
and that the Jacquet--Langlands transfer of any cuspidal representation
of $\G$ distinguished by $\G^\s$ is~non-cuspidal
(which holds for instance when $F$ has characteristic $0$
by Theorem \ref{Vermacarqcq}),
we prove that any
cuspidal representation of $\G$ whose~Jac\-quet--Langlands
transfer is non-cuspidal is~dis\-tinguished by $\G^\s$.}

\section{Type theoretic material}
\label{SecTypes}

In this section,
we introduce the type theoretic material which we will need in
Sections~\ref{Sec2}--\ref{thecircle}.
{Let $\De$ be
\textit{any finite-dimensional central division $F$-algebra}
(this extra generality will be useful~in \S\ref{enrico}.)}
Let $A$ be the central simple $F$-algebra $\Mat_{n}(\De)$ of $n \times n$ matrices 
with~coef\-fi\-cients in $\De$ for some integer $n\>1$,
and $\G=A^\times =\GL_n(\De)$. 
Let us fix a character
\begin{equation}
\label{psicond1}
\psi : F \to \CC^\times
\end{equation}
trivial on $\p_F$ but not on $\Oo_F$.
For the definitions and main results stated in this section,
we refer the reader to \cite{BK,BHEffective} (see also \cite{VSautodual,NRV}).

\subsection{}
\label{prelim}

A \textit{simple stratum} in $A$ is a pair $[\aa,\b]$ made of
a hereditary $\Oo_\F$-order $\aa$ of $\A$ and 
an element $\b \in \A$ such that the $\F$-algebra $\E=\F[\b]$ is a field, 
and the multiplicative group $\E^\times$ normalizes $\aa$
(plus an extra technical condition on
$\b$ which it is not necessary to recall here).
The centralizer~$B$ of $E$ in $A$ is a central simple $E$-algebra, 
and $\bb=\aa\cap\B$ is a {hereditary $\Oo_E$-order} in $\B$.

Associated {to} a simple stratum $[\aa,\b]$, 
there are a pro-$p$-subgroup $\H^1(\aa,\b)$ of $G$ 
and a~non-empty finite set $\Cc(\aa,\b)$
of characters of $\H^1(\aa,\b)$ called \textit{simple~cha\-racters}, 
depending on $\psi$.

\begin{rema}
\label{nullstratum}
This includes the case where $\b=0$.
The simple stratum $[\aa,0]$ is then said to be \textit{null}.
One has $\H^1(\aa,0) = 1+\p_\aa$
(where $\p_\aa$ is the Jacobson radical of $\aa$)
and the set $\Cc(\aa,0)$ is reduced to the trivial character of $1+\p_\aa$.
\end{rema}

When the order $\bb$ is maximal in $B$,
the simple stratum $[\aa,\b]$ is said to be \textit{maximal},
and the simple characters in $\Cc(\aa,\b)$ are said to be \textit{maximal}. 
If this is the case,
and if $[\aa',\b']$ is another simple stratum in $A$ such that
$\Cc(\aa,\b) \cap \Cc(\aa',\b')$ is non-empty,
then
\begin{equation}
\label{invmax}
\Cc(\aa',\b') = \Cc(\aa,\b), \quad
\aa' = \aa, \quad
[F[\b']:F] = [F[\b]:F],
\end{equation} 
and the simple stratum $[{\aa'},\b']$ is maximal (\cite{VSautodual} Proposition 3.6).

\subsection{}
\label{nom42}

Let $\De'$ be
a finite dimensional central division $F$-algebra
and $[\aa',\b']$ be a simple stratum~in $\Mat_{n'}(\De')$
for some $n'\>1$.
Assume that there is a morphism of $F$-algebras
$\h : F[\b] \to F[\b']$~such that $\h(\b)=\b'$.~Then there is a natural
bijection from $\Cc(\aa,\b)$ to $\Cc(\aa',\b')$ called transfer.

\subsection{}

Let $\Cc$ denote the union of the sets $\Cc(\aa',\b')$ for all maximal simple
strata $[\aa',\b']$ of $\Mat_{n'}(\De')$,
for all $n'\>1$ and all finite dimensional central division $F$-algebras $\De'$.
Any two maximal simple characters $\t_1,\t_2 \in \Cc$ are said to be
\textit{endo-equivalent} if they are transfers of each other,
that is,
if there exist
\begin{itemize}
\item 
maximal simple strata $[\aa_1,\b_1]$ and $[\aa_2,\b_2]$, 
\item
a morphism of $F$-algebras $\h : F[\b_1] \to F[\b_2]$~such that $\h(\b_1)=\b_2$, 
\end{itemize}
such that $\t_i\in\Cc(\aa_i,\b_i)$ for $i=1,2$, 
and $\t_2$ is the image of $\t_1$ by the transfer map
from $\Cc(\aa_1,\b_1)$ to $\Cc(\aa_2,\b_2)$.
This defines an equivalence relation on $\Cc$,
called \textit{endo-equivalence}.
An equivalence class for this equi\-va\-len\-ce relation is called an
\textit{endoclass}. 

The \textit{degree} of an endoclass $\TT$ is the degree of $F[\b']$ over $F$,
for any choice of $[\aa',\b']$ such that $\Cc(\aa',\b') \cap \TT$ is
non-empty. 
(It is well defined thanks to \eqref{invmax}.)

\subsection{}

Let $\Cc(G)$ be the union of the sets $\Cc(\aa,\b)$ for all maximal simple
strata $[\aa,\b]$ of $\A$.
Any two maximal simple characters $\t_1,\t_2 \in \Cc(G)$ 
are endo-equivalent if and only if they are $\G$-conjugate.  

Given a cuspidal representation $\pi$ of $G$, 
there exists a maximal simple character $\t \in \Cc(G)$~con\-tained in $\pi$, 
and any two maximal simple characters contained in $\pi$ are $G$-conjugate.
The~max\-imal simple characters contained in $\pi$ thus all belong to the
same~endo\-class~$\TT$,~called the~endo\-class of $\pi$. 
Conversely, 
any maximal simple character $\t\in\Cc(G)$ of en\-do\-class $\TT$
is contained in $\pi$ (\cite{VSautodual} Corollaire 3.23).

\subsection{}
\label{prelimmax3}

Let $\t\in\Cc(\aa,\b)$ be a simple character
with respect to a maximal simple stratum $[\aa,\b]$ in~$A$ as in \S\ref{prelim}.
There are a divisor $m$ of $n$ and a finite dimensional central division
$E$-algebra $\Ce$ such that $B$ is isomorphic to $\Mat_m(\Ce)$.
Let $\BJ_\t$ be the normalizer of $\t$ in $\G$.
Then 
\begin{enumerate}
\item
the group $\BJ_\t$ has a unique maximal compact subgroup $\BJ^0=\BJ^0_\t$
and a unique maximal~nor\-mal pro-$p$-subgroup $\BJ^1=\BJ^1_\t$, 
\item
the group $\BJ_\t\cap\B^\times$ is the normalizer of $\bb$ in $\B^\times$ and
$\BJ^0\cap\B^\times=\bb^\times$, $\BJ^1\cap\B^\times=1+\p_{\bb}$,
\item
one has $\BJ_\t = (\BJ_\t\cap\B^\times)\BJ^0$ and
$\BJ^0 = (\BJ^0\cap\B^\times)\BJ^1$.
\end{enumerate}
Since $\bb$ is a maximal order in $B$,
it follows from (2) and (3) that there is~a group iso\-mor\-phism 
\begin{equation}
\label{JJ1UU1GLmax3}
\BJ^0/\BJ^1 \simeq \GL_{m}(\ll)
\end{equation}
where $\ll$ is the residue field of $\Ce$,
and an element $\w\in\B^\times$ normalizing $\bb$ such that $\BJ_\t$ is
generated by $\BJ^0$ and $\w$.

There is an irreducible representation $\n=\n_\t$ of $\BJ^1$
whose restriction to $\H^1(\aa,\b)$ contains $\t$.~It
is~unique up to isomorphism,
and   it  is   called   the   \textit{Heisenberg  representation}   associated
with~$\t$.~It
ex\-tends to the group $\BJ_\t$ (thus its normalizer in $G$ 
is equal to $\BJ_\t$).

If $\bk$ is a representation of $\BJ_\t $ extending $\n$, 
any other extension of $\n$ to $\BJ_\t $ has the form $\bk\bx$
for a unique character $\bx$ of $\BJ_\t $ trivial on $\BJ^1$.
More generally, the map
\begin{equation}
\boldsymbol{\tau} \mapsto \bk\otimes\boldsymbol{\tau}
\end{equation}
is a bijection between
isomorphism classes of irreducible representa\-tions of $\BJ_\t $ trivial on $\BJ^1$ and 
isomorphism classes of irreducible representations of $\BJ_\t $ whose~res\-triction 
to $\BJ^1$ contains $\n$.

\section{$\s$-self-dual simple characters}
\label{Sec2}

\textit{From this section until Section \ref{Sec6},
the residue characteristic $p$ of $F$ is assumed to be odd.}

We go back to the group $G = \GL_n(D)$ of \S\ref{defdebase} 
and fix a cuspidal representation $\pi$ of $G$
with non-cuspidal transfer to $\GL_{2n}(F)$.
In particular,
as explained in \S\ref{saccard},~the~in\-teger $n$ is odd.
The main result of this section is the following proposition. 
Recall that $*$ has been defined in \S\ref{KyberGate}.

\begin{prop}
\label{camille}
There are a~maxi\-mal simple stratum $[\aa,\b]$ in
$\A$ and a maximal simple~cha\-racter $\t \in \Cc(\aa,\b)$
contained in $\pi$ such that
\begin{enumerate}
\item
the group $H^1(\aa,\b)$ is stable by $\s$ and $\t\circ \s = \t^{-1}$,
\item
the order $\aa$ is stable by $*$ and $\b$ is invariant by $*$.
\end{enumerate} 
\end{prop}

\subsection{}
\label{degTheta}

Let $\TT$ denote the endoclass of $\pi$. 
Since $\pi$~contains any maximal simple character 
of~$\Cc(G)$ of~endo\-class $\TT$, 
it suffices to prove the exis\-ten\-ce of a maximal simple stratum $[\aa,\b]$ in 
$A$ and a character~$\t\in\Cc(\aa,\b)$ of
endoclass~$\TT$~satisfy\-ing~Condi\-tions (1) and (2) of~Pro\-position \ref{camille}.

Since the Jacquet--Langlands transfer of $\pi$ is non-cuspidal,
it follows from \S\ref{JLcor}
that this~trans\-fer is of the form $\St_2(\tau)$ 
for a cuspidal irreducible representation $\tau$ of $\GL_n(F)$.
By Dotto \cite{Dotto}, the representations $\pi$ and $\tau$ have the same 
endoclass.
It follows that the degree of $\TT$ divides~$n$. 

\subsection{}

Let $d$ denote the degree of $\TT$.
Thanks to \S\ref{degTheta}, it is an odd integer dividing $n$.

Let $\s_0$ be the involution $x \mapsto {}^{\intercal}x^{-1}$ on 
$\GL_d(F)$
where $\intercal$ denotes transposition with respect to the antidiagonal. 
The fixed points of $\GL_d(F)$ by $\s_0$ is a split orthogonal group. 

By \cite{ZouO} Theorem 4.1,
there are
a maximal simple stratum $[\aa_0,\b]$ in $\Mat_{d}(F)$ and 
a maximal~sim\-ple~character $\t_0 \in \Cc(\aa_0,\b)$
of endoclass $\TT$ such that
\begin{itemize}
\item
the group $H^1(\aa_0^{\phantom{1}},\b)$ is stable by $\s_0^{\phantom{1}}$
and $\t_0^{\phantom{1}} \circ \s_0^{\phantom{1}} = \t_0^{-1}$,
\item
the order $\aa_0$ is stable by $\intercal$
and $\b$ is invariant by $\intercal$.
\end{itemize} 
Write $E$ for the sub-$F$-algebra $F[\b] \subseteq \Mat_d(F)$.
It is made of $\intercal$-invariant matrices. 
Its centralizer in $\Mat_d(F)$ is equal to $E$ itself.
The intersection $\aa_0\cap E$ is $\Oo_E$, the ring of integers of $E$.

\subsection{}

Let us write $n=md$.
We identify $\Mat_n(F)$ with $\Mat_m(\Mat_{d}(F))$ 
and $E$ with its diagonal~ima\-ge in $\Mat_n(F)$.
The centralizer of $E$ in $\Mat_n(F)$ is thus $\Mat_m(E)$.

Now consider $\Mat_n(F)$ as a sub-$F$-alge\-bra of $A$.
The centralizer $B$ of $E$ in $\A$ {is} equal to $\Mat_m(C)$
where $C$ is an $E$-algebra isomorphic to $E \otimes_F D$.
Since the degree $d$ of $E$ over $F$ is odd,
$C$ is a non-split quaternion~$E$-algebra.
Denote by $*_B$~the~anti-in\-volution on $B$ analogous to \eqref{antiinvA}.

\begin{prop}\label{prop:res*}
The restriction of $*$ to $B$ is equal to $*_B$.
\end{prop}

\begin{proof}
It suffices to treat the case where $m=1$.
We will thus assume that $m=1$, in which case we have $B=C$.
We thus have to prove that
\begin{equation}
\label{eqtoprove}
c^* = {\rm trd}_{C/E}(c) - c
\end{equation}
for all $c\in C \subseteq \Mat_d(D)$.
Let us identify $\Mat_d(D)$ with $\Mat_d(F)\otimes_FD$.
Then $(a\otimes x)^* = {}^{\intercal}a\otimes\overline{x}$ for 
all $a\in\Mat_n(F)$ and $x\in D$,
and $C$ identifies with $E \otimes_F D$.
Thus \eqref{eqtoprove} is equivalent to
\begin{equation}
e \otimes \overline{x} = \trd_{C/E}(e\otimes x) - e\otimes x
\end{equation}
for all $e\in E$ and $x\in D$.
Thanks to \eqref{antiinvD},
we are thus reduced to proving that 
\begin{equation}
\trd_{C/E}(x) = \trd_{D/F}(x)
\end{equation}
for all $x\in D$,
where the $F$-algebra $D$ is embedded in $C$ via $x \mapsto 1\otimes_F x$.

Let $L$ be a quadratic unramified extension of $F$.
Since the degree of $E$ over $F$ is odd,
$E\otimes_FL$ is a field,
denoted $EL$.
The reduced trace is invariant by~ex\-ten\-sion of scalars
(see \cite{BourbakiAlg8} \S17.3,~Pro\-po\-sition 4).
Thus $\trd_{D/F}(x)$ is the trace of $x$ in
$D\otimes_F EL \simeq \Mat_{2}(EL)$.
(By the Skolem-Noether theorem, 
the computation of this trace does not depend on the
choice of the isomorphism.)~Si\-mi\-larly,
$\trd_{C/E}(x)$
is the trace of $x$ in $C\otimes_E EL \simeq \Mat_{2}(EL)$.
The proposition is proven.
\end{proof}

Let $\bb$ denote the standard maximal order $\Mat_m(\Oo_C)$ in $B$.
Then $\bb^\times$ is a 
maximal open compact subgroup of $B^\times$ which is stable by $\s$.
Let $\aa$ denote the unique $\Oo_F$-order in $A$ normalized by $\E^\times$
such that $\aa\cap B = \bb$ (see \cite{VSrep2} Lemme 1.6).
We thus obtain a~maximal simple stratum $[\aa,\b]$ in $\A$
where $\aa$ is stable by $*$ and $\b^*=\b$,
and $E$ is made of $*$-invariant matrices.

Let $\t \in \Cc(\aa,\b)$ be the transfer of $\t_0$.
We are going to prove that the group 
$H^1(\aa,\b)$ is stable by $\s$ and $\t \circ \s = \t^{-1}$,
which will finish the proof of Proposition~\ref{camille}.
For this, 
set $\t^* = \t^{-1}\circ\s$.
This is a character of $\s(H^1(\aa,\b))$.
We thus have to prove that $\t^*=\t$.

Let $\vartheta_0$ be any character of $\Cc(\aa_0,\b)$ and
$\vartheta$ be its transfer to $\Cc(\aa,\b)$.~Let us define 
the~cha\-rac\-ters 
${}^{\intercal} \vartheta_0^{\phantom{1}}=\vartheta_0^{-1}\circ\s^{\phantom{1}}_0$
and $\vartheta^* = \vartheta^{-1}\circ\s$. 
By Lemma \ref{applem}
{(which we will prove in~a~separa\-te~section since its proof
requires techniques which are not used anywhere else
in the paper),}
we have
\begin{equation*}
\vartheta^*\in\Cc(\aa^*,\b^*),
\quad
{}^{\intercal}\vartheta_0 \in \Cc({}^{\intercal}\aa_0, {}^{\intercal}\b).
\end{equation*}
On the one hand,
by \cite{SkodlerackRT20} Proposition 6.3, 
the transfer of ${}^{\intercal}\vartheta_0 \in
\Cc({}^{\intercal}\aa_0,{}^{\intercal}\b)$~to $\Cc(\aa^*,\b^*)$ 
is equal~to $\vartheta^*$. 
On the other hand, 
we have
$\Cc(\aa^*,\b^*)=\Cc(\aa,\b)$ 
since $\aa$ is stable by $*$ and $\b^*=\b$, 
and~like\-wise~$\Cc({}^{\intercal}\aa_0, {}^{\intercal}\b)=\Cc(\aa_0,\b)$. 
Now choose $\vartheta_0=\t_0$.
Since ${}^{\intercal}\t_0=\t_0$,~we
deduce that $\t^*=\t$. 

\section{$\s$-self-dual extensions of Heisenberg representations}
\label{relaxe}

In this section,
the residue characteristic $p$ of $F$ is odd.
We focus on the maximal simple~stra\-tum $[\aa,\b]$ 
and the maximal simple~cha\-racter $\t \in \Cc(\aa,\b)$
constructed in Section \ref{Sec2},
forgetting~tem\-porarily about the cuspidal representation~$\pi$.
We thus have $\aa^*=\aa$, $\b^*=\b$ and
$\t^{-1}\circ\s=\t$.~Re\-call~that the centralizer $\B$ of $E$ in $\A$ {is} equal
to $\Mat_m(C)$ where $m[E:F]=n$ and $C$ is
a~qua\-ter\-nion~$E$-algebra isomorphic to $E \otimes_F D$,
and $\bb$ is the standard maximal order $\Mat_m(\Oo_C)$ in $B$.
Let $\TT$ denote the endoclass of~$\t$.

\subsection{}
\label{prelimmax}

Let $\BJ_\t$ be the normalizer of $\t$ in $\G$.
According to \S\ref{prelimmax3}, it has a unique maximal compact
subgroup $\BJ^0=\BJ^0_\t$ and a unique maximal~nor\-mal
pro-$p$-subgroup $\BJ^1=\BJ^1_\t$.
One has the~iden\-tity $\BJ_\t = C^\times\BJ^0$,
where $C^\times$ is~dia\-gonally embedded in 
$\GL_m(C) = B^\times\subseteq\G$,
and a group isomorphism 
\begin{equation}
\label{JJ1UU1GLmax}
\BJ^0/\BJ^1 \simeq \GL_{m}(\ll)
\end{equation}
where $\ll$ is the residue field of $\C$,
coming from the identities 
$\BJ^0 = (\BJ^0\cap\B^\times)\BJ^1$, 
$\BJ^0\cap\B^\times=\bb^\times$
and $\BJ^1\cap\B^\times=1+\p_{\bb}$. 

\subsection{}
\label{mst}

Let $\n$ denote the Heisenberg representation associated with $\t$
and $\bk$ be a representation~of $\BJ_{\t}$ extending~$\n$. 
Let $\varrho$ be a cuspidal irreducible representation of $\BJ^0/\BJ^1$.
Its inflation to $\BJ^0$ will still be denoted by $\varrho$.
The normalizer $\BJ$ of $\varrho$ in $\BJ_\t$ satisfies 
\begin{equation*}
\E^\times\BJ^0 \subseteq \BJ \subseteq \BJ_\t. 
\end{equation*}
Since $\E^\times\BJ^0$ has index $2$ in $\BJ_\t = C^\times\BJ^0$,
there are only two possible values for $\BJ$,
namely~$\E^\times\BJ^0$ and $\BJ_\t$.
More precisely,
$\BJ_\t$ is generated by $\BJ^0$ and a uniformizer $\w$ of $C$,
and the action of $\w$ on $\BJ^0$ by~con\-jugacy identifies through  
\eqref{JJ1UU1GLmax} with the action on $\GL_m(\ll)$
of the generator of $\Gal(\ll/\ll_0)$,
where~$\ll_0$ is the residue field of $E$.
It follows that $\BJ=\BJ_\t$ if and only if $\varrho$ is $\Gal(\ll/\ll_0)$-stable.
For the following three assertions,
see for instance \cite{VSautodual} 3.5.

Let $\bs$ be a representation of $\BJ$ extending $\varrho$.
Then the representation of $G$ compactly induced from
$\bk\otimes\bs$~is irreducible and cuspidal, of endoclass $\TT$.

Conversely, 
any cuspidal representation of $G$ of endoclass $\TT$ is obtained this way,
for~a~sui\-ta\-ble~choice of $\varrho$ and of an extension $\bs$ to $\BJ$.

Two pairs $(\BJ,\bk\otimes\bs)$ and $(\BJ',\bk\otimes\bs')$ constructed as
above give rise 
to the same cuspidal~re\-pre\-sentation of $G$ if and only if they are 
$\BJ_\t$-conjugate,
that is,
if and only if $\BJ'=\BJ$ and $\bs'$ is $\BJ_\t$-conjugate to $\bs$.

The following theorem will be crucial in our proof of Theorem \ref{MAINTHM1}. 

\begin{theo}
\label{BHpardeg}
The Jacquet--Langlands transfer to $\GL_{2n}(F)$ of the cuspidal
representation~of $G$ compactly induced from $(\BJ,\bk\otimes\bs)$
is cuspidal if and only if $\BJ = E^\times\BJ^0$.
\end{theo}

\begin{proof}
The Jacquet--Langlands transfer to $\GL_{2n}(F)$
of the cuspidal representation compactly induced from 
$(\BJ,\bk\otimes\bs)$ is~cuspi\-dal if and only if the integer $r$ associated
to it (in \S\ref{JLcor}) is $1$. 
By \cite{VSautodual} Remarque~3.15(1)~(which is based on \cite{BHJL3}),
this integer is equal to the order of the stabilizer of $\varrho$ in 
$\Gal(\ll/\ll_0)$,
that is, to the index of $E^\times\BJ^0$ in $\BJ$.
\end{proof}

\subsection{}
\label{enrico}

Let us prove that there exists a representation $\bk$ of $\BJ_\t$
exten\-ding $\n$ such that $\bk^{\s\vee}$ is~iso\-mor\-phic to $\bk$.
As in \cite{VSautodual} Lemme 3.28,
we prove it in a more general context (see Section~\ref{SecTypes}).

\begin{lemm}
\label{kappasigmaselfdual}
Let $\De$ be a finite dimensional central division $F$-algebra, 
let $\tau$ be a continuous au\-tomorphism of $\GL_r(\De)$
for some integer $r\>1$,
let $\vartheta$ be a maximal simple character of $\GL_r(\De)$ such that
$\vartheta\circ\tau=\vartheta^{-1}$, 
let $\BJ_\vartheta$ be its $\GL_r(\De)$-normalizer 
and $\n$ be its Heisenberg representation. 
\begin{enumerate}
\item 
The representation $\n^{\tau\vee}$ is isomorphic to $\n$. 
\item 
For any representation $\bk$ of $\BJ_\vartheta$ extending $\n$,
there exists a unique character $\bx$ of $\BJ_\vartheta$ trivial on $\BJ^1$
such that $\bk^{\tau\vee}$ is isomorphic to $\bk\bx$. 
\item
{Assume that the order of $\tau$ is finite and prime to $p$.}
There exists a representation $\bk$ of~$\BJ_\vartheta$ extending $\n$ 
such that $\bk^{\tau\vee}$ is~iso\-morphic to $\bk$.
\end{enumerate}
\end{lemm}

\begin{proof}
The first two assertions are given by \cite{VSautodual} Lemme 3.28. 
For the third one, note that
\begin{equation*}
{\rm val}_F \circ \Nrd \circ\, \tau =
\epsilon(\tau) \cdot {\rm val}_F \circ \Nrd 
\end{equation*}
where ${\rm val}_F$ is any valuation on $F$,
$\Nrd$ is the reduced norm on $\Mat_r(\De)$ 
and $\epsilon(\tau)$ is a sign uniquely determined by $\tau$.
Indeed,
the left hand side is a morphism from $\GL_r(\De)$ to $\ZZ$.
As $\tau$ is continuous,
it stabilizes the kernel of ${\rm val}_F \circ \Nrd$,
which is generated by~com\-pact~subgroups.
The left hand side thus factors through ${\rm val}_F \circ \Nrd$,
and the surjective~mor\-phisms from $\ZZ$ to $\ZZ$ are the identity and
$x\mapsto-x$.

If $\epsilon(\tau)=1$,
the result is given by \cite{VSautodual} Lemme 3.28.
{(Note that,
in this case,
the assumption~on the order of $\tau$ is unnecessary.)}
We thus assume that $\epsilon(\tau)=-1$.
Let $\bk$~be such that $\det(\bk)$~has $p$-power order on $\BJ_\vartheta$
(whose existence is granted by \cite{VSautodual} Lemme 3.12).~The representation
$\bk^{\tau\vee}$~is then isomorphic to $\bk{\bx}$ for some character $\bx$ of 
$\BJ_\vartheta$ trivial on $\BJ^1$. 
As in the proof of \cite{VSautodual} Lemma 3.28,
since $p$ is odd, 
this $\bx$ is trivial on $\BJ^0$ and it has $p$-power order.

The group $\BJ_\vartheta$ is generated by $\BJ^0$ and an element $\w$
whose reduced norm has non-zero valua\-tion (see \S\ref{prelimmax3}).
Since $\epsilon(\tau)=-1$ and $\BJ_\vartheta$~is stable by $\tau$, 
we have $\tau(\w)\in\w^{-1}\BJ^0$.
And since $\bx$ is trivial on $\BJ^0$, we deduce that $\bx\circ\tau=\bx^{-1}$.

Now write~$a$ for the order of~$\tau$,
which we assume to be prime to $p$.
Then the identity $\bk^{{\tau}\vee} 
\simeq \bk\bx$ applied~$2a$ times shows that 
$\bk\bx^{2a}\simeq\bk$ so that~$\bx^{2a}=1$. But since~$\bx$ has~$p$-power 
order, and~$2a$~is prime to~$p$, we deduce that~$\bx$ is trivial. 
\end{proof}

\begin{rema}
\label{mammouth}
In the case when $\epsilon(\tau)=-1$
and the order of $\tau$ is finite and prime to $p$,
we~even proved that any $\bk$ such that $\det(\bk)$ has $p$-power order 
satisfies $\bk^{\tau\vee} \simeq \bk$.
We also have
\begin{equation}
\label{atrides}
\BJ_\vartheta \cap \GL_r(\De)^\tau = \BJ^0 \cap \GL_r(\De)^\tau.
\end{equation}
Indeed,
if $x\in\BJ_\vartheta$ is $\tau$-invariant,
its valuation has to be equal to its opposite: it is thus $0$.
\end{rema}

\subsection{}

Now let us go back to the situation of \S\ref{mst}
with the group $G=\GL_n(D)$ equipped with~the involution $\s$.
Note that $\epsilon(\s)=-1$
(in the notation of the proof of Lemma \ref{kappasigmaselfdual})
and the order~of $\s$ is prime to $p$,
so Lemma \ref{kappasigmaselfdual} and Remark \ref{mammouth} apply. 
We will need the following lemma,
which is \cite{VSautodual} Lemme 3.30.

\begin{lemm}
\label{kappatau}
Let $\bk$ be a representation of $\BJ_\t$
exten\-ding $\n$ such that $\bk^{\s\vee} \simeq \bk$.
\begin{enumerate} 
\item 
There is a unique character $\chi$ of $\BJ_\t\cap G^\s=\BJ^0\cap G^\s$
trivial on $\BJ^1\cap G^\s$ such that
\begin{equation*}
\Hom_{\BJ_\t\cap G^\s}(\bk,\chi) \neq \{0\}
\end{equation*}
and this $\chi$ is quadratic (that is, $\chi^2=1$).  
\item
Let $\bs$ be an irreducible representation of $\BJ_{\t}$ trivial on $\BJ^1$.
The canonical linear map:
\begin{equation*}
\Hom_{\BJ^{1}\cap\G^{\s}}(\n,\CC) \otimes
\Hom_{\BJ_{\t}\cap\G^{\s}}(\bs,\chi)
\to \Hom_{\BJ_{\t}\cap\G^{\s}}(\bk\otimes\bs,\CC) 
\end{equation*}
is an  isomorphism. 
\end{enumerate}
\end{lemm}

\section{Proof of Theorem \ref{MAINTHM1}}
\label{Sec6}

Recall that $F$ has odd residue characteristic, 
and assume that the Jacquet--Langlands transfer of any cuspidal representation
of $\G$ distinguished by $\G^\s$ is~non-cuspidal
(which is known to be true when $F$ has characteristic $0$,
thanks to Verma \cite{Verma} Theorem 1.2,
and also Theorem \ref{Vermacarqcq}.)

Let $\pi$ be a cuspidal irreducible representation of $G$ with non-cuspidal
transfer to $\GL_{2n}(F)$, as in Section \ref{Sec2}.
By Proposition \ref{camille}, there are
a maximal simple stratum $[\aa,\b]$ in $A$
and a maximal simple character $\t \in \Cc(\aa,\b)$
such that $\aa^*=\aa$, $\b^*=\b$
and $\t^{-1}\circ\s=\t$.
We use the notation of Section \ref{relaxe}.
In particular,
we have groups $\BJ_\t$, $\BJ^0$ and $\BJ^1$.

\subsection{}

Identify $\BJ^0 / \BJ^1$ with $\GL_m(\ll)$
thanks to the group~iso\-morphism \eqref{JJ1UU1GLmax}.
Through this identi\-fication,
and thanks to Proposition~\ref{prop:res*},
the involution $\s$ on $\BJ^0 / \BJ^1$ identifies 
with the unitary~in\-volution 
\begin{equation}
\label{invres}
x \mapsto {}^{\intercal} \overline{x}^{-1}
\end{equation}
on $\GL_m(\ll)$,
where
{$\intercal$ denotes transposition with respect to the antidiagonal and}
$x \mapsto \overline{x}$ is the~ac\-tion
of the non-trivial element of $\Gal(\ll/\ll_0)$ componentwise.~It 
follows that $(\BJ^0 \cap G^\s) / (\BJ^1 \cap G^\s)$
identifies with the unitary group ${\rm U}_m(\ll/\ll_0)$.
The~following~lem\-ma~will be useful.
Note that $m$ is odd since it divides $n$,
which is odd by \S\ref{saccard}.

\begin{lemm}
\label{distBC}
Let $\varrho$ be a cuspidal irreducible representation of $\GL_m(\ll)$. 
The following assertions are equivalent.
\begin{enumerate}
\item 
The representation $\varrho$ is $\Gal(\ll/\ll_0)$-invariant.
\item
The representation $\varrho$ is distinguished by ${\rm U}_m(\ll/\ll_0)$. 
\end{enumerate}
Moreover,
there exist $\Gal(\ll/\ll_0)$-invariant cuspidal irreducible representations
of $\GL_m(\ll)$. 
\end{lemm}

\begin{proof}
For the equivalence between (1) and (2), 
see for instance \cite{PrasadCMAT99} Theorem 2
or \cite{Gow}~Theo\-rem 2.4. 
For the last assertion, 
see for instance \cite{VSANT19} Lemma 2.3. 
\end{proof}

Let $\bk$ be a representation of $\BJ_\t$
extending $\n$ such that $\bk^{\s\vee} \simeq \bk$
(whose existence is {given} by Lemma \ref{kappasigmaselfdual})
and $\chi$ be the quadratic character~of the group 
$\BJ_\t\cap G^\s$
(which is equal to $\BJ^0 \cap G^\s$ by \eqref{atrides})
given by Lemma \ref{kappatau}.

\begin{prop}
\label{chitrivial}
The character $\chi$ is trivial.
\end{prop}

\begin{proof}
Assume this is not the case.
Then $\chi$,
considered as a character of ${\rm U}_m(\ll/\ll_0)$, 
is trivial~on unipotent~ele\-ments 
because these elements have $p$-power order and $p\neq2$.
Thus~$\chi$~is trivial~on the subgroup generated by all 
transvections.
By \cite{Grove} Theorem 11.15,
this~sub\-group is ${\rm SU}_m(\ll/\ll_0)$.
Thus $\chi = \a \circ \det$ for some quadratic character $\a$ of $\ll^1$,
where $\det$ is~the~de\-terminant on $\GL_m(\ll)$ and $\ll^1$ is the subgroup
of $\ll^\times$ made of elements of $\ll/\ll_0$-norm $1$.
Let $\b$ extend $\a$ to $\ll^\times$,
and let $\varkappa$ be the character of $\BJ^0$ inflated from 
$\b \circ \det$. 
It extends $\chi$.

Since $m$ is odd,
there is a cuspidal irreducible representation 
$\varrho$ of $\GL_m(\ll)$ which is~in\-va\-riant~by $\Gal(\ll/\ll_0)$ 
(equivalently, which is distinguished by ${\rm U}_m(\ll/\ll_0)$),
thanks to Lemma \ref{distBC}.

Let $\varrho'$ be the cuspidal representation $\varrho\varkappa$.
Let us prove that it is not $\Gal(\ll/\ll_0)$-invariant.~Let~$\g$
denote the generator of $\Gal(\ll/\ll_0)$.
If $\varrho'$ were $\Gal(\ll/\ll_0)$-invariant,
$\varrho\varkappa^\g$ would be isomorphic to~$\varrho\varkappa$.
Comparing the central characters,
one would get $(\b^{\g}\b^{-1})^m=1$,
that is,
$\a(x^{\g}x^{-1})^m=1$ for all $x\in\ll^\times$,
or equivalently $\a^m=1$.
But $\a$ is quadratic and $m$ is odd:
contradiction.

Let $\bs'$ be a representation of $\BJ'=\E^\times\BJ^0$ 
whose restriction to $\BJ^0$ is the inflation of $\varrho'$.
Since $\varrho'$ is not $\Gal(\ll/\ll_0)$-invariant,
the normalizer of $\bk\otimes\bs'$ in $\BJ_\t$ is $\BJ'$
(which has index $2$ in $\BJ_\t$).

On the one hand,
the representation $\pi'$ compactly induced by $(\BJ',\bk\otimes\bs')$
is irreducible and cuspidal,
and its Jacquet--Langlands transfer to $\GL_{2n}(F)$ is cuspidal
by Theorem \ref{BHpardeg}.

On the other hand, the map
\begin{equation*}
\Hom_{\BJ^{1}\cap\G^\s}(\n,\CC) \otimes
\Hom_{\BJ'\cap\G^\s}(\bs',\chi)
\to \Hom_{\BJ'\cap\G^\s}(\bk\otimes\bs',\CC) 
\end{equation*}
is an isomorphism
(by Lemma \ref{kappatau})
and the space $\Hom_{\BJ'\cap\G^\s}(\bs',\chi)$ is non-zero by construction.
This implies that 
$\bk\otimes\bs'$ is $\BJ'\cap\G^\s$-distinguished.
Thus $\pi'$ is distinguished,
which contradicts
the assumption of Theorem \ref{MAINTHM1}.
Thus $\chi$ is trivial.
\end{proof}

\subsection{}

According to \S\ref{mst},
our cuspidal representation $\pi$ of $G$
contains a representation of the form $(\BJ,\bk\otimes\bs)$,
where
\begin{itemize}
\item 
the group $\BJ$ satisfies $\E^\times\BJ^0 \subseteq \BJ \subseteq \BJ_\t$,
\item 
the representation $\bk$ is the restriction to $\BJ$ of a representation 
of $\BJ_\t$ extending $\n$, 
\item
the representation $\bs$ of $\BJ$ is trivial on $\BJ^1$
and its restriction to $\BJ^0$ is the inflation of
a~cus\-pi\-dal~representation $\varrho$ of $\BJ^0/\BJ^1 \simeq \GL_m(\ll)$
whose normalizer in $\BJ_\t$ is equal to $\BJ$.
\end{itemize}

Thanks to Lemma \ref{kappasigmaselfdual} and Proposition \ref{chitrivial},
we may and will assume that $\bk^{\s\vee}\simeq\bk$ and $\bk$ is
distinguished by $\BJ_\t\cap\G^\s$.

Thanks to Theorem \ref{BHpardeg}, 
the fact that the Jacquet--Langlands transfer of $\pi$
is non-cuspidal~im\-plies that $\BJ=\BJ_\t$.
By \S\ref{mst},
this implies that $\varrho$ is $\Gal(\ll/\ll_0)$-invariant.
It follows from Lemma~\ref{distBC} that $\varrho$ is distinguished
by ${\rm U}_m(\ll/\ll_0)$,
thus $\bs$ is distinguished by $\BJ^0 \cap G^\s = \BJ_\t \cap G^\s$.
By Lemma \ref{kappatau},
the representation $\bk\otimes\bs$ is distinguished by $\BJ_\t \cap G^\s$.
It follows from Mackey's formula
\begin{equation*}
  \Hom_{\G^\s} (\pi,\CC) \simeq \prod\limits_{g}
  \Hom_{\BJ_\t \cap g\G^{\s}g^{-1}} (\bk\otimes\bs,\CC)
\end{equation*}
(where $g$ ranges over a set of representatives of
$(\BJ_\t,G^\s)$-double cosets of $G$)
that $\pi$ is distin\-gui\-shed by $\G^\s$.
This finishes the proof of Theorem \ref{MAINTHM1}.

\section{The depth $0$ case}
\label{thecircle}

In this section, 
we discuss in more detail the case of representations of depth $0$.
We assume throughout the section that $n$ is odd.
We no longer assume that $F$ has odd residue characteristic.

\subsection{}
\label{requin}

Let $\pi$ be a cuspidal irreducible representation of $G$ with non-cuspidal
transfer to $\GL_{2n}(F)$ as in Section \ref{Sec6}.
Assume moreover that $\pi$ has depth $0$.
In that case,
we are in the situation~des\-cribed by Remark \ref{nullstratum}. 
In this situation,
we have $m=n$ and~$\ll$~is~the residue field~of~$D$~(thus~$\ll_0$ is that~of $F$).
We have $\BJ=D^\times\GL_n(\Oo_D)$ and $\BJ^0=\GL_n(\Oo_D)$,
and one can choose for $\bk$~the~tri\-vial~character of~$\BJ$. 
The representation $\pi$ is compactly induced 
from an irreducible~represen\-ta\-tion~$\bs$ of $\BJ$ 
whose restriction to $\BJ^0$ is the~in\-flation of a
$\Gal(\kk_D/\kk_F)$-invariant, cuspidal represen\-tation $\varrho$ of 
$\GL_n(\kk_D)$.

\begin{rema}
\label{buffetfroid1}
In \cite{Verma} Proposition 5.1, 
the inducing subgroup should 
be $D^\times\GL_n(\Oo_D)$ and not $F^\times\GL_n(\Oo_D)$.
Inducing from the latter subgroup gives a representation
which is not irredu\-cible. 
The same comment applies to \cite{Verma} Remark 5.2(1).
See also Remark \ref{buffetfroid2} below.
\end{rema}

\subsection{}

Let us now consider the map $\bc_{D/F}$ defined in \S\ref{par18now}.
This is a bijection from~cus\-pidal~re\-pre\-sentations of $\GL_n(F)$
to those cuspidal representations of $\GL_n(D)$ 
which are distinguished~by~the subgroup $\Sp_n(D)$.
In this paragraph,
given a cuspidal representation $\tau$ of depth $0$ of $\GL_n(F)$,~we
describe explicitly the cuspidal representation $\pi=\bc_{D/F}(\tau)$,
that is,
the unique cuspidal~represen\-tation of $G$ whose Jacquet--Langlands transfer
to $\GL_{2n}(F)$ is $\St_2(\tau)$.

On the one hand,
it follows from \cite{SZlevel01} Proposition 3.2 that
the representation $\pi$ has depth~$0$.~It
can thus be described~as in \S\ref{requin},
that is,
it is compactly induced from a representa\-tion~$\bs$ of the group 
$\BJ=D^\times\GL_n(\Oo_D)$
whose restriction to $\BJ^0$ is the~in\-flation of $\varrho$. 

On the other hand,
the representation $\tau$ can be described in a similar way:
it is~com\-pactly~in\-duced from a representation
of $F^\times\GL_n(\Oo_F)$ whose restriction to $\GL_n(\Oo_F)$ 
is the~in\-flation of a cuspidal representation $\varrho_0$ of 
$\GL_n(\kk_F)$.
(See for instance \cite{BushnellPIM328} 1.2.) 

\cite{SZlevel01} Theorem 4.1 provides a simple and natural relation between
$\varrho$ and $\varrho_0$:
the representation $\varrho$ is the base change
(that is, the Shintani lift)
of $\varrho_0$.
(This relation was pointed out in \cite{Verma}~Re\-mark 5.2(1)
without reference to \cite{SZlevel01}.)

The knowledge of $\varrho$ does not quite determine the
representation $\pi$.
In order to completely~de\-termine it, 
fix a uniformizer $\w_F$~of $F$
and a uniformizer $\w=\w_D$ of $D$ such that $\w^2=\w_F$.~As
the group $\BJ$ is generated by $\BJ^0$ and $\w$, 
it remains to compute the operator $A=\bs(\w)$,
which~in\-tertwines~$\varrho$ with $\varrho^\g$, 
where $\g$ is the non-trivial element of $\Gal(\kk_D/\kk_F)$,
that is, one has
\begin{equation*}
A \circ \varrho(x) = \varrho(x^\g) \circ A,
\quad
x\in\GL_n(\kk_D).
\end{equation*}
The space of intertwining operators between $\varrho$ and $\varrho^\g$
has dimension $1$.
To go further,
we have~to identify this space with $\CC$ in a natural way.

Fix a non-trivial character $\psi_0$ of $\kk_F$,
and let $\psi$ be the character of $\kk_D$ obtained by composing
$\psi_0$ with the trace of ${\kk_D/\kk_F}$.
Let $U$ denote the subgroup of $\GL_n(\kk_D)$ made of all
unipotent~up\-per triangular matrices,
and consider $\psi$ as the character
$u \mapsto \psi(u_{1,2}+u_{2,3}+\dots+u_{n-1,n})$ of $U$. 
It is well-known that,
if $V$ is the underlying vector space of $\varrho$, then
\begin{equation*}
\varrho^\psi = \{ v\in V \ |\ \varrho(u)(v) = \psi(u)v, \ u \in U \}
\end{equation*}
has dimension $1$.
Since the character $\psi$ is $\Gal(\kk_D/\kk_F)$-invariant,
this $1$-dimensional space is~sta\-ble by $A$.
There is thus a non-zero scalar $\a\in\CC^\times$ such that
$A(v)=\a v$ for all $v\in \varrho^\psi$,~and~$A$~is
uniquely~deter\-mined by $\a$.
Let $\omega_0$ denote the central character of $\tau$.
(Note that the representation
$\tau$ is entirely determined by $\varrho_0$ and $\omega_0$.)

\begin{prop}
\label{babil}
One has $\a = \omega_0(-\w_F)$.
\end{prop}

We now have completely determined $\pi$ from the knowledge of $\tau$.
The proof of this proposition,
based on \cite{SZlevel02} and \cite{BHlevel0},
will be done in the next paragraph.

\begin{rema}
The proposition thus implies that the result does not depend on the choice 
of a $\w\in\D$ such that $\w^2=\w_F$.
Replacing $\w$ by $-\w$ should thus lead to the same result,
that is,
one should have $\bs(-\w)=\bs(\w)$,
or equivalently,
the central character of $\pi$ should be trivial at $-1$.
This is the case indeed,
since $\pi$ is distinguished by $\Sp_n(D)$,
which contains $-1$.
\end{rema}

\begin{rema}
\label{buffetfroid2}
This paragraph corrects the description made in \cite{Verma} Remark 5.2(1),
which~is incorrect due to the error pointed out in Remark \ref{buffetfroid1}.
Note that \cite{Verma} Remark 5.2(2) is correct:~it follows from 
\cite{SZlevel01} Theorem 4.1 or \cite{BHJL3} Theorem 6.1.
\end{rema}

\subsection{}

We now proceed to the proof of Proposition \ref{babil},
which is essentially an exercise of~trans\-la\-tion into the language
of \cite{SZlevel02} and \cite{BHlevel0}.
Fix a separable closure $\overline{F}$ of $F$.

A \textit{tame admissible pair}~is~a pair $(K/F,\xi)$ made of an 
unramified finite extension $K$ of $F$~con\-tained in $\overline{F}$
together with~a~ta\-me\-ly~rami\-fied~character $\xi:K^\times\to\CC^\times$
all of whose
$\Gal(K/F)$-con\-jugate $\xi^\g$, $\g\in\Gal(K/F)$,~are
pair\-wise dis\-tinct.~The \textit{degree} of such a pair is the degree of
$K$ over $F$.

Given any integer $m\>1$ and any inner form $H$ of $\GL_m(F)$,
Silberger--Zink \cite{SZlevel02} have~defi\-ned a bijection ${\it\Pi}^H$ between:
\begin{enumerate}
\item 
the set of Galois conjugacy classes of tame admissible pairs of degree 
dividing $m$, 
\item
the set of isomorphism classes of discrete series representations of depth
$0$ of $H$.
\end{enumerate}
They have also described (in \cite{SZlevel02} Theorem 3)
the behavior of this parametrization of the discrete series of inner forms
of $\GL_m(F)$ 
with respect to the
Jac\-quet--Langlands correspondence:
if $H$~is isomorphic to
$\GL_r(\De)$ for some divisor $r$ of $m$ and some central division 
$F$-algebra~$\De$~of reduced degree $m/r$,
and if $(K/F,\xi)$ is a~ta\-me admissible~pair of degree $f$ dividing $m$,
then the Jac\-quet--Lang\-lands transfer of
${\it\Pi}^{H} (K/F,\xi)$ to $\GL_m(F)$ is equal to
\begin{equation}
\label{rectifierSZ}
{\it\Pi}^{\GL_m(F)} \left( K/F,\xi \mu_K^{m-r+(f,r)-f} \right)
\end{equation}
where $\mu_K$ is the unique unramified character of $K^\times$ of order $2$ 
and $(a,b)$ denotes the greatest~com\-mon divisor of two integers $a,b\>1$.
(Silberger--Zink state their result by using the multiplicative group of a
central division $F$-algebra of reduced degree $m$ as an inner form of 
reference,
but it~is more convenient for us to use $\GL_m(F)$ as the inner form of 
reference.)

Let us start with our cuspidal representation $\tau$ of depth $0$ of 
$\GL_n(F)$.
Let $(K/F,\xi)$ be a~ta\-me admissible pair associated with it
by the bijection ${\it\Pi}^{\GL_n(F)}$.
It follows from \cite{BHlevel0} 5.1 that
this pair has degree $n$,
and that the central character $\omega_0$ of $\tau$ is equal to the
restriction of $\xi$ to $F^\times$. 

Now form the discrete series representation $\St_2(\tau)$ of $\GL_{2n}(F)$.
By \cite{BHlevel0} 5.2,
the tame admissible pair associated with it 
by the bijection ${\it\Pi}^{\GL_{2n}(F)}$
is of the form $(K/F,\xi')$
where $\xi'$ coincides with~$\xi$ on the units of $\Oo_K$.
By \cite{BHlevel0} 6.4, one has
\begin{equation*}
\xi'(\w_F) = -\omega_0(\w_F) = -\xi(\w_F).
\end{equation*}
One thus has $\xi'=\xi\mu_K$. 
We now claim that the representation $\pi$ is parametrized,
through~the~bi\-jection ${\it\Pi}^{G}$, 
by the tame admissible pair $(K/F,\xi)$. 
Indeed,~by \eqref{rectifierSZ},
and since $n$ is odd and $\mu_K$ is quadratic,
we have
\begin{equation*}
{}^{\rm JL} {\it\Pi}^{G}(K/F,\xi)
= {\it\Pi}^{\GL_{2n}(F)} (K/F,\xi \mu_K^{2n-n+(n,n)-n})
= {\it\Pi}^{\GL_{2n}(F)} (K/F,\xi \mu_K^{\phantom{n}}) 
\end{equation*}
which is equal to $\St_2(\tau)$.
It now follows from \cite{SZlevel02} (8), p.~196,
that the scalar $\a$ by which
$A$ acts on the line $\varrho^\psi$ is equal~to $\xi(-\w_F)=\omega_0(-\w_F)$
as expected.

\section{More preliminaries}

The remainder of the article is devoted to the proof of Theorem \ref{MAINTHM2}.
In this section,
we give more preliminaries,
in addition to those of Section \ref{apfelstrudel}.

\subsection{}

Fix an integer $n\>1$ and set $\G=\GL_n(\CC)$.
By \textit{representation} of $\G$ we mean a
smooth~ad\-missible~Fré\-chet repre\-sen\-tation of moderate growth. 

We denote by $|\cdot|_\CC$ the normalized absolute value of $\CC$, 
that is, the square of the usual~modu\-lus~of $\CC$.

Let $P=MN$ be a standard parabolic subgroup of $G$
together with its standard Levi decompo\-si\-tion.
As in the non-Archimedean case (\S\ref{defip}),
we denote by $\ip^G_P(\s)$ the representation
of~$G$~ob\-tained from a representation $\s$ of $M$
by (normalized) parabolic induction along $P$.

Fix a non-trivial {unitary} additive character $\psi$ of $\CC$.
As in the non-Archimedean case (\S\ref{defgeneric}),~it defines
a character (still denoted by $\psi$) of the subgroup $U$ of
upper~trian\-gular unipotent matrices of $\G$.
An irreducible representation $\pi$ of $G$ is said to be generic if
$\Hom_{U}(\pi,\psi)$ is non-zero.

Let $\nu$ denote the character 
``norma\-lized absolute value of the determinant'' of $G$.

\subsection{}
\label{DuchesseduMaine}

Let $F$ be either $\CC$ or a non-Archimedean locally compact field
and let $\De$ be a finite-dimen\-sional central division $F$-algebra.
(If $F=\CC$, we thus have $\De=\CC$.)
Fix an integer $n\>1$~and~set $\G=\GL_n(\De)$.
It is a locally compact group. 

Let $K$ be the compact subgroup $\GL_n(\Oo_{\De})$ 
if $F$ is non-Archi\-me\-dean,
and the compact unitary group ${\rm U}_n(\CC/\RR)$ 
if $F=\CC$.
In either case, 
$K$ is a maximal compact subgroup of $G$. 

Let $(n_1,\dots,n_r)$ be a composition of $n$,
and let $P=MN$ be the standard parabolic subgroup~of $G$
associated with it.
One has the Iwasawa~decom\-po\-sition $G=PK$.
For $i=1,\dots,r$,~let $\pi_i$~be~a
representation of $\GL_{n_i}(\De)$.
Given any $s=(s_1,\dots,s_r) \in \CC^r$,~restric\-tion
of functions from $G$ to $K$ induces an isomorphism 
\begin{equation}
\label{ilan}
\pi_1\nu^{s_1} \tdt \pi_r\nu^{s_r}
= \Ind^{G}_{P}(\pi_1\nu^{s_1}\odo\pi_r\nu^{s_r})
\to \Ind^{K}_{K\cap P}(\pi_1\odo\pi_r)
\end{equation}
of representations of $K$,
which we denote by $r_s$.
A family of functions
\begin{equation*}
\h_s \in \pi_1\nu^{s_1} \tdt \pi_r\nu^{s_r},
\quad
s\in\CC^r,
\end{equation*} 
is called a \emph{flat section} if $r_s(\h_s)$ is independent of $s$,
that is,
if there exists a function $f$ in~the~right hand side of \eqref{ilan}
such that $\h_s^{\phantom{1}}=r_s^{-1}(f)$ for all $s\in\CC^r$.
The reader may refer to \cite{WPL} IV.1~for mo\-re detail.

Let ${\sf X}(M)$ be the free $\ZZ$-module of algebraic characters of $M$,
set $\aa_P^{\phantom{*}} = {\sf X}(M)\otimes_{\ZZ}\RR$
and let $\aa^\ast_P$ be its dual. 
The \emph{Harish-Chandra map} $H_P : \G \to \aa_P^\ast$ is defined by
\begin{equation*}
e^{\langle \chi, H_P(muk)\rangle} = |\chi(m)|_F
\end{equation*}
for $m\in M$, $u\in N$, $k\in K$ and $\chi\in{\sf X}(M)$. 

\subsection{}
\label{pole}

Now assume that $\De=F$,
and let $\pi$, $\pi'$ be generic irreducible representations of $\GL_n(F)$. 
Associated with them in \cite{JPSS} if $F$ is non-Archimedean,
and \cite{Jrsarch} if $F=\CC$,
there are the Rankin--Selberg local factors 
\begin{equation*}
L(s, \pi, \pi'), 
\quad
\e(s, \pi, \pi', \psi),
\quad
\g(s, \pi, \pi', \psi),
\end{equation*}
related by the identity
\begin{equation*}
\g(s, \pi, \pi', \psi) = \e(s, \pi, \pi', \psi)
\frac{L(1-s, \pi^\vee, \pi'^\vee)}{L(s, \pi, \pi')}.
\end{equation*}
Note that,
if $F$ is non-Archimedean and if $\pi$, $\pi'$ are cuspidal,
the local $L$-factor $L(s, \pi, \pi')$ {is~no\-where vanishing},
and it has~a pole at $s_0\in\CC$
if and only if $\pi'$ is isomorphic to $\pi^\vee \nu^{-s_0}$
(see \cite{JPSS} Proposition 8.1).

\subsection{}

Let $k$ be a totally imaginary number field.
Fix a non-trivial additive character of $\AA=\AA_k$ which is trivial on $k$.
For any place $v$ of $k$,
let $\psi_v$ denote the local component of $\psi$ at $v$.
It is a non-trivial unitary additive character of $k_v$.

Let $\Pi$, $\Pi'$ be cuspidal automorphic representations of $\GL_n(\AA)$. 
Given an Archimedean place $v$,
their local components $\Pi_v^{\phantom{'}}$, $\Pi'_v$ are 
Harish-Chandra modules,
having Casselman--Wallach completions
denoted $\overline{\Pi}{}^\infty_v$, $\overline{\Pi}{}'^\infty_v$,
respectively.
These completions are generic irreducible representations~of $\GL_n(k_v)$
(see \cite{RRG2} for details)
and we define
\begin{equation*}
L(s, \Pi_v^{\phantom{'}}, \Pi'_v) = L(s, \overline{\Pi}{}^\infty_v, 
\overline{\Pi}{}'^\infty_v)
\end{equation*}
and similarly for $\e$-factors and $\g$-factors.

\begin{prop}[\cite{JSeuler1} Theorem 5.3]\label{global-RS}
Let $S$ be a finite set of places of $k$ containing~all Ar\-chi\-medean 
places and all places at which at least one of $\Pi$, $\Pi'$ or $\psi$ is ramified. 
The product
\begin{equation*}
L^S(s, \Pi, \Pi')=\prod_{v\not\in S}L(s, \Pi_v^{\phantom{'}}, \Pi'_v)
\end{equation*}
converges absolutely for $\Re(s)$
sufficiently large and extends to a meromorphic function on
$\CC$~sa\-tis\-fying the functional equation
\[
  L^S(s, \Pi, \Pi') = \prod_{v\in S}
  \g(s, \Pi_v^{\phantom{'}}, \Pi'_v, \psi_v^{\phantom{'}})
  \cdot L^S(1-s, \Pi^\vee, \Pi'^\vee).
\]
\end{prop}

\section{Geometry of the symmetric space}
\label{Kelh}

In this section,
$F$ is a field of characteristic different from $2$,
and $A$ is either $\Mat_2(F)$ or a~non-split~qua\-ternion $F$-algebra. 
For any integer $n\>1$,
we write $G_n=\GL_n(A)$ and $H_n=\Sp_n(A)$.  

\subsection{}
\label{sec:parab orb}

Fix an integer $n\>1$ and write $G=G_n$ and $H=H_n$ for simplicity. 
The symmetric~space corresponding to the symmetric pair
$(G, H)$ is
\[
  X = \{x\in G \mid x\s(x) = 1\}.
\]
It is endowed with a transitive action of $G$ defined by
$g\cdot x=gx\s(g)^{-1}$ for $g\in G$, $x\in X$. 
Since $H$ is the stabilizer of the identity matrix in $X$, 
we can identify $X$ with $G/H$ via $g\mapsto g\s(g)^{-1}$.

\subsection{}

Fix integers $m,\ad\>1$ such that $n=m\ad$ and 
let $P=MN$ be the standard parabolic~sub\-group of $G$
associated with
{the composi\-tion~$(\ad,\dots,\ad)$.} 
The subgroups $P$, $M$ and $N$ are $\s$-stable.

Let $\M_0$ be the group of diagonal matrices in $\G$.
We write $W_M$ for the {Weyl group} of~$\M$,
that is,
the quotient~of the $M$-normalizer of $M_0$ by $M_0$. 
It is isomorphic to
$\frS_{\ad}\times\cdots\times\frS_{\ad}$,
where $\frS_{\ad}$~is the symmetric group of order $\ad !$. 
We also write $W$ for the Weyl group of $G$.

Any $(W_M,W_M)$-double coset in $W$ contains a unique 
element $w$ with mini\-mal length in both $W_Mw$ and $wW_M$.
This defines a set of representatives
of $(W_M,W_M)$-double~co\-sets in $W$,
which we denote by ${}_MW_M$.

Let $W(M)$ be the set of elements of ${}_MW_M$ normalizing $M$.
It naturally identifies with $\frS_m$.

\subsection{}
\label{admorbits}

We describe the $P$-orbits~of~$X$,
or equivalently the $(P,H)$-double~co\-sets of
$G$.~For~more~de\-tails,
we refer the reader to \cite{MOS} \S3 in the case when~$A$ is split, 
and \cite{ShVe} \S3 in the case when~$A$~is non-split
(see also \cite{Offen-Sp1,OJNT}).
Given an $x\in X$,
there is a unique $w\in {}_MW_M$ such that~the
$P$-orbit $P\cdot x$ is~con\-tained in $PwP$,
since $P$ and $M$ are $\s$-stable.~Hence
we obtain a map
\begin{equation}
\iota_M : P\bsl X\to {}_MW_M.
\end{equation}
A $P$-orbit $P\cdot x$ is said to be $M$-\emph{admissible} if
$w=\iota_M(P\cdot x)$ normalizes $M$,
that is,
if $w\in\W(M)$.

Let $P\cdot x$ be an $M$-admissible $P$-orbit of $X$ and
consider $w=\iota_M(P\cdot x) \in W(M)$ as~an~ele\-ment of $\frS_m$.
Let $w_m$ be the element of maximal length in $\frS_m$.
The permutation $\tau=ww_{m} \in \frS_m$~sa\-tis\-fies~$\tau^2=1$. 
The map
\begin{equation}
\label{kasha}
P\cdot x\mapsto\tau
\end{equation}
is a bijection from the set of
$M$-admissible $P$-orbits of $X$
to the set $\{\tau \in \frS_m\ |\ \tau^2=1\}$.

To~des\-cri\-be~the inverse of the bijection \eqref{kasha}, 
let $w\mapsto[w]_{\ad}$ denote the bijection from $\frS_m$
to the~group~of 
permutation ma\-tri\-ces of $\GL_m(F)$
composed with the natural embedding of $\GL_m(F)$ in
$\GL_m(\Mat_{\ad}(A))=G$.
Given~a~per\-mu\-tation $\tau \in \frS_m$ such that $\tau^2=1$,
the element $x=[\tau w_{m}]_{\ad}$ is in $X$
and the image of its $P$-orbit by \eqref{kasha} is equal to $\tau$.
For such an $x$,
the stabilizer $M_x$ of~$x$ in $M = \G_{\ad} \tdt \G_{\ad} $ is equal to
\begin{equation}
\label{descriptionMx}
  M_x = \left\{ (g_1, \dots, g_m) \in \G_{\ad} \tdt \G_{\ad} \ | \
  g_{\tau(i)} = \s(g_i) \text{ for all } i=1,\dots,m \right\}.
\end{equation}
In particular,
one has $g_i\in \H_{\ad}$ for all $i$ fixed by $\tau$.
Thus $M_x$ is the group of $F$-rational points~of a~con\-nec\-ted reductive
algebraic group~de\-fined over $F$.

\begin{exem}
\label{openorbit} 
\begin{enumerate}
\item 
In particular,
the choice $\tau=1 \in \frS_m$ gives the representative 
\begin{equation}
\label{duchessedeberry}
x = [w_m]_{\ad} =
\begin{pmatrix} & & \id_{\ad} \\ & \idots & \\ \id_{\ad} & & \end{pmatrix}
\in X
\end{equation}
where $\id_{\ad}$ is the identity matrix in {$\Mat_{\ad}(A)$}.
The $M$-admissible orbit $P\cdot x$ is the open $P$-orbit of~$X$
and we have $M_x = H_{\ad}\times\cdots\times H_{\ad}$ 
where $H_{\ad}$ occurs $m$ times.
\item
Let us give an explicit representative $\n\in\G$ such that
$\n\s(\n)^{-1}=x$.
First,
fix non-zero~ele\-ments $\a,\b\in A^\times$ such that 
$\Nrd_{A/F}(\a)=-1/2$ and $\Nrd_{A/F}(\b)=1/2$,
which is possible since the reduced norm $\Nrd_{A/F}$ is surjective.
If $m$ is even, write $m=2k$ and set
\begin{equation*}
\label{eq:representative1} 
\n =
\begin{pmatrix} \a\cdot\id_{k\ad} & -\a\cdot {[w_k]_t} \\
  \b\cdot {[w_k]_t}& \b\cdot\id_{k\ad} \end{pmatrix}.
\end{equation*}
If $m$ is odd,
set $m=2k+1$ and
\begin{equation*}
\label{eq:representative2}
\n = 
\begin{pmatrix} \a\cdot\id_{k\ad} & & -\a\cdot {[w_k]_t} \\ & \id_{k\ad} & \\
  \b\cdot  {[w_k]_t} & & \b\cdot\id_{k\ad}.
\end{pmatrix}. 
\end{equation*}
Such an $\n$ satisfies $\n\s(\n)^{-1}=x$
and $P\n H$ is the open $(P,H)$-double coset of $G$. 
\end{enumerate}
\end{exem}

\subsection{}
\label{CNA}

In this paragraph, 
$F$ is either $\CC$ or a non-Archimedean local field of characteristic $0$.
Let us recall some known properties of symplectic periods which we will use later.
The next theorem is called the \emph{multiplicity one property} of symplectic
periods.

\begin{theo}
\label{thm:mult1}
For any irreducible representation $\pi$ of $G$,
the vector space $\Hom_{H}(\pi, \CC)$ has dimension at most $1$.
\end{theo}

\begin{proof}
In the non-Archimedean case,
this is proven by Heumos and Rallis \cite{Heumos-Rallis} Theorem~2.4.2
when $A=\Mat_2(F)$ and by Verma \cite{Verma} Theorem 3.8
when $A$ is non-split. 
Their argument is based on Prasad's idea \cite{Ptril} and the key is to 
find a certain anti-automorphism of $G$ acting~tri\-vial\-ly on 
$H$-bi-invariant distributions on $G$.
Here we are in the simplest 
case where the $H$-double~cosets are invariant under the 
anti-involution at hand. 
Due to \cite{AG}, these arguments also apply~to the Archi\-me\-dean case where
$A=\Mat_2(\CC)$.
\end{proof}

\subsection{}
\label{GastonDeChanlay}

In this paragraph, 
$F$ is still either $\CC$ or a non-Archimedean locally compact
field of~cha\-rac\-teristic $0$, 
and we assume moreover that $A=\Mat_2(F)$.
{We will iden\-ti\-fy $G$ with $\GL_{2n}(F)$.}

Given any integer $k\in\{0,\dots,n\}$,
consider the subgroup ${}_{k}H$ of $\G$ defined by
\[
{}_{k}H = \left\{ 
\begin{pmatrix} u & a \\ 0 & h \end{pmatrix} 
\ \middle|\ u\in U_{2r},\ h\in\Sp_{2k}(F),\ a\in\Mat_{2r,2k}(F) \right\}
\]
where $r=n-k$
and $U_{2r}$ is the group of upper triangular unipotent matrices in 
$\GL_{2r}(F)$.
Note that ${}_{n}H = H = \Sp_{2n}(F)$. 
Define a character ${}_{k}\Psi$ of ${}_{k}H$ by
\[
  {}_{k}\Psi \begin{pmatrix} u & a \\ 0 & h \end{pmatrix}
  = \psi(u_{1,2}+u_{2,3}+\cdots+u_{2r-1,2r}).
\]
We say that an irreducible representation $\pi$ of $\G$
has a \emph{Klyachko model} if there exists an integer
$k\in\{0,\dots,n\}$ 
such that
$\Hom_{{}_{k}H}(\pi,{}_{k}\Psi)$ is non-zero. 

The following theorem is a consequence of results of
Offen--Sayag (\cite{OffenSayagJFA08} Theorem 1)
and~Aizen\-bud--Offen--Sayag (\cite{AOSdisjoint} Theorem 1.1). 
		
\begin{theo}
\label{thm:disjoint}
Let $\pi$ be an irreducible representation
of $\GL_{2n}(F)$ distinguished by $\Sp_{2n}(F)$. 
Then
the space $\Hom_{{}_{k}H}(\pi, {}_{k}\Psi )$ is zero 
for all $k\in\{0,\dots,n-1\}$.
\end{theo}

\subsection{}
\label{Morrel}

In this paragraph,
the assumptions of \S\ref{GastonDeChanlay} on $F$ and $A$ still hold. 
For later use, we~pre\-pa\-re the following lemma.
Let $(n_1,\dots,n_r)$ be a composition of $n$ and,
for $i=1, \dots, r$, let $\d_i$ 
be a discrete series representation of $\GL_{n_i}(F)$. 
Suppose that
$\pi=\d_1\times\dots\times\d_r$ is a unitary generic irreducible
representation of $\GL_n(F)$.
Then
\begin{equation*}
\pi\nu^{1/2}\times\pi\nu^{-1/2}
\end{equation*}
has a unique irreducible quotient, denoted $\Sp_2(\pi)$.
(Recall that $\nu$ denotes the character 
``norma\-lized absolute value of the determinant''.) 
Note that the representation $\Sp_2(\pi)$~is~iso\-morphic to
$\Sp_2( \d_1)\times\cdots\times \Sp_2(\d_r)$. 
We refer the reader to \cite{Bjlu} \S4.1 (and \cite{Tadic})
when $F$ is non-Ar\-chi\-medean,
and when $F=\CC$ the argument there applies as well. 

\begin{lemm}\label{lem:factor}
In the above situation,
suppose that
$\pi$ is a unitary generic principal
series~re\-presentation of $\GL_n(F)$ 
and $\pi\nu^{1/2}\times\pi\nu^{-1/2}$ is distinguished by $H$. 
Any linear form~in 
\begin{equation*}
\Hom_{H}(\pi\nu^{1/2}\times\pi\nu^{-1/2} , \CC)
\end{equation*}
factors through the quotient map from 
$\pi\nu^{1/2}\times\pi\nu^{-1/2}$ to $\Sp_2( \pi)$. 
\end{lemm}

\begin{proof} 
It suffices to prove that the kernel $\kappa$ of this map
is not distinguished by~$H$.
By~assump\-tion on $\pi$,
we have $n_1=\dots=n_r=1$,
that is, the representations $\d_i$ are characters of $\F^\times$,
and $\pi=\d_1\times\cdots\times\d_n$
is unitary and generic. 
Up to semi-simplification,
we have
\begin{eqnarray*}
\pi\nu^{1/2}\times\pi\nu^{-1/2} &=& 
                                    \left(\d_1\nu^{1/2}\times\d_1\nu^{-1/2}\right)
                                    \tdt
                                    \left(\d_n\nu^{1/2}\times\d_n\nu^{-1/2}\right),\\
\d_i\nu^{1/2}\times\d_i\nu^{-1/2} & =& \St_2(\d_i)+\Sp_2( \d_i),
\end{eqnarray*} 
for all $i=1,\dots,r$. 
The semi-simplification of $\k$ is thus equal to 
\[
  \bigoplus_{I\subsetneq\{1, \ldots, n\}} \SS^I_1\tdt \SS^I_n,
  \quad \SS^I_i =
  \begin{cases} \Sp_2( \d_i) & \text{if $i\in I$}, \\ \St_2(\d_i) & \text{if
      $i\not\in I$}, \end{cases}
\]
where the sum ranges over the proper subsets $I$ of $\{1, \ldots, n\}$. 
By \cite{Offen-Sayag2} Theorem 3.7 when~$F=\CC$ and
\cite{GOSS} Theorem A when $F$ is non-Archimedean,
each~irreducible sum\-mand of $\kappa$ has a Klyachko model
which is not the symplectic model.
By local disjointness (Theorem \ref{thm:disjoint}),
we see that $\kappa$~is not distinguished by $H$. 
\end{proof}

\section{Proof of Theorem \ref{MAINTHM2}}
\label{sec:proof_main}

In all this section,
$F$ is a non-Archimedean locally compact field of characteristic $0$
and $D$ is~a qua\-ternion division $F$-algebra.
Recall that $\G=\G_n=\GL_n(D)$ and $\H=\H_n=\Sp_{n}(D)$.

We deduce Theorem \ref{MAINTHM2} from the key
Proposition \ref{redite}
which will be proven~in Sec\-tion \ref{sec:final}. 

\subsection{}
\label{sec:contribution-orbits}

Let $\pi$ be a discrete series representation of $G$.
According to \S\ref{DSclassiflocal},
this representation~is of the form $\St_m(\rho)$ for a unique
divisor $m$ of $n$ 
and a cuspidal representation $\rho$ of $\GL_{n/m}(D)$,~uni\-que\-ly
determined~up to isomorphism.
Let $r=r(\rho)$ denote the positive integer associated with~$\rho$.
It is equal to either $1$ or $2$.
We set $\ad=n/m$.

Let $P=MU$ denote the standard parabolic subgroup of $G$ corresponding to~the 
composition $(\ad,\dots,\ad)$ of $n$.
Given a complex number $s\in\CC$,
define the parabolically induced representation 
\begin{equation*}
I(s,\rho) = \rho\nu^{s\ss(m-1)/2} \times \rho\nu^{s\ss(m-3)/2} \times
\dots \times \rho\nu^{s\ss(1-m)/2}.
\end{equation*} 
In particular,
the representation $\pi$ is the unique irreducible quotient of $I(-1,\rho)$.
It is also~iso\-mor\-phic to the unique irreducible subrepresentation of
$I(1,\rho)$.

\begin{prop}
\label{PhilippedOrleans}
Let $s\in\CC$.
\begin{enumerate}
\item 
If $\rho$ is distinguished by $\H_{\ad}$,
then $I(s,\rho)$ is distinguished by $\H$. 
\item
Conversely, 
suppose that $I(s,\rho)$ is distinguished by $\H$ and the real part~of $s$
does not belong to the set
$\{(kr)^{-1} \mid k=1, 2, \ldots, m-1\}$.
Then
\begin{enumerate}
\item
The representation $\rho$ is distinguished by $\H_{\ad}$. 
\item 
The dimension of $\Hom_{H}(I(s, \rho), \CC)$ is equal to $1$.
\item
Any non-zero linear form in $\Hom_{\H}(I(s, \rho), \CC)$
does not vanish on the subspace made of func\-tions of $I(s, \rho)$
supported in the open $(P,H)$-double coset of $G$.
\end{enumerate}
\end{enumerate}
\end{prop}

\begin{proof}
First assume that $I(s,\rho)$ is distinguished by $H$. 
It follows from \cite{OJNT} Corollary 5.2~that, 
with~the notation and definitions of Section \ref{Kelh}, 
there is an $x\in X$ such that $P\cdot x$ is $M$-admissible
and the space 
\begin{equation}
\label{cellamare}
\Hom_{M_x}\left(
\rho\nu^{s\ss(m-1)/2} \otimes \rho\nu^{s\ss(m-3)/2} \otimes
\dots \otimes \rho\nu^{s\ss(1-m)/2},
\d_{P_x}^{\phantom{-1}}\d_P^{-1/2}\right) 
\end{equation}
is non-zero,
where $M_x$ and $P_x$ are the stabilizers of $x$ in $M$ and $P$,
and $\d_P$ and $\d_{P_x}$ are the~mo\-dulus cha\-rac\-ters of
$P$ and $P_x$, respectively.
We are going to prove that $P\cdot x$ is the open orbit of the symmetric
space $X$. 

Thanks to the description of $M_x$ given in \eqref{descriptionMx}
for a suitable choice of representative $x$,
and according to \cite{MOS} \S3.3.5 and \cite{ShVe} \S3.4, 
we have
\begin{equation*}
\d_{P_x}^{\phantom{-1}}\d_P^{-1/2}(g) =
\prod_{ i<\tau(i)} \nu(g_i)
\end{equation*}
for all $g=(g_1, g_2, \ldots, g_m)\in M_x\subseteq M = \G_{\ad}\tdt\G_{\ad}$,
where the product is over the integers $i\in\{1,\dots,m\}$
such that $ i<\tau(i)$. 
Associated with $P\cdot x$,
there is a $\tau\in\frS_m$ such that $\tau^2=1$.
Let us prove that $\tau$  is the identity.
It will follow from Example \ref{openorbit}
that $P\cdot x$ is the open orbit.

Suppose to the contrary that $i<\tau(i)$ for some $i$.
Then the space
\begin{equation*}
  \Hom_{G_{\ad}} \left(\rho\nu^{s r(m+1-2i)/2}\otimes
    \rho\nu^{s r(m+1-2\tau(i))/2}, \nu\otimes1\right)
\end{equation*} 
is non-zero,
where $G_{\ad}$ is identified with the subgroup
$\{(g,\s(g)) \mid g\in G_{\ad}\}$
of $G_{\ad}\times G_{\ad}$.
Looking~at the central character,
and thanks to the fact that $\s(z)=z^{-1}$ for all $z$ in the centre of
$G_{\ad}$,~we~ob\-tain
\begin{equation*}
\frac 1 2 \cdot \Re(s) \cdot r(m+1-2i) -
\left(\frac 1 2 \cdot \Re(s) \cdot r(m+1-2\tau(i)) \right) = 1 
\end{equation*}
or equivalently $\Re(s)\cdot r(\tau(i)-i)=1$,
where $\Re(s)$ is the real part of $s$.
Since $r$ is an in\-teger~and thanks to the assumption on $\Re(s)$, 
this contradicts the assumption $i<\tau(i)$.

The open $P$-orbit of $X$ is thus the unique $P$-orbit such that
\eqref{cellamare} holds.
Example \ref{openorbit} shows that 
we thus may choose $x=[w_m]_{\ad}$.
For such a choice of representative of the open $P$-orbit of $X$,
the subgroup $M_x$ is equal to $\H_{\ad}\times\cdots\times \H_{\ad}$.
It thus follows from \cite{OJNT} Section 4 that
\begin{eqnarray*}
\Hom_{\H}\left((I(s,\rho),\CC\right) & \simeq &
                                                  \Hom_{\H}\left(\Ind_P^{P\n\H}(s,\rho),\CC\right)\\ 
                                       & \simeq &
                                                  \Hom_{M_x}\left(\rho\nu^{s\ss(m-1)/2}
                                                  \odo
                                                  \rho\nu^{s\ss(1-m)/2},\CC\right) \\
  & \simeq & \Hom_{\H_{\ad}}\left(\nu^{s r(m-1)/2}\rho, \CC\right)\odo 
    \Hom_{\H_{\ad}} \left(\nu^{s r(1-m)/2}\rho,\CC\right) \\
& \simeq & \Hom_{\H_{\ad}}(\rho, \CC)\odo\Hom_{\H_{\ad}}(\rho, \CC)
           \phantom{\Hom_{M_x}\left( \rho\nu^{s(m-1)/2}\right) }
\end{eqnarray*} 
where the first isomorphism is induced by the restriction map from $G$
to the open double coset $P\n\H$
(where $\n\in\G$ satisfies $\n\s(\n)^{-1}=x$)
and the last~one fol\-lows from the fact that the character $\nu$
is trivial on $\H_{\ad}$. 
Hence assertion (2) follows from the previous series of isomorphisms, 
together with the fact that the dimension of 
$\Hom_{\H_{\ad}}(\rho, \CC)$ is at most $1$ by Theorem \ref{thm:mult1}.

Now assume that $\rho$ is distinguished by $\H_{\ad}$,
and let $x=[w_m]_{\ad}$ be the representative of the open $P$-orbit
given in Example \ref{openorbit}, 
as above.
Since $\rho\nu^{s\ss(m-1)/2} \odo \rho\nu^{s\ss(1-m)/2}$
is distinguished by~$M_x$,
and since $\s(P)=P$ and $xPx^{-1}$ is the parabolic subgroup of $G$
opposite to $P$ with~res\-pect~to $M$, 
it follows from \cite{OJNT} Proposition 7.2 
that $I(s,\rho)$ is distinguished by $\H$. 
\end{proof}

\begin{coro}
\label{Anker}
Let $\pi=\St_m(\rho)$ as above.
If $\pi$ is $\H$-distinguished, then $r=2$.
\end{coro}

\begin{proof}
If $\pi$ is distinguished by $\H$, 
then so is $I(-1,\rho)$,
and it follows from Proposition \ref{PhilippedOrleans}(2.a) that $\rho$ is
distingui\-shed by $\H_{\ad}$.
Theorem \ref{Vermacarqcq} thus implies that
${}^{\rm JL}\rho$ is not cuspidal,
that is,
$r\neq1$.
Hence we~ob\-tain $r = 2$.
\end{proof}

\subsection{}
\label{sec:proof}

Assume now that the discrete series representation $\pi$ is distinguished by $\H$.
In particular,
the induced representation $I(-1,\rho)$ is distinguished by $\H$,
thus $\rho$ is distinguished by $\H_{\ad}$ thanks to Proposition
\ref{PhilippedOrleans}(2.a).
Our goal is to prove that $m=1$,
that is, $\pi$ is cuspidal.
{Suppose to the contrary that $m\geq 2$.}

We fix a non-zero linear form $\el\in\Hom_{\H_{\ad}}(\rho, \CC)$.
Since the character $\nu$ is trivial on $\H_{\ad}$,
we~can regard it as an $H_{\ad}$-invariant linear form on
$\rho\nu^s$ for any $s\in\CC$. 
Let $x$ and $\n$ correspond to the~open orbit,
as in Example \ref{openorbit}.
In particular,
$M_x$ is equal to $\H_{\ad}\tdt\H_{\ad}$.
Set
\begin{equation*}
\li= \el \otimes\cdots\otimes \el \in
\Hom_{M_x}\left(\rho \otimes 
\dots \otimes \rho , \CC\right).
\end{equation*}
Given a flat section $\varphi_s\in I(s, \rho)$ (see \S\ref{DuchesseduMaine}),
we consider the integral 
\begin{equation}\label{BD_integral}
J(s, \h_s, \li) = \int_{(\eta^{-1}P\eta\cap H)\bsl H} \li(\h_s(\eta h))\rd h.
\end{equation}
The following theorem follows from \cite{BD} Th\'eor\`eme 2.8, Th\'eor\`eme 2.16 
(see \cite{BD} Remarque 2.17 and \cite{Lag} Th\'eor\`eme 4).
{At this point, and more generally in this paragraph, the choice
  of~in\-va\-riant measures is not important, and  we postpone the discussion of this
  matter to \S\ref{sec Haar} below.} 

\begin{theo}
\label{BlancDelorme}
\begin{enumerate}
\item
There is an $x_0\in\RR$ such that,
for any flat section $\h_s\in I(s,\rho)$,
the~in\-te\-gral \eqref{BD_integral} converges 
when $\Re(s)>x_0$. 
\item
There exists a non-zero Laurent polynomial $P(X)\in\CC[X,X^{-1}]$ such that,
for any flat~sec\-tion $\h_s\in I(s,\rho)$,~the function
\begin{equation*}
P(q^{-s}) \cdot J(s, \varphi_s, \li) 
\end{equation*}
extends to a function in $\CC[q^{-s},q^{s}]$.
\end{enumerate} 
\end{theo}

This defines a meromorphic family of $H$-invariant linear forms
$J(s,\cdot,\li)$ on $I(s,\rho)$,
called \emph{open intertwining periods}.

\begin{prop}
\label{Felzl}
The open intertwining period $J(s, \cdot ,\li)$ on $I(s,\rho)$ 
is holomorphic~and~non-zero at $s=1$ and $s=-1$.
\end{prop}

\begin{proof}
Note that $\ss =2$ by Corollary \ref{Anker}.
Let $e\in\{-1,1\}$.
As $\rho$ is distinguished by $\H_{\ad}$,~Pro\-position
\ref{PhilippedOrleans}(1) says that $\Hom_{H}(I(e, \rho), \CC)$
has dimension $1$ and Proposition \ref{PhilippedOrleans}(2.c)
says~that any non-zero linear form in this $1$-dimensional space 
has a non-zero restriction to the subspace~of functions in $I(e,\rho)$
supported in the open double coset $P\eta H$. 
The assertion now follows from~the same argument as
\cite{MatJFA} Proposition 10.4. 
\end{proof}

Let $w=w_{m}$ be the longest element in $W(M) \simeq \frS_m$.
Let $M(s, w ) : I(s, \rho)\to I(-s, \rho)$ denote the standard
intertwining operator given by the convergent integral
\[
  M(s, w )\varphi_s(g) = \int_{N} \varphi_s(w ug)\,\rd u
\]
for a flat section $\varphi_s\in I(s, \rho)$ when $\Re(s)$ is sufficiently 
large. 
It has a meromorphic continuation to the whole~com\-plex plane.
Note that $M(s, w)$ is holomorphic and non-zero at $s=1$, 
{as follows for example from \cite{MWres} I.1(4) together with the 
aforementioned results on the location of poles~of lo\-cal Rankin--Selberg 
$L$-factors of pairs of cuspidal representations.} 

We thus get two meromorphic families $J(s,\cdot,\li)$ and
$J(-s, M(s, w )\ \cdot,\li)$ of $H$-invariant linear forms on $I(s,\rho)$.
It follows from Theorem \ref{thm:mult1} together with generic irreducibility
of $I(s,\rho)$~that the space $\Hom_{H}(I(s,\rho),\CC)$ has dimension at
most $1$ for $s\in\CC$ generic. 
There is thus a meromor\-phic function $\alpha(s,\rho)$ such that 
\begin{equation}
\label{funct-eq}
J(-s, M(s, w )\h_s, \li) = \a(s, \rho)J(s, \h_s, \li)
\end{equation}
for any flat section $\h_s\in I(s, \rho)$. 
By Theorem \ref{BlancDelorme}(2),
the function $\a(s, \rho)$ is in $\CC(q^{-s})$.~We
will prove in Section \ref{sec:final} the following property of $\a(s, \rho)$.

\begin{prop}[Proposition \ref{cor:gamma}] 
\label{redite}
The meromorphic function $\a(s, \rho)$ is holomorphic and non-zero at $s=1$. 
\end{prop}

Let us show how
Theorem \ref{MAINTHM2} immediately follows from
Proposition \ref{redite}.

\subsection{}
\label{leber}

In this paragraph,
we prove Theorem \ref{MAINTHM2} assuming Proposition \ref{redite}.

Write $\pi$ as $\St_m(\rho)$,
where $m, \ad$ are positi\-ve integers such that $n=m\ad$
and $\rho$ is a cuspidal~re\-presentation of $G_{\ad}$. 
Since $\pi$ is distinguished~by $H$,
the space $\Hom_{H}(I(-1, \rho), \CC)$ has dimension~$1$.
As $r=r(\rho)=2$ by Corollary \ref{Anker}, 
it~fol\-lows from Proposition \ref{PhilippedOrleans} that
$\rho$ is distinguished by~$H_{\ad}$,
thus $\Hom_{H}(I(1,\rho), \CC)$ has dimension $1$.

Recall that we assume $m\geq2$.
Let $M^\ast(-1)$ be a non-zero element in
$\Hom_{G}(I(-1, \rho), I(1, \rho))$. 
Note that such a morphism is unique up to scalar,
its image is isomorphic to $\pi$
and its kernel is the image of $M(1, w)$. 
In particular,
$M^\ast(-1)M(1, w)$ vanishes on $I(1,\rho)$.

We fix a non-zero $H_{\ad}$-invariant linear form $\el$ on $\rho$, 
and associate to it the meromorphic~fami\-ly of open intertwining
periods $J(s, \cdot, \li)$.
As $J(-1, \cdot, \li)$ is a non-zero element of the $1$-di\-men\-sional space 
$\Hom_{H}(I(-1, \rho), \CC)$ by Proposition \ref{Felzl},
there is a non-zero ${\it\La}_{\pi}\in\Hom_{H}(\pi, \CC)$ such that
\[
  J(-1, \varphi, \li)={\it\La}_\pi(M^\ast(-1)\varphi)
\] 
for all $\varphi\in I(-1, \rho)$.
By \eqref{funct-eq} applied at $s=1$, we deduce that
\[
  \alpha(1, \rho) J(1, \varphi, \li) = J(-1, M(1, w)\varphi, \li)
  = {\it\La}_\pi(M^\ast(-1)M(1, w)\varphi)
\] 
for all $\h\in I(1, \rho)$.
Since $M^\ast(-1)M(1, w)\varphi=0$ for all $\varphi\in I(1, \rho)$ and
$\alpha(1, \rho)\neq0$ from Propo\-si\-tion~\ref{redite},
it follows that $J(1, \varphi, \li)=0$ for all $\varphi\in I(1, \rho)$.
This contradicts Proposition \ref{Felzl}.

\section{Global theory and computation at split places}
\label{sec:global}

We study the local intertwining period and the meromorphic function
$\a(s,\rho)$ via a global~ar\-gu\-ment. 

\subsection{}

Let $k$ be a totally imaginary number field and let $\AA=\AA_k$
denote its ring of adèles.~Let~$\DD$
be a non-split~quater\-nion $k$-algebra.
Given any integer $n\>1$,
the groups $\GL_n(\DD)$ and $\Sp_n(\DD)$~are
the groups of~$k$-ra\-tio\-nal points of reductive algebraic
$k$-groups which we denote by $\GG_n$ and $\HH_n$.

Given an Archimedean place $v$ of $k$,
any choice of isomorphism $k_v \simeq \CC$ of topological
fields~in\-du\-ces a group isomor\-phism $\GG_n(k_v) \simeq \GL_{2n}(\CC)$.
Note that there are only two such isomorphisms from $k_v$ to $\CC$,
which are~conju\-ga\-te to each other. 

Given a non-Archimedean place $v$ where $\DD$ splits, 
any choice of isomorphism $\DD_v \simeq \Mat_2(k_v)$~in\-duces
a group isomor\-phism $\GG_n(k_v) \simeq \GL_{2n}(k_v)$.
By the Skolem--Noether theorem,~any~two~iso\-mor\-phisms from
$\DD_v$ to $\Mat_2(k_v)$
are $\GL_2(k_v)$-conjuga\-te to each other.
It follows that any two~iso\-morphisms~$\GG_n(k_v) \simeq \GL_{2n}(k_v)$
obtained as above are $\GL_{2n}(k_v)$-conjuga\-te to each other.

Given any irreducible automorphic representation $\Pi$ of $\GG_n(\AA)$, 
there is a decomposition of $\Pi$ into local components $\Pi_v$ for each
place $v$ of $k$.
If $v$ is non-Archimedean,
$\Pi_v$ is an irreducible~re\-presentation of $\GG_n(k_v)$,
well defined up to isomorphism.
If, in addition,
the $k$-algebra $\DD$ splits~at $v$,
then $\Pi_v$ defines an irreducible
re\-presentation of $\GL_{2n}(k_v)$,
well defined up to isomorphism.~If $v$ is~Ar\-chi\-medean,
$\Pi_v$ defines an irreducible Harish-Chandra module of
$\GL_{2n}(\CC)$,
well defined~up to conjugacy. 

\subsection{}\label{sec Haar}

We did not fix choices of Haar measures before since it did not matter much. 
In what~fol\-lows, it will be helpful to do so,
in order to normalize spherical vectors and invariant linear forms,
especially at unramified places of automorphic representations.

Fix an integer $n\>1$
and write $\GG=\GG_n$ and $\HH=\HH_n$ for simplicity. 
For any place $v$ of~$k$,~we will write $G_v=\GG(k_v)$ and $H_v=\HH(k_v)$,
and similarly for any group defined over $k$.

At any place $v$ where $\DD_v$ is split,
we fix once and for all an isomorphism $\DD_v \simeq \Mat_2(k_v)$
of~$k_v$-al\-ge\-bras.
It induces a group isomorphism $G_v\simeq \GL_{2n}(k_v)$.
We will identify these groups
through this isomorphism,
without any further discussion,
whenever convenient.

Fix integers $m,\ad\>1$ such that $m\ad=n$
and let $\PP=\MM\NN$ be the standard parabolic~sub\-group
of $\GG$ associated with the composition $(\ad, \dots,\ad)$.
We will write $P_v=\PP(k_v)$, etc. 

As in \S\ref{DuchesseduMaine},
we set $K_v=\GG(\Oo_v)$ when $v$ is finite,
whereas 
we set $K_v=\mathrm{U}_{2n}(\CC/\RR)$ when it is Archimedean. 
This fixes a {good}
maximal compact subgroup $K$ of $\GG(\AA)$,
defined as the product of the $K_v$ for all $v$.

Given any closed subgroup $X_v$ of $G_v$, 
we define its Haar measure to give volume $1$ to the~in\-tersection $X_v\cap\K_v$.
This also normalizes the right invariant
measures on $(X_v\cap K_v)\backslash X_v$,
and it also normalizes similar global right invariant measures.

\subsection{}

Let~$\mathfrak{z}_{\MM}$~be the centre of the universal enveloping
algebra of the Lie algebra of $\MM(k\otimes_{\QQ}\RR)$.
A complex~function
on $\MM(k)\bsl \MM(\AA)$ is~called an automorphic form
for $\MM(\AA)$ if it is smooth,~of
mo\-derate growth,
right $K\cap\MM(\AA)$-finite~and~$\mathfrak{z}_{\MM}$-finite.
{We refer for example to \cite[Section 2.7]{BPCZ} for~a detailed 
  discussion of automorphic forms in the smooth setting.} 

An automorphic form $f : \MM(k)\bsl \MM(\AA)\to\CC$ is called a
cusp form if,
for any proper parabolic subgroup ${\rm Q}$ of $\MM$
with unipotent radical ${\rm V}$, 
we have
\[
  \int_{{\rm V}(k)\bsl {\rm V}(\AA)} f(vm)\,\rd v = 0, \quad m\in \MM(\AA).
\quad 
\]
Let $\CA_\PP(\GG)$ be the space of right $K$-finite functions
$\h : \NN(\AA)\MM(k)\bsl \GG(\AA)\to\CC$ such that,
for each $\kappa\in K$,
the function $g\mapsto \h(g\kappa)$ on $\MM(\AA)$ is an automorphic
form. 

Given a cuspidal automorphic representation $\Pi$ of $\GG_{\ad}(\AA)$, 
let $\CA_{\PP}^{\Pi}(\GG)$ denote the subspace~of all $\h\in\CA_{\PP}(\GG)$
such that for any $\kappa\in K$,
the function $m\mapsto\h(m\kappa)$ is in the space of
the cus\-pi\-dal automorphic representation
$\Pi\otimes\Pi\otimes\cdots\otimes\Pi$ of $\MM(\AA)$. 

We have the standard intertwining operator
$M(s,w):\CA_\PP(\GG)\to\CA_\PP(\GG)$,
given by the absolu\-tely convergent integral
\[
  M(s, w)\h(g) = e^{\langle s\rho_P, H_P(g)\rangle} \int_{\NN(\AA)}
  \h(w^{-1}ug) e^{\langle s\rho_P, H_P(w^{-1}ug)\rangle}\,\rd u
\]
where
$\rho_P=((m-1)/2, (m-3)/2, \cdots, (1-m)/2)\in\aa_P^\ast$ and $s$ is a
complex number with sufficient\-ly~large real part. 
It has a meromorphic continuation to $\CC$.

\subsection{}

Let $x=[w_m]_{\ad}$ be the representative of the open $\PP(k)$-orbit of
$\GG(k)/\HH(k)$ given by \eqref{duchessedeberry}
and $\eta\in\GG(k)$ be some representative of the open
$(\PP(k),\HH(k))$-double coset of $\GG(k)$ such that 
$x$ is equal to $\eta\s(\eta)^{-1}$
(see Exam\-ple \ref{openorbit}). 
Set $\MM_x = \MM \cap \eta \HH\eta^{-1} = \HH_{\ad}\tdt\HH_{\ad}$.

We define the global open intertwining period by
\begin{equation}\label{global_intertwining_period}
  J(s, \h) = \int_{(\eta^{-1} \PP(\AA)\eta \cap\HH(\AA)) \bsl \HH(\AA)} \left(
    \int_{\MM_x(\CF)\bsl \MM_x(\AA)}
    \h(m\eta h)\,\rd m \right) e^{\langle s\rho_P, H_P(\eta
    h)\rangle}\,\rd h
\end{equation}
for $s\in\CC$ and $\h\in\CA_\PP^{\Pi}(\GG)$.
Before adressing its convergence and meromorphic continuation,~we
  already observe that the inner integral is convergent and factorizable. 
  Namely, from \cite{Ash-Ginzburg-Rallis} Proposition 1, 
the non-zero linear form on $\Pi \otimes \cdots\otimes \Pi$ given by the 
period integral 
\[
p_{M_x}:\phi \mapsto \int_{\MM_x(\CF)\bsl \MM_x(\AA)} \phi(m)\,\rd m
\]
converges absolutely and provides,
for each place $v$,
an $\MM_x(k_v)$-invariant linear form $\li_v$ on~the local component
$\Pi_{v} \otimes \cdots\otimes\Pi_{v} $. Thanks to Theorem \ref{thm:mult1},
there is for each place $v$ a non-zero linear form
$\li_v\in \Hom_{\MM_{x,v}}(\Pi_v,\CC)$
such that if $\phi$ decomposes as $\ph=\bigotimes\limits_v\ph_v$,
then
\begin{equation}\label{JO}
p_{M_x}(\phi)=\prod_v \mu_v(\phi_v).
\end{equation}
Note that for this to make sense, we have fixed a choice of
$\MM_x(\Oo_v)$-spherical vectors $\phi_v$ at all~fi\-nite places $v$ of
$k$ such that $\DD_v$ splits and $\Pi_v$ is unramified,
and normalized the linear forms~$\li_v$ at these places such that 
$\li_v(\phi_v)=1$.
We then naturally normalize the $\GL_{2n}(\Oo_v)$-spherical function $\h_v\in 
\Pi_v\tdt \Pi_v$ by {requiring that its value at the identity element
$1_{2n,v}$ of $\GL_{2n}(k_v)$ is $\phi_v$.}

\begin{prop}
\label{thm:globalFE}
Let $\Pi$ be a cuspidal automorphic representation of
$\GG_{\ad}(\AA)$, and let $\h\in \CA_\PP^{\Pi}(\GG)$.
\begin{enumerate}
\item
The integral \eqref{global_intertwining_period} is absolutely
convergent when $\Re(s)$ is sufficiently large,
and~has~a~me\-ro\-mor\-phic continuation to $\CC$.
\item
We have the global functional equation
\begin{equation}
\label{GFE}
J(s, \h) = J(-s, M(s, w)\h). 
\end{equation} 
\end{enumerate}
\end{prop}

\begin{proof}
First we adress convergence and meromorphic continuation.
We may assume that~$\h$~is decomposable into $\prod_v \h_v$,
and at almost all 
places $v$ of $k$, the function $\h_v$ is the $\GL_{2n}(\Oo_v)$-sphe\-rical 
function normalized as above. 
According to \eqref{JO},~the inner integral in \eqref{global_intertwining_period} 
can~be~fac\-to\-ri\-zed as the well-defined product
\begin{equation*}
\prod_v \mu_v(\varphi_{v,s}(\eta h_v)).
\end{equation*}
Now we observe that for~a place $v$ at which $\DD$ splits,
the local factor $\mu_v(\varphi_{v,s}(\eta h_v))$ identifies with that~for
intertwining periods on~the representation of $\GL_{2n}(\AA)$ induced from
${}^{\rm JL}\Pi\otimes\cdots\otimes{}^{\rm JL}\Pi$ with~res\-pect~to
$\Sp_{2n}(\AA)$.
It thus follows from \cite{Yamana} Proposition 3.1,
that if we fix $S_0$ a
finite set of places of $k$, outside of which $\DD$ is split, then there 
is a $r_0\in \RR$ independent of $\h$ such that, for $\Re(s)>r_0$, 
the quantity
\begin{equation*}
\prod_{v\notin S_0} \mu_v(\varphi_{v,s}(\eta h_v))
\end{equation*}
is integrable on the restricted product of the
$(\eta^{-1} \PP(k_v)\eta \cap\HH(k_v)) \bsl \HH(k_v)$
for $v\notin S_0$
with respect to $(\eta^{-1} \PP(k_v)\eta \cap \HH(k_v) \cap K_v) \bsl \HH(k_v)
\cap K_v$.  
But,~up~to taking $r_0$ larger, the finite product
\begin{equation*}
\prod_{v\in S_0} \mu_v(\varphi_{v,s}(\eta h_v))
\end{equation*}
is integrable on the product of the
$(\eta^{-1} \PP(k_v)\eta \cap\HH(k_v)) \bsl \HH(k_v)$
for $v\in S_0$
thanks to \cite{BD}
and~\cite{Brylinski-Delorme-H-inv-form-MeroExtension-Invention},
and the convergence statement is proved.
The meromorphy then follows from \cite{BD},
\cite{Brylinski-Delorme-H-inv-form-MeroExtension-Invention}
and \cite{Yamana} Theorem 3.5.
This concludes~the proof of Assertion (1).

The functional equation \eqref{GFE} follows from the same argument as that for
intertwining periods of $\GL_{2n}(\AA)$ with respect to $\Sp_{2n}(\AA)$
proved by Offen in \cite{Offen-Sp1} Theorem 7.7,
and we observe that this also gives another proof of the meromorphy of the
global intertwining period.
\end{proof}

\begin{lemm}\label{lem:Pi_0}
Let $\Pi$ be a cuspidal automorphic representation of $\GG_{\ad}(\AA)$
having a non-zero $\HH_{\ad}(\AA)$-period. 
There exists a cuspidal automorphic representation $\Si$ of $\GL_{\ad}(\AA)$
such that ${}^{\rm JL}\Pi$~is equal to $\MW_2(\Si)$.
\end{lemm}

\begin{proof}
According to \S\ref{DSclassifglobal},
there are a positive integer $t$ dividing $2$ and a cuspidal representa\-tion
$\Si$ of $\GL_{2a/t}(\AA)$ such that ${}^{\rm JL}\Pi=\MW_t(\Si)$. 
From \cite{Verma} Theorem 1.3,
we know that~${}^{\rm JL}\Pi$~is distin\-guished by $\Sp_{2n}(\AA)$,
hence \cite{Ash-Ginzburg-Rallis} Theorem
implies that ${}^{\rm JL}\Pi$ is not cuspidal. 
Thus $t=2$.
\end{proof}

\subsection{}
\label{sec:intertwining-split}

From now on
and until the end of Section \ref{sec:global},
we assume that $\Pi$ is a cuspidal automor\-phic~representation
of $\GG_{\ad}(\AA)$ having~a~non-zero $\HH_{\ad}(\AA)$-period.

Fix a finite set $S$ of places~of $k$
{containing the set $S_\infty$ of Archimedean places and}
such that, for all $v\notin S$, 
one has
\begin{enumerate}
\item 
the $k_v$-algebra $\DD_v$ is split, 
\item the character $\psi_v$ has conductor $\Oo_v$ if $v$ is finite, 
\item
the local component $\Pi_v$ is unramified. 
\end{enumerate}
Fix a $\h\in \Pi\times \dots \times \Pi \subseteq \CA_\PP^{\Pi}(\G)$
and assume that it decomposes as a tensor product
\begin{equation}
\label{formautphi}
\h=\bigotimes\limits_v\h_v.
\end{equation}
We further assume that,
for any place $v\not\in S$,
the vector $\h_v$ is $\GL_{2n}(\Oo_v)$-spherical in $\Pi_{v}\tdt\Pi_{v}$
and normalized by
requiring that its value at $1_{2n,v}$ is $\phi_v$,~as~in the
discussion before Proposition \ref{thm:globalFE} above. 

By its very definition,
and up to potentially modifying one of the linear forms $\mu_v$ at
one~pla\-ce~by a non-zero scalar, the global intertwining~pe\-riod 
facto\-rises into the product of local intertwining periods as
{\[
  J(s, \h) = \prod_v J_v(s, \h_{v,s}, \li_v)
\]
where $\h_{v,s}$ is the flat section in
$\Pi_v\nu^{s(m-1)}\tdt\Pi_v\nu^{s(1-m)}$ such that $\h_{v,0}=\h_v$.
(As usual,
$\nu$ denotes the character ``normalized absolute value of the reduced norm''.)

We~set
\[
  J_S(s, \h) = \prod_{v\in S} J_v(s, \h_{v,s}, \li_v),
  \quad
  J^S(s, \h) 
  = \prod_{v\notin S} J_v(s, \h_{v,s}, \li_v).
\]
As in \eqref{funct-eq},
for each place $v$,
there exists a meromorphic function $\alpha_v(s)$ satisfying
\begin{equation}\label{eq:alpha_def}
J_v(-s, M_v(s, w)\h_{v,s}, \li_v) = \alpha_v(s)J_v(s, \h_{v,s},\li_v)
\end{equation}
where {$w=w_m$} is the longest element of $W(M_v)\simeq\frS_m$
and $M_v(s, w)$ is the
standard~inter\-twin\-ing~operator
from
$\Pi_v \nu^{s(m-1)}\tdt\Pi_v \nu^{s(1-m)}$
to
$\Pi_v \nu^{s(1-m)}\tdt\Pi_v \nu^{s(m-1)}$. 
From the global functional equation Proposition \ref{thm:globalFE}(2),
we obtain
\begin{eqnarray}
\label{eq:reduction}
\notag
\prod_{v\in S} \alpha_v(s) &=& \frac{J_S(-s, M(s, w)\h)}{J_S(s, \h)} \\
& = & \frac{J(-s, M(s, w)\h)J^S(s, \h)}{J(s, \h)J^S(-s, M(s,
    w)\h)} \\
  \notag
& = & \frac{J^S(s, \h)}{J^S(-s, M(s, w)\h)}. 
\end{eqnarray}
Hence, the functional equations of
$J_S(s,\h)$ and $J^S(s, \h)$
are related by inversion.

\subsection{}
\label{sec:open_period}

Recall that,
according to Lemma \ref{lem:Pi_0},
there is a cuspidal automorphic representation~$\Si$ of 
$\GL_{\ad}(\AA)$ such that ${}^{\rm JL}\Pi=\MW_2( \Si)$. 
Fix a place $v\notin S - S_\infty$, and write $\si=\Si_{v}$.
By our choice of $S$ made at the beginning of \S\ref{sec:intertwining-split},
the representation $\si$~is a~generic irreducible principal series
representation of $\GL_{\ad}(k_v)$,
which is moreover unramified when $v$ is finite.

Since $\DD_v$ is split,
the local components of $\Pi$ and ${}^{\rm JL}\Pi$ at $v$ are~iso\-mor\-phic. 
We thus have
\begin{equation}
\label{eq:other_places}
\Pi_v \simeq \MW_2(\Si)_v \simeq \Sp_2( \si)
\end{equation}
where the notation $\Sp_2(\s)$ has been defined earlier in \S\ref{Morrel}
and the last isomorphism~comes~from \cite{MWres} I.11.
We fix such an isomorphism, and identify $\Pi_v$ and $\Sp_2( \si)$ with no 
further notice, when convenient. 

Let $\Om_{\ad}$ denote the diagonal matrix of $\GL_{\ad}(k_v)$
with diagonal entries $(1,-1,1,\dots,(-1)^{\ad-1})$.~It 
has already been defined by \eqref{defOm} when $\ad$ is even,
and we have
\begin{equation*}
\Om_{2\ad} =
\begin{pmatrix} \Om_{\ad} & \\ & (-1)^{t}\Om_{\ad} \end{pmatrix}.
\end{equation*}
Given any representation $\pi$ of $\GL_{\ad}(k_v)$,
let $\pi^\diamond$~be the representation
$g \mapsto \pi(g^\diamond)$ of $\GL_{\ad}(k_v)$, 
where $g^\diamond = \Om_{\ad} \cdot {}^{\tt}g^{-1} \cdot \Om_{\ad}$
for any $g\in\GL_{\ad}(k_v)$. 
If~$\pi$~is~irre\-ducible,
then $\pi^\diamond$ is isomorphic to $\pi^\vee$
(thanks to \cite{BZ1} Theorem 7.3 if $v$ is non-Archimedean~and
\cite{AG} Theo\-rem 8.2.1 if $v$ is Archimedean).~We observe that the
intersection of $\Sp_{2\ad}(k_v)$ with the standard Levi subgroup
$\GL_{\ad}(k_v) \times \GL_{\ad}(k_v)$ of $\GL_{2\ad}(k_v)$ is equal to
the group $C=\{\diag(g,g^\diamond)\ |\ g\in \GL_{\ad}(k_v)\}$.

The representation $\si\nu^{1/2}\times \si\nu^{-1/2}$
affords a 
closed intertwining period, given by a compact~in\-tegration which~we now 
describe. 
We first need to describe the inducing linear form. 
For~this,~we
fix an isomorphism between $\si$
and the~para\-bolically~in\-du\-ced representa\-tion
$\Ind_{B}^{\GL_{\ad}(k_v)}(\chi)$
for~so\-me character $\chi$ of $(k_v^{\times})^t$, 
where $B$ is the upper triangular Borel subgroup of $\GL_{\ad}(k_v)$.
We~iden\-ti\-fy these two representations with no 
further notice, when convenient. 

By irreducibility of $\si$,
there exists a unique up to scalar isomorphism
\begin{equation}
\label{Gervaise}
M : \Ind_{B}^{\GL_{\ad}(k_v)}(\chi) \simeq \Ind_{B}^{\GL_{\ad}(k_v)}(\chi^{w_t})
\end{equation}
where we recall that $w_{\ad}$ is the longest element of $\frS_{\ad}$.
We then set, for $f_1$ and $f_2$ in $\si$: 
\[
  \ga(f_1\otimes f_2) =
  \int_{B\backslash \GL_{\ad}(k_v)}f_1(g)Mf_2(g^\diamond)dg.
\]
The map $\ga$ is a non-zero $C$-invariant linear form 
on the space of $\si\otimes \si$. 
Note that,
for any com\-plex~numbers $a,b\in \CC$,
the representations 
$\si\otimes\si$ and $\si\nu^{a}\otimes\si\nu^{b}$ 
have the same~under\-lying~space. 
Then,
for $f \in \si\nu^{1/2}\times \si\nu^{-1/2}$,
one can consider the well-defined~in\-te\-gral
\[
  \at (f)=\int_{(R\cap \HH_{\ad}(k_v))\backslash \HH_{\ad}(k_v)}
  \ga(f(h))dh
\]
where $R$ is the standard~parabo\-lic~subgroup of
$\GG_{\ad}(k_v)\simeq\GL_{2\ad}(k_v)$ 
associated with~the~composi\-tion $(\ad,\ad)$.
The map $\at$ is a non-zero $\HH_{\ad}(k_v)$-invariant linear form 
on $\si\nu^{1/2}\times\si\nu^{-1/2}$,
which is the closed~in\-ter\-twining period we were referring to above.
We now fix a quotient map 
\begin{equation*}
\pp : \si \nu^{1/2}\times\si \nu^{-1/2} \to \Sp_2( \si ) = \Pi_{v}.
\end{equation*} 
It fol\-lows from Lemma \ref{lem:factor} that
there is a linear form $\apt$ on $\Pi_v$ such that $\at = \apt \circ \pp$.
For $s\in\CC$, we set 
\begin{eqnarray*}
\ta_s & = & (\si\nu^{1/2}\times \si\nu^{-1/2})\nu^{s(m-1)} \odo
              (\si\nu^{1/2}\times \si\nu^{-1/2})\nu^{s(1-m)}, \\
\tau_s & = & \Pi_v^{\phantom{s}}\nu^{s(m-1)}\odo \Pi_v^{\phantom{s}}\nu^{s(1-m)}.
\end{eqnarray*}
The map $\pp$ induces a quotient map from $\ta_s$ to $\tau_s$ for all $s$, 
inducing the surjection
\[
  \qq_s:\Ind_{P_{v}}^{G_{v}}(\ta_s)\to \Ind_{P_{v}}^{G_{v}}(\tau_s) =
  \Pi_v^{\phantom{s}}\nu^{s(m-1)} \tdt \Pi_v^{\phantom{s}}\nu^{s(1-m)}
\]
where we recall that $P_v=M_vN_v$ is the standard parabolic subgroup of $G_v$
corresponding to the composi\-tion $(2\ad, \ldots,2\ad)$. 
Note that if 
$\widetilde{\varphi}_s$ is a flat section of $\Ind_{P_{v}}^{G_{v}}(\ta_s)$, 
then $\varphi_s=\qq_s(\widetilde{\varphi}_s)$ is a flat section of
$\Ind_{P_{v}}^{G_{v}}(\tau_s)$. 
Finally, write $\ta=\ta_0$ and $\tau = \tau_0$ and set
\begin{equation*}
\bt = \at\odo\at \in \Hom_{\MM_{x}(k_v)} (\ta,\CC), \quad
\be = \apt\odo\apt \in \Hom_{\MM_{x}(k_v)} (\tau,\CC). 
\end{equation*}
We observe that,
up to modifying the isomorphism $M$ of \eqref{Gervaise}
above by a non-zero scalar,~the~li\-near form $\be$ agrees with $\be_v$ in \eqref{JO}.
We now consider the open intertwining periods 
\[
  J_{P_{v}, \ta}(s, \widetilde{\varphi}_s, \bt)
  = \int_{(\eta^{-1}P_{v} \eta\cap H_{v}) \bsl H_{v}}
  \bt (\widetilde{\h}_s(\eta h))\,\rd h,
\]
and
\[
  J_{P_{v}, \tau }(s, \h_s ,\be )
  = \int_{(\eta^{-1}P_{v} \eta\cap H_{v}) \bsl H_{v}}
  \be (\h_s(\eta h))\,\rd h.
\]  
The convergence and meromorphic continuation of these integral are proved in
\cite{BD} and \cite{Brylinski-Delorme-H-inv-form-MeroExtension-Invention},~and 
we refer to \cite{MOY} for further generalization and properties.
By definition, it is immediate that 
\begin{equation}\label{eq:obv}
  J_{P_{v}, \ta}(s,\widetilde{\h}_s, \bt)
  = J_{P_{v}, \tau}(s, \h_s, \li)
\end{equation}
whenever the flat sections $\widetilde{\h}_s$ and $\h_s$ are related 
by $\h_s=\qq_s(\widetilde{\h}_s)$.

Recall that $w=w_m$ is the longest element of $W(M_v) \simeq \frS_m$,
that is,
we have $w(i)=m+1-i$ for $i=1, 2, \ldots, m$. 
Let $w(\ta_s)$ be the representation of $M_v$ defined as
\begin{eqnarray*}
  w(\ta_s) &=& (\si\nu^{1/2}\times \si\nu^{-1/2})\nu^{s(1-m)} \odo
              (\si\nu^{1/2}\times \si\nu^{-1/2})\nu^{s(m-1)} \\ 
&=& (\si\nu^{1/2}\times \si\nu^{-1/2})\nu^{-s(m-1)} \odo
              (\si\nu^{1/2}\times \si\nu^{-1/2})\nu^{-s(1-m)} \\
&=& \ta_{-s}.
\end{eqnarray*} 
Let $M_{P_{v}, \ta}(s, w)$ denote the standard intertwining operator 
\[
  \Ind_{P_{v}}^{G_{v}}(\ta_s) \to \Ind_{P_{v}}^{G_{v}}(w(\ta_s))
  = \Ind_{P_{v}}^{G_{v}}(\ta_{-s}).
\]
In this notation,
\eqref{eq:alpha_def} becomes
\begin{equation}
\label{eq:local FE not explicit}
  J_{P_{v}, \ta}(-s, M_{P_{v}, \ta}(s, w) \widetilde{\varphi}_s, \bt)
  = \a_v(s) J_{P_{v}, \ta}(s, \widetilde{\varphi}_s, \bt).
\end{equation}
Similarly,
let $\QQQ=\LL\VV$ be the standard parabolic subgroup of $\GL_{2n}$
associated with $(\ad, \dots, \ad)$.~We
observe that the $Q_v$-orbit of $x$ is $L_v$-admissible,
although it is not open anymore.
Let $\la$ denote the linear form
\[
  \lambda=\gamma\otimes \dots \otimes
  \gamma\in \Hom_{\LL_x(k_v)}( (\si \otimes \si )
  \otimes \dots \otimes (\si \otimes \si),\CC)
\]
where $\LL_x$ denotes the stabilizer of $x$ in $\LL$.
Note that $\LL_x(k_v)$ is equal to the subgroup
$C\tdt C$ (where $C$ occurs $m$ times).
For $s\in\CC$ and $i=1,\dots,m$,
write $\si_{i,s}=\s\nu^{s(m-2i+1)}$
and
\begin{equation*}
\sie_s=
  \si_{1,s}^{\phantom{1/2}}\nu^{1/2}\otimes\si_{1,s}^{\phantom{1/2}}\nu^{-1/2} 
  \odo
  \si_{m,s}^{\phantom{1/2}}\nu^{1/2}\otimes\si_{m,s}^{\phantom{1/2}}\nu^{-1/2}
\end{equation*}
and $\sie=\sie_0$.
Given a flat section
\[
  f_s\in
  \si_{1,s}^{\phantom{1/2}}\nu^{1/2}\times\si_{1,s}^{\phantom{1/2}}\nu^{-1/2} 
  \tdt
  \si_{m,s}^{\phantom{1/2}}\nu^{1/2}\times\si_{m,s}^{\phantom{1/2}}\nu^{-1/2} 
  = \Ind_{Q_v}^{G_{v}}(\sie_{s})
\]
we consider the local intertwining period,
in the sense of \cite{MOY}, given by the meromorphic~conti\-nuation of the 
integral 
\[
  J_{Q_v, \sie}(s, f_s, \la)
  = \int_{(\eta^{-1}Q_v \eta\cap H_{v}) \bsl H_{v}}
  \la (f_s(\eta h)) e^{\langle -\rho', H_{Q_v}(\eta h)\rangle} \,\rd h,
\]
where $\rho'=(1/2, -1/2, \ldots, 1/2, -1/2)\in\aa_{Q_v}^\ast$.
We observe that this is not an open intertwining~pe\-riod anymore,
however it is well defined,
since the modulus assumption of \cite{MOY},
namely $\delta_{\QQQ_x(k_v)}=\delta_{Q_v}^{1/2}$ on $\LL_x(k_v)$, is satisfied. 
Suppose that $f_s$ and $\widetilde{\varphi}_s$ correspond to each other under
the natural 
isomorphism between the in\-duced representations
$\Ind_{Q_v}^{G_{v}}(\sie_{s})$ and $\Ind_{P_{v}}^{G_{v}}(\ta_s)$. 
Then we have
\begin{equation}\label{eq:compattrans}
  J_{Q_v, \sie}(s, f_s, \l) = J_{P_{v}, \ta}(s, \widetilde{\h}_s, \bt)
\end{equation}
by \cite{LuMat} Proposition 3.7, or rather its proof, which requires only the 
unimodularity of the vertices involved. 

Let $\widetilde{w}$ denote the element of $W(L_v) \simeq \frS_{2m}$ given by 
\[
\widetilde{w}(k) =
        \begin{cases}
         2(m+1-i) & \text{if $k=2i$ is even},  \\
        2(m+1-i)-1 & \text{if $k=2i-1$ is odd}
      \end{cases}
\] 
for $k=1, 2, \dots, 2m$. 
Let $\widetilde{w}(\sie_s)$ be the representation of $L$ defined as 
\begin{eqnarray*}
  \widetilde{w}(\sie_s) &=& 
    \si_{m,s}^{\phantom{1/2}}\nu^{1/2}\otimes\si_{m,s}^{\phantom{1/2}}\nu^{-1/2}
                            \odo
                            \si_{1,s}^{\phantom{1/2}}\nu^{1/2}\otimes\si_{1,s}^{\phantom{1/2}}\nu^{-1/2} \\
   & = & \si_{1,-s}^{\phantom{1/2}}\nu^{1/2}\otimes\si_{1,-s}^{\phantom{1/2}}\nu^{-1/2} \otimes 
    \dots                                                              \otimes
         \si_{m,-s}^{\phantom{1/2}}\nu^{1/2}\otimes\si_{m,-s}^{\phantom{1/2}}\nu^{-1/2}
  \\ 
&=& \sie_{-s}.
\end{eqnarray*}
Let $M_{Q_v, \sie}(s, \widetilde{w})$ denote the standard intertwining operator 
    \[
    \Ind_{Q_v}^{G_{v}}(\sie_{s}) \to \Ind_{Q_v}^{G_{v}}(\widetilde{w}(\sie_s)) 
    = \Ind_{Q_v}^{G_{v}}(\sie_{-s}). 
    \]
Then the functional equation \eqref{eq:local FE not explicit} can be rewritten as
    \begin{equation} \label{eq:local FE not explicit 2} 
    J_{Q_v, \sie}(-s,  M_{Q_v, \sie}(s,  \widetilde{w}) f_s,  \lambda)
    = \alpha_v(s) J_{Q_v,  \sie}(s,  f_s,  \lambda)
    \end{equation}
for any flat section $f_s$ as above.     

\subsection{}
\label{sec:closed_period}

Let  $w'$ and $w''$ be the elements of $W(L)\simeq\frS_{2m}$ defined by 
    \[
    w'(i) = 
        \begin{cases}
        2m+1-i/2 & \text{if $i$ is even,} \\
        (i+1)/2 & \text{if $i$ is odd},  
        \end{cases}
    \quad
    w''(i) = 
        \begin{cases}
        m+i/2 & \text{if $i$ is even,} \\
        m+1-(i+1)/2 & \text{if $i$ is odd,}
        \end{cases}
    \]
for $i=1, 2, \ldots, 2m$, so that we have $w''\widetilde{w}=w'$. 
We define a linear form $\la'$ on the representation
    \[
    w'(\sie_s) = (\si_{1,s}^{\phantom{1/2}}\nu^{1/2}\otimes \dots \otimes \si_{m,s}^{\phantom{1/2}}\nu^{1/2}) 
    \otimes (\si_{m,s}^{\phantom{1/2}}\nu^{-1/2} \otimes \dots \otimes \si_{1,s}^{\phantom{1/2}}\nu^{-1/2}) .
    \] 
of $L$ by 
    \[
    \la'(x_1\otimes \dots x_m\otimes y_m\otimes \dots \otimes y_1)
    = \prod_{i=1}^m\la_i(x_i\otimes y_i)
    \] 
for $x_i\in\si_{i,  s}\nu^{1/2}$ and $y_i\in\si_{i,s}\nu^{-1/2}$.
Given a flat section 
    \[
    f'_s \in \Ind_{Q_v}^{G_{v}}(w'(\sie_s))
     = (\si_{1,s}^{\phantom{1/2}}\nu^{1/2}\times \dots \times \si_{m,s}^{\phantom{1/2}}\nu^{1/2}) 
    \times (\si_{m,s}^{\phantom{1/2}}\nu^{-1/2} \times \dots \times \si_{1,s}^{\phantom{1/2}}\nu^{-1/2})
  \]
we have the closed intertwining period
  \[
    J'_{Q_v, \sie}(s,  f'_s,  \la')  
    = \int_{(Q_v \cap H_{v}) \backslash H_{v}}\la'(f'_s(h))dh.
    \]  
Let $M_{Q_v,\sie}(s, w')$ denote the standard intertwining operator from 
$\Ind_{Q_v}^{G_{v}}(\sie_{s})$ to $\Ind_{Q_v}^{G_{v}}(w'(\sie_s))$.~As
before,
suppose that the flat sections $f_s$ and $\h_s$ 
correspond to each other under the natural~iso\-morphism
$\Ind_{Q_v}^{G_{v}}(\sie_{s}) \simeq \Ind_{P_{v}}^{G_{v}}(\ta_s)$. 
Then we have 
    \begin{align}\label{eq:open vs closed LIP} 
    \begin{split}
    J_{P_{v},  \ta}(s,  \widetilde{\varphi}_s,  \bt) & = J_{Q_v,  \sie}(s,  f_s, \lambda)  \\
    & = J'_{Q_v, \sie}(s, M_{Q_v, \sie}(s, w') f_s, \lambda').
    \end{split}
    \end{align} 
The first equality is \eqref{eq:compattrans} and
the second equality is verified by repeated use of \cite{MOY}~Proposi\-tion~5.1 
along an appropriate reduced expression of $w'\in \frS_{2m}$. 
Consider flat sections $f_s^-$,~$\widetilde{\varphi}{}_s^-$~in 
$\Ind_{Q_v}^{G_{v}}(\sie_{-s})$ and 
$\Ind_{P_{v}}^{G_{v}}(\tau_{-s})$
respectively, 
corresponding to each other under the natural isomor\-phism
$\Ind_{Q_v}^{G_{v}}(\sie_{-s}) \simeq \Ind_{P_{v}}^{G_{v}}(\ta_{-s})$.
Let $M_{Q_v, \widetilde{w}(\sie)}(s, w'')$ be the standard
intertwining~ope\-ra\-tor
\[
  \Ind_{Q_v}^{G_{v}}(\sie_{-s}) = \Ind_{Q_v}^{G_{v}}(\widetilde{w}(\sie_s))
  \to \Ind_{Q_v}^{G_{v}}(w''\widetilde{w}(\sie_s))
  = \Ind_{Q_v}^{G_v}(w'(\sie_s)).
\] 
From Equation \eqref{eq:compattrans},
and applying \cite{MOY} Proposition 5.1 
repeatedly along an appropriate~re\-du\-ced expression of $w''\in \frS_{2m}$,
we obtain 
    \begin{align}\label{eq:open vs closed LIP 2} 
    \begin{split}
    J_{P_{v},  \ta}(-s,  \widetilde{\varphi}_s^-, \bt)
    & = J_{Q_v,  \sie}(-s,  f_{s}^-,  \lambda) \\
    & = J'_{Q_v, \sie}(s, M_{Q_v, \widetilde{w}(\sie)}(s, w'') f_s^-, \lambda').
     \end{split}
    \end{align} 
In the rest of this section, we write
\begin{equation*}
L_\si(s) = L(s, \si, \si^\vee), \quad
\e_\si(s) = \e(s, \si, \si^\vee,\psi_v), \quad
\g_\si(s) = \g(s, \si, \si^\vee,\psi_v).
\end{equation*} 
Following \cite{MWres} I.1,
let us introduce the normalized intertwining operators $N_{\sie}(s, w')$,
$N_{\widetilde{w}(\sie)}(s, w'')$ and $N_{\sie}(s,\widetilde{w})$ defined as
\begin{eqnarray*}
  N_{\sie}(s, w') &=& r_{\sie}(s, w')^{-1}M_{Q_v, \sie}(s, w'), \\
  N_{\widetilde{w}(\sie)}(s,  w'')  &  =&  r_{\widetilde{w}(\sie)}(s,  w'')^{-1}
                                         M_{Q_v, \widetilde{w}(\sie)}(s, w''), \\
  N_{\sie}(s, \widetilde{w}) &=& r_{\sie}(s, \widetilde{w})^{-1}M_{Q_v, \sie}(s, 
                                \widetilde{w}),
\end{eqnarray*} 
where
$r_{\sie}(s, w')$,
$r_{\widetilde{w}(\sie)}(s, w'')$ and
$r_{\sie}(s,\widetilde{w})$ are meromorphic functions given by
\begin{eqnarray*}
  r_{\sie}(s, w') & = & \prod_{1\leq i<j \leq m} \frac{L_\si(2(j-i)s-1)}
                       {\varepsilon_\si(2(j-i)s-1)\varepsilon_\si(2(j-i)s)L_\si(2(j-i)s+1)}, \\
  r_{\widetilde{w}(\sie)}(s, w'') & = & \prod_{1\leq i<j \leq m} \frac{L_\si(2(i-j)s-1)}
      {\varepsilon_\si(2(i-j)s-1)\varepsilon_\si(2(i-j)s)L_\si(2(j-i)s+1)},
\end{eqnarray*}
and
\begin{align*}
  r_{\sie}(s, \widetilde{w}) = \prod_{1\leq i<j \leq m}
                              \varepsilon_\si(2(j-i)s)^{-1} \varepsilon_\si(2(j-i)s&+1)^{-1} 
\varepsilon_\si(2(j-i)s-1)^{-1} \varepsilon_\si(2(j-i)s)^{-1} \\ 
    & \times \prod_{1\leq i<j \leq m} \frac{L_\si(2(j-i)s-1)L_\si(2(j-i)s)}
      {L_\si(2(j-i)s+2)L_\si(2(j-i)s+1)}.
\end{align*}    
From \cite{MWres} I.1 we have
\begin{equation*}
N_{\widetilde{w}(\sie)}(s, w'') \circ { N_{\sie} (s, \widetilde{w}) } = N_{\sie}(s, w').
\end{equation*}
This is equivalent to
\begin{equation}\label{eq:intertw_op_FE}
    M_{Q_v,  \widetilde{w}(\sie)}(s,  w'')\circ M_{Q_v, \sie}(s,  \widetilde{w}) 
    = \kappa_v(s) M_{Q_v,  \sie}(s,  w'),
    \end{equation}
where we set
$\kappa_v(s) = r_{\widetilde{w}(\sie)}(s, w'')\cdot
r_{\sie}(s, \widetilde{w}) \cdot r_{\sie}(s, w')^{-1}$,
that is 
\begin{align*}\label{eq:kappa}
\notag
  \kappa_v(s) 
     = \prod_{1\leq i < j\leq m}   
    \varepsilon_\si(2(j-i)s)^{-1} \varepsilon_\si(2(j-i)s&+1)^{-1}
    \varepsilon_\si(2(i-j)s-1)^{-1} \varepsilon_\si(2(i-j)s)^{-1} \\ 
    & \times \prod_{1\leq i < j\leq m}
    \frac{L_\si(2(j-i)s)L_\si(2(i-j)s-1)} {L_\si(2(j-i)s+2)L_\si(2(i-j)s+1)}. 
    \end{align*}
    From \eqref{eq:intertw_op_FE},
    \eqref{eq:open vs closed LIP 2} and \eqref{eq:open vs closed LIP}, we obtain
\begin{eqnarray*}
    J_{Q_v,  \sie}(-s,  M_{Q_v, \sie}(s,  \widetilde{w}) f_s,  \lambda)
    & = & J'_{Q_v, \sie}(s, M_{Q_v, \widetilde{w}(\sie)}(s, w'') 
    \circ M_{Q_v, \sie}(s, \widetilde{w}) f_s, \lambda') \\ 
    & = & \kappa_v(s) J'_{Q_v, \sie}(s, M_{Q_v, \sie}(s, w')f_s, \lambda') \\ 
    & = & \kappa_v(s) J_{Q_v, \sie}(s, f_s, \lambda). 
    \end{eqnarray*}
Comparing with the functional equation \eqref{eq:local FE not explicit 2},
we deduce that $\alpha_v(s)=\kappa_v(s)$. 
We state~this re\-sult as a proposition.

\begin{prop}
\label{split_computation} 
Let
\begin{equation*}
  M_{Q_v, \sie}(s, \widetilde{w}) : \Ind_{Q_v}^{G_{v}}(\sie_{s})
  \to \Ind_{Q_v}^{G_{v}}(\sie_{-s})
\end{equation*}
denote the standard intertwining operator.
For any flat section $f_s\in \Ind_{Q_v}^{G_{v}}(\sie_{s})$,
we have the~func\-tional equation of local intertwining period
\begin{equation*}
  J_{Q_v, \sie}(-s, M_{Q_v, \sie}(s, \widetilde{w}) f_s, \la)
  = \alpha_v(s) J_{Q_v, \sie}(s, f_s, \la)
\end{equation*}
with  
\begin{align}\label{eq:alpha}
  \notag
  \a_v(s) 
     = \prod_{1\leq i < j\leq m}   
    \varepsilon_\si(2(j-i)s)^{-1} \varepsilon_\si(2(j-i)s&+1)^{-1}
    \varepsilon_\si(2(i-j)s-1)^{-1} \varepsilon_\si(2(i-j)s)^{-1} \\ 
    & \times \prod_{1\leq i < j\leq m}
    \frac{L_\si(2(j-i)s)L_\si(2(i-j)s-1)} {L_\si(2(j-i)s+2)L_\si(2(i-j)s+1)}. 
\end{align}
\end{prop}

If $\si$ and $\psi_v$ are moreover unramified,
which happens if $v$ is finite, 
then $\e_\si(s)=1$ and~\eqref{eq:alpha} sim\-plifies to
\begin{equation*}
\label{numero34}
\a_v(s) = \prod_{1\< i < j\< m}\frac{\g_{\si}(2(j-i)s+2)} {\g_{\si}(2(j-i)s)}.
\end{equation*}
This observation also follows from the unramified computation of intertwining 
periods below,~to\-gether with the Gindikin--Karpelevich formula. 

\begin{prop}
\label{prop: UR comp}
Suppose that $v$ is finite,
that $\si$ is unramified, and that $\h\in \Pi_{v}\tdt \Pi_{v}$ is the
spherical vector such that 
$\mu(\h(1_{2n,v}))=1$. 
Then the following equality holds good:
\[
 J_{P_{v}, \mu}(s, \varphi_s, \be)=
  \prod_{1\< i < j\< m}\frac{L_\si(2(j-i)s-1)}{L_\si(2(j-i)s+1)}.
\]
\end{prop}

\begin{proof}
Putting Equations \eqref{eq:obv} and \eqref{eq:open vs closed LIP} together,
we obtain
\begin{equation*}
 J_{P_{v}, \tau}(s, \varphi_s, \be)=J'_{Q_v, \sie}(s, M_{Q_v,\sie}(s,w') f_s, \l'),
\end{equation*}
where $\varphi_s$ and $f_s$ are any flat sections related as 
in the above discussion.  
If $\h$ is $\GL_{2n}(\Oo_v)$-sphe\-ri\-cal~and normalized as in the statement, 
which is as in the discussion before Theorem~\ref{thm:globalFE}, ~then 
$\widetilde{\h}$ can be chosen to be $\GL_{2n}(\Oo_v)$-spherical,
and $\widetilde{\mu}(\widetilde{\h}(1_{2n,v}))=\mu(\h(1_{2n,v}))=1$. 
In turn this~implies that the function
$f$ is $\GL_{2n}(\Oo_v)$-spherical,
and that one has $\l(f(1_{2n,v}))=1$.
The computation for 
$J'_{Q_v, \sie}(s, M_{Q_v, \sie}(s, w') f_s, \lambda')$
now follows from the 
Gindikin--Karpelevich formula in the form recalled in \cite{MatJFA} Lemma 9.1.
\end{proof}

\section{Conclusion}
\label{sec:final}

We now deduce Proposition \ref{redite} from the computations in 
Section \ref{sec:global}.
Thanks to \S\ref{leber},
this will end the proof of Theorem \ref{MAINTHM2}.

Recall that
$F$ is~a~non-Archimedean locally compact field of characteristic $0$,
and $D$ is a~non-split qua\-ter\-nion $F$-algebra.
We also have a discrete series representation $\pi$ of
$G=\GL_n(D)$~dis\-tinguished by $H=\Sp_n(D)$.
It is of the form $\St_m(\rho)$ for some divisor $m$ of $n$ and some 
cuspidal representation $\rho$ of $\GL_{\ad}( D)$ distinguished by
$\Sp_{\ad}(D)$, with $\ad=n/m$.

{For any unramified character $\chi$ of $G$,
  the representation $\pi\chi$ is a discrete series representation
  of~$G$ dis\-tinguished by $H$,
which is cuspidal if and only if $\pi$ is cuspidal.
Without loss of generality,
we thus may (and will) assume that $\pi$,
or equivalently $\rho$, 
is unitary.}

Let us fix 
\begin{itemize}
\item
  a totally imaginary number field $\CF$ with a finite place $\pl$ 
  so that $\CF_{\pl}=F$ and $\pl$ is the only place of $\CF$ above $p$, 
\item
a quaternion division algebra $\DD$ over $\CF$ which is non-split at $\pl$.
\end{itemize}
By Lemma \ref{lem:globalize}, 
there exists
an irreducible cuspidal automorphic representation $\Pi$ of
$\GL_{\ad}(\DD\otimes_k\AA)$ such that $\Pi_{\pl}\simeq\rho$ and 
$\Pi$ is $\Sp_{\ad}(\DD\otimes_k{\AA})$-distinguished,
that is, there is an $f\in\Pi$ such that
\begin{equation*}
\int_{\Sp_{\ad}(\DD)\bsl\Sp_{\ad}(\DD\otimes_k \AA)} f(h)\,\rd h\neq0.
\end{equation*} 
By Lemma \ref{lem:Pi_0},
there is a cuspidal automorphic~re\-presentation $\Si$ of
$\GL_{\ad}(\AA)$ such that ${}^{\rm JL}\Pi$ is equal to $\MW_2(\Si)$. 
Let $S$ be the finite set of places of $\CF$ 
consisting of $\pl$, the Archimedean places of $\CF$
and~all~fi\-ni\-te places $v$ such that $\Si_{v}$ is ramified. 
Let $S_{\infty}$ be the set of Archimedean places~of $k$. 
The~next~pro\-po\-sition is a consequence of
Propositions \ref{global-RS}, \ref{split_computation}, \ref{prop: UR comp}
and \cite{MatJFA} Lemma 9.1. 

\begin{prop}
\label{prop:ratio}
Let $\h\in \Pi\tdt\Pi$ be the automorphic form fixed in 
\eqref{formautphi} where,
for any place $v\not\in S$,
the vector $\h_v$ is the unique $\GL_{2n}(\Oo_v)$-spherical
function in $\Pi_{v}\tdt\Pi_{v}$ given by Proposition \ref{prop: UR comp}.
We have
\[
  J^S(s, \h) = \prod_{1\leq i<j \leq m}
  \frac{L^S(2(j-i)s-1, \Si, \Si^\vee)} {L^S(2(j-i)s+1, \Si, \Si^\vee)}
\]
and hence
\[
  \frac{J^S(-s, M(s, w)\h)}{J^S(s, \h)}
  = \prod_{v\in S}\prod_{1\leq i<j\leq m}
  \frac{\g(2(j-i)s, \Si_{v}^{\phantom{\vee}}, \Si_{v}^{\vee},\psi_v^{\phantom{1}})}
  {\g(2(j-i)s+2, \Si_{v}^{\phantom{\vee}}, \Si_{v}^{\vee},\psi_v^{\phantom{1}})}.
\]
\end{prop}

{For $f$ and $g$ two functions of a complex variable $s$,
  write $f(s)\sim g(s)$ if there exists a $c\in \CC^\times$ such that
  $g(s)=c f(s)$.}
 
\begin{theo}\label{thm:alpha} 
With the above notation, we have
\[
  \a(s, \rho) \sim \prod_{1\leq i<j\leq m}
  \frac{\g(2(j-i)s+2, \Si_{\pl}^{\phantom{\vee}}, \Si_{\pl}^\vee,\psi_{\pl}^{\phantom{1}})}
  {\g(2(j-i)s, \Si_{\pl}^{\phantom{\vee}}, \Si_{\pl}^\vee,\psi_{\pl}^{\phantom{1}})}. 
\]
\end{theo}

\begin{proof}
Recall from \eqref{eq:reduction} that we have
\[
  \frac{J^S(s, \h)}{J^S(-s, M(s, w)\h)} = \prod_{s\in S} \alpha_v(s).
\]
By Proposition \ref{prop:ratio}, we obtain
\[
  \prod_{v\in S} \alpha_v(s) = \prod_{v\in S}\prod_{1\leq i<j\leq m}
  \frac{\g(2(j-i)s+2, \Si_{v}^{\phantom{\vee}}, \Si_{v}^{\vee}, \psi_v^{\phantom{1}})}
  {\g(2(j-i)s, \Si_{v}^{\phantom{\vee}}, \Si_{v}^{\vee}, \psi_v^{\phantom{1}})}.
\]
Since Archimedean root numbers are constant
(see \cite{Jrsarch} 16 Appendix), it follows from Proposi\-tion~\ref{split_computation}
for Archimedean places that 
we {can simplify the above identity to}
\[
  \prod_{v\in S\bsl S_{\infty}} \alpha_v(s)\sim \prod_{v\in S\bsl S_{\infty}}\prod_{1\leq i<j\leq m}
  \frac{\g(2(j-i)s+2, \Si_{v}^{\phantom{\vee}}, \Si_{v}^{\vee}, \psi_v^{\phantom{1}})}
  {\g(2(j-i)s, \Si_{v}^{\phantom{\vee}}, \Si_{v}^{\vee}, \psi_v^{\phantom{1}})}.
\]
Since $\pl$ is the only place above $p$ and
$\a(s, \rho)=\a_{\pl}(s)$, the assertion now follows from
\cite{MatJFA} Lemma 9.3.
\end{proof}

Now we obtain the next proposition as promised in \S\ref{sec:proof}.

\begin{prop}\label{cor:gamma}
The meromorphic function $\a(s,\rho)$ is holomorphic and non-zero at $s=1$. 
\end{prop}

\begin{proof}
Note that $\Si_{\pl}$ is cuspidal and $\JL(\rho)=\St_2(\Si_{\pl})$.
{From the properties of $L$ factors~re\-cal\-led in Section \ref{pole}
  and central character considerations, 
  $\g(s,\Si_{\pl},\Si_{\pl}^\vee,\psi_{\pl})$ and~its inverse can only vanish for
  $s$ of real part equal to $0$ or $1$. It then follows from Theorem
  \ref{thm:alpha}~that each quotient of $\g$-factors in the formula for
  $\a(s,\rho)$ is holomorphic and non-vanishing at $s=1$.} 
\end{proof}

\section{Appendix}
\label{app}

In this section, $A$ is as in Section \ref{sensdirect}.
Let $[\aa,\b]$ be a simple stratum in $\A$.
Let $\psi^A$ denote the character $x \mapsto \psi(\trd_{A/F}(x))$
of $A$, where $\trd_{A/F}$ is the reduced trace of $A$ over $F$.

We prove Lemma \ref{applem}
(which has been used in Section \ref{Sec2}),
whose proof was postponed to this last section
since it requires techniques which are not used anywhere else
in the paper.

\begin{lemm}
\label{applem}
Let $\t\in\Cc(\aa,\b)$ be a simple character.
Then
\begin{enumerate}
\item $[\aa^*,\b^*]$ is a simple stratum realizing $\b$, and
\item $\t^*$ is a simple character in $\Cc(\aa^*,\b^*)$. 
\end{enumerate}
\end{lemm}

\begin{proof}
If $[\aa,\b]$ is the null stratum, there is nothing to prove. 
We will thus assume that $[\aa,\b]$ has positive depth.

The map $\ii : x \mapsto x^*$ is an $F$-linear involution
of $A$ such that $\ii(xy)=\ii(y)\ii(x)$.
Restricting~to the commutative $F$-algebra $E=F[\b]$,
it is thus an embedding of $F$-algebras from $E$ to $A$.
This proves (1).
Note that,
if $B$ is the centralizer of $E$ in $A$,
then the centralizer of $E^*$ in $A$ is $B^*$. 

\begin{lemm}
\label{nrdstar}
One has $\Nrd_{B/E}(x^*) = \Nrd_{B^*/E^*}(x)^*$ for all $x\in B^*$.
\end{lemm}

\begin{proof}
The proof is similar to \cite{VSautodual} Lemme 5.15.
\end{proof}

We will prove (2) by induction on the integer $q=-k_0(\aa,\b)$.
(See \cite{VSrep1} \S2.1 for the definition of $k_0(\aa,\b)$.)
Define $r=\lfloor q/2\rfloor+1$.
First note that $k_0(\aa^*,\b^*) = k_0(\aa,\b)$ and 
$\t^*$ is normalized by $\EuScript{K}(\aa^*)\cap\B^{*\times}$
as $\t$~is~nor\-ma\-lized by $\EuScript{K}(\aa)\cap\B^\times$. 
{(Here,~$\EuScript{K}(\aa)$ denotes the normalizer in~$G$ of the order~$\aa$.)}
For any integer $i\>1$,
let us write $\U^i(\aa)=1+\p^i_\aa$.

Assume first that $\b$ is minimal over $F$ (see \cite{VSrep1} \S2.3.3).
In this case, we have
\begin{itemize}
\item $\H^1(\aa,\b) = \U^1(\bb)\U^r(\aa)$, 
\item the restriction of $\t$ to $\U^r(\aa)$ is the character
$\psi^A_\b : 1+x \mapsto \psi^A(\b x)$, 
\item the restriction of $\t$ to $\U^1(\bb)$ is equal to
$\xi \circ \Nrd_{B/E}$ for some character $\xi$ of $1+\p_E$.
\end{itemize}
The character $\t^*$ is defined on the group 
$\s(\H^1(\aa,\b)) = \U^1(\bb^*)\U^r(\aa^*) = \H^1(\aa^*,\b^*)$.
Its restriction to $\U^r(\aa^*)$ is the character
\begin{equation*}
1+y \mapsto \psi^A(\b y^*)= \psi^A(\b^* y) = \psi^A_{\b^*}(1+y)
\end{equation*}
since $\psi^{A}$ is invariant by $*$.
By Lemma \ref{nrdstar},
its restriction to $\U^1(\bb^*)$ is $\xi^* \circ \Nrd_{B^*/E^*}$ 
where~$\xi^*$ is the character $x \mapsto \xi(x^*)$ of $1+\p_{E^*}$.
It follows from \cite{VSrep1} Proposition 3.47 that $\t^*$ is a
simple cha\-racter~in $\Cc(\aa^*,\b^*)$.

Now assume that $\b$ is not minimal over $F$,
and that $\g$ is an approximation of $\b$ with respect to $\aa$
(see \cite{VSrep1} \S2.1).
We have
\begin{itemize}
\item $\H^1(\aa,\b) = \U^1(\bb) \H^r(\aa,\g)$, 
\item the restriction of $\t$ to $\H^r(\aa,\g)$ is equal to 
  $\psi^A_{\b-\g}\t'$ for some simple character $\t' \in \Cc(\aa,\g)$, 
\item the restriction of $\t$ to $\U^1(\bb)$ is equal to
  $\xi \circ \Nrd_{B/E}$ for some character $\xi$ of $1+\p_E$.
\end{itemize}
The character $\t^*$ is defined on the group 
\begin{equation*}
\s(\H^1(\aa,\b)) = \U^1(\bb^*) \s(\H^1(\aa,\g))
= \U^1(\bb^*) \H^1(\aa^*,\g^*)
= \H^1(\aa^*,\b^*)
\end{equation*}
since $\g^*$ is an approximation of $\b^*$ with respect to $\aa^*$.
By induction,
its restriction to $\H^1(\aa^*,\g^*)$ is the character $\psi^A_{\b^*-\g^*}\t'^*$
where $\t'^* \in \Cc(\aa^*,\g^*)$ is the transfer of $\t'$.
Its restriction to $\U^1(\bb^*)$ is the character
$\xi^* \circ \Nrd_{B^*/E^*}$. 
It follows from \cite{VSrep1} Proposition 3.47 that $\t^* \in \Cc(\aa^*,\b^*)$.
\end{proof}

\bibliography{symplectic}

\end{document}